\documentclass[12pt]{amsart}
\usepackage{geometry}                
\usepackage{graphicx}
\usepackage{tikz}
\usetikzlibrary{decorations}
\usepackage{amssymb}
\usepackage{epstopdf}
\usepackage{amsmath}
\usepackage{color}
\usepackage{hyperref}
\usepackage{pdfsync}
\usepackage{bbm}
\usepackage{caption}
\usepackage{subcaption}
\usepackage{mathtools}

\usepackage[all]{xy}
\xyoption{matrix}
\xyoption{arrow}

\usetikzlibrary{arrows,decorations.pathmorphing,backgrounds,positioning,fit,petri}

\tikzset{help lines/.style={step=#1cm,very thin, color=gray},
help lines/.default=.5} 
\tikzset{thick grid/.style={step=#1cm,thick, color=gray},
thick grid/.default=1} 

\newtheorem{thm}{Theorem}[section]
\newtheorem{lem}[thm]{Lemma}
\newtheorem{cor}[thm]{Corollary}
\newtheorem{prop}[thm]{Proposition}

\newtheorem{innercustomthm}{Theorem}
\newenvironment{customthm}[1]
{\renewcommand\theinnercustomthm{#1}\innercustomthm}{\endinnercustomthm}

\newtheorem{innercustomlem}{Lemma}
\newenvironment{customlem}[1]
{\renewcommand\theinnercustomlem{#1}\innercustomlem}{\endinnercustomlem}

\newtheorem{innercustomcor}{Corollary}
\newenvironment{customcor}[1]
{\renewcommand\theinnercustomcor{#1}\innercustomcor}{\endinnercustomcor}

\theoremstyle{definition}
\newtheorem{defn}[thm]{Definition}
\newtheorem{eg}[thm]{Example}

\theoremstyle{remark}
\newtheorem{rem}[thm]{Remark}

\numberwithin{equation}{section}

\newcommand{\ZZ}{{\ensuremath{\mathbb{Z}}}}

\newcommand{\RR}{{\ensuremath{\mathbb{R}}}}

\newcommand{\Hom}{{\ensuremath{\rm Hom}}}%
\newcommand{\Aut}{{\ensuremath{\rm Aut}}}%

\newcommand{\cA}{\ensuremath{{\mathcal{A}}}}
\newcommand{\cB}{\ensuremath{{\mathcal{B}}}}
\newcommand{\cC}{\ensuremath{{\mathcal{C}}}}
\newcommand{\cD}{\ensuremath{{\mathcal{D}}}}

\newcommand{\cI}{\ensuremath{{\mathcal{I}}}}
\newcommand{\cJ}{\ensuremath{{\mathcal{J}}}}
\newcommand{\cK}{\ensuremath{{\mathcal{K}}}}

\newcommand{\cM}{\ensuremath{{\mathcal{M}}}}

\newcommand{\cO}{\ensuremath{{\mathcal{O}}}}
\newcommand{\cP}{\ensuremath{{\mathcal{P}}}}

\newcommand{\cR}{\ensuremath{{\mathcal{R}}}}
\newcommand{\cS}{\ensuremath{{\mathcal{S}}}}

\newcommand{\cX}{\ensuremath{{\mathcal{X}}}}

\newcommand{\Ob}{\cO b}
\newcommand{\Mor}{\cM or}
\newcommand{\vare}{\varepsilon}

\newcommand{\mat}[1]{\ensuremath{
\left[\begin{matrix}#1
\end{matrix}\right]
}}

\title{Continuously triangulating the continuous cluster category}

\author{Matthew Garcia}
\address{Department of Mathematics, Brandeis University, Waltham, MA 02454}\email{mcgarcia@brandeis.edu}

\author{Kiyoshi Igusa}
\address{Department of Mathematics, Brandeis University, Waltham, MA 02454}\email{igusa@brandeis.edu}



\begin{document}

\begin{abstract}
In \cite{ccc}, the continuous cluster category was introduced. This is a topological category whose space of isomorphism classes of indecomposable objects forms a Moebius band. It was found in \cite{ccc} that, in order to have a continuously triangulated structure on this category, one needs at least two copies of each indecomposable object forming a 2-fold covering space of the Moebius band. This paper classifies all continuous triangulations of finite coverings of the basic continuous cluster category. This includes the connected 2-fold covering of Igusa-Todorov \cite{ccc}, the disconnected 2-fold covering of Orlov \cite{Orlov} and a third unexpected continuously add-triangulated 2-fold covering of the Moebius strip category.
\end{abstract}

\maketitle

\tableofcontents


%
%


\section*{Introduction}

Cluster categories were introduced by Buan, Marsh, Reineke, Reiten and Todorov \cite{BMRRT} to categorify the cluster algebras of Fomin and Zelevinsky \cite{FZ}. This has become an active area of research \cite{A}, \cite{BIRS}, \cite{R}, \cite{P}, \cite{posets}. 

This paper is about the continuous cluster category, introduced in \cite{ccc} and expanded in \cite{ccc2}, \cite{cccD}, \cite{IRT}. This is basically the limit of the cluster categories of type $A_n$ as $n\to \infty$. In this continuous limit, the set of isomorphism classes of indecomposable objects is in bijection with the open Moebius band. Thus, the continuous cluster category has the natural structure of a topological category. A surprising feature of this topology is that the minimal models must contain at least two copies of each indecomposable object giving a covering space of the Moebius band. This paper gives a classification of these categories. An unexpected feature of this classification is that, up to isomorphism, there are three 2-fold coverings of the Moebius strip category! (Corollary \ref{cor: E2}.)

The description of possible topologies on the (infinite) continuous cluster category comes from the study of the simplest of categories. Following the footsteps of the Pythagoreans, we examine the following quiver:
\begin{center}
\begin{tikzpicture}
\coordinate (A) at (0,0);
\begin{scope}
	\draw[fill] (A) circle[radius=3pt];
\end{scope}
\end{tikzpicture}
\end{center}
This is the quiver, called $A_1$, with one vertex and no arrows. This quiver represents a $K$-category with one object with endomorphism ring equal to the field $K$. We consider an equivalent $K$-category $\cC_n$ with $n$ isomorphic objects $1,2,\cdots,n$ so that $\cC_n(i,j)=K$ for all $i,j$ (Def. \ref{def: Cn}, Thm. \ref{thm: fat K is Cn}). Continuous triangulations of the continuous cluster category will be given by triples $(\sigma,\tau,\varphi)$ of discrete structures on this finite category $\cC_n$ for some $n\ge2$.

{Briefly, $\sigma,\tau$ are commuting autoequivalences of $\cC_n$, with $\sigma$ being an automorphism, and $\varphi:\sigma\to \tau$ a ``skew-continuous'' natural isomorphism (Def. \ref{def: skew-continuous}) the existence of which puts a severe restriction on $\tau$ (by Prop. \ref{prop: summary of skew-continuity}). These discrete structures on the finite category $\cC_n$ will be classified in Section \ref{Sec 5: App} assuming that $K$ is an algebraically closed field of characteristic not equal to 2. The main body of this paper describes how a continuously triangulated topological $K$-category which is algebraically equivalent to the continuous cluster category is given by such a triple of discrete structures.}

\subsection{Topological $K$-categories}

The paper begins with a precise description of topological $K$-categories. Basically, it is a $K$-category together with a topology on the set of objects and the set of morphisms so that all structure maps are continuous. We restrict attention to the topological full subcategory of indecomposable objects and assume that the rest of the category is constructed in a canonical way using the ``James construction'' (Def. \ref{def: James construction}).

The two main examples we consider are the circle category $\cS^1$ and open Moebius band category $\cM_0$ (Def. \ref{def: S1 category}, \ref{def: Moebius strip category M0}). These are topological $K$-categories where the endomorphism ring of each object is $K$ (making each object indecomposable) and the space of objects is the circle or open Moebius band, respectively. During this discussion, we recall the definition of the continuous cluster category from \cite{ccc} and the topological difficulties we are working to overcome. Basically the problem is the negative sign in the rotation axiom for a triangulated category and the fact that distinguished triangles can be continuously rotated which avoids the negative sign. Since we are taking $K$ to have the discrete topology, we cannot move continuously from 1 to $-1$ (since $char\,K\neq2$). So, the only resolution of this problem is to replace the category with an equivalent covering category.

\subsection{Equivalence coverings}

To do covering theory, we assume that the object space of our topological category $\cB$ is connected and locally simply-connected. We define the finite category $\cC_n$ and derive the basic properties of $\Aut(\cC_n)$. Then, we show that an $n$-fold ``equivalence coverings'' (Def. \ref{def: equivalence covering}) $\widetilde \cB$ of $\cB$ gives a homomorphism
\[
	\sigma:\pi_1(\cB,X_0)\to \Aut(\cC_n),
\]
well-defined up to conjugation. This is called the ``holonomy'' of $\widetilde \cB$. Conversely, for any homomorphism $\sigma$ as above, there is an equivalence covering $\widetilde \cB_\sigma$ with holonomy $\sigma$. And, by Lemma \ref{lem: morphism between equiv covers}, any two equivalence coverings with conjugate holonomies are continuously isomorphic.

Since homomorphisms $\ZZ\to \Aut(\cC_n)$ are given by the image of the generator of $\ZZ$, which we call the ``holonomy functor'' of the covering, we get the following.

\begin{customthm}{A}[Theorem \ref{thmA}]\label{thmA1}
$n$-fold equivalence coverings of $\cS^1$ and $\cM_0$ are classified by their holonomy functors which are automorphisms of $\cC_n$ well-defined up to conjugation.
\end{customthm}

Given $\sigma\in\Aut(\cC_n)$, the corresponding $n$-fold equivalence coverings of $\cS^1$ and $\cM_0$ are denoted $\widetilde\cS^1_\sigma$ and $\widetilde\cM_\sigma$. More generally, we denote by $\widetilde\cB_\sigma$ the equivalence covering of $\cB$ with holonomy $\sigma:\pi_1\cB\to \Aut(\cC_n)$.

\subsection{Autoequivalences of equivalence coverings}

Next we consider continuous autoequivalences of equivalence coverings $\widetilde\cB_\sigma$ which cover the identity functor on $\cB$. Using an explicit construction of $\widetilde\cS_\sigma^1$ and $\widetilde\cM_\sigma$ we prove the following.

\begin{customthm}{B}[special case of Theorem \ref{thmB}]\label{thmB1}
Continuous autoequivalences $F_\tau$ of $\widetilde\cS^1_\sigma$ and $\widetilde\cM_\sigma$ are given by autoequivalences $\tau$ of $\cC_n$ which commute with $\sigma$.
\end{customthm}

Since $\sigma$ commutes with itself, we have, in particular, continuous automorphisms of $\widetilde\cS^1_\sigma$ and $\widetilde\cM_\sigma$ given by the automorphism $\sigma$. Both are denoted  $F_\sigma$.

\subsection{Skew-continuous natural isomorphism}

The last factor needed to construct a continuous triangulation of $add\,\widetilde\cM_\sigma$ is a ``skew-continuous'' natural isomorphism $\varphi:\sigma\to \tau$. We explain briefly how this is used and why $\varphi$ cannot be continuous. We are constructing a triangulated category with distinguished triangles
\[
	X\xrightarrow f Y\xrightarrow g Z\xrightarrow \psi TX
\]
where $X,Y,Z$ are direct sums of objects in $\widetilde\cM_\sigma$. The shift functor is $T=F_\tau$. But the morphism $\psi:Z\to TX=F_\tau X$ is the composition of two morphisms: a morphism $h:Z\to F_\sigma X$ followed by a ``skew-continuous'' isomorphism $F_\varphi X:F_\sigma X\cong F_\tau X$.

By continuity, $F_\sigma$ takes distinguished triangles to distinguished triangles, but $T=F_\tau$ does not (by the rotation axiom) unless a negative sign is inserted. But $F_\sigma$ and $F_\tau$ are naturally (but discontinously) isomorphic by $F_\varphi$. This gives the following commuting diagram in which both horizontal sequences are distinguished triangles.
\[
\xymatrixrowsep{30pt}\xymatrixcolsep{40pt}
\xymatrix{
F_\sigma X\ar[r]^{F_\sigma f}\ar[d]^{F_\varphi(X)}&
F_\sigma Y\ar[r]^{F_\sigma g}\ar[d]^{F_\varphi(Y)}&
F_\sigma Z\ar[r]^{F_\sigma h}\ar[d]^{F_\varphi(Z)}&
F_\sigma F_\sigma X\ar[r]^(.4){F_\sigma(F_\varphi X)}\ar[d]^{F_\varphi(F_\sigma X)}& 
F_\sigma F_\tau X=F_\tau F_\sigma X\ar[d]_\cong^{F_\tau (F_\varphi X)} \\
F_\tau X\ar[r]^{F_\tau f} &
F_\tau Y\ar[r]^{F_\tau g} &
F_\tau Z\ar[r]^{F_\tau h} &
F_\tau F_\sigma X\ar[r]_{\cong}^{-F_\tau (F_\varphi X)} &
F_\tau F_\tau X
	}
\]
The first three squares commute since $F_\varphi$ is a natural transformation. This forces the last square to commute. Since $F_\tau(F_\varphi X)$ is an isomorphism this forces
\[
	F_\sigma F_\varphi=-F_\varphi F_\sigma.
\]
So, $F_\varphi$ is not a continuous functor (since it does not commute with holonomy $F_\sigma$). Because of this minus sign we say that $F_\varphi$ and $\varphi$ are ``skew-continuous''. Proposition \ref{prop: summary of skew-continuity} implies that the existence of $\varphi$ imposes the following restriction on $\sigma$ and $\tau$.

\begin{customlem}{C}[Proposition \ref{prop: summary of skew-continuity}]\label{lemC1}
The following are equivalent.
\begin{enumerate}
\item There exists a skew-continuous natural isomorphism $\varphi:\sigma\to\tau$.
\item Every natural isomorphism $\varphi:\sigma\to\tau$ is skew-continuous.
\item $\sigma,\tau$ are anti-compatible (Def. \ref{def: anti-compatible}).
\end{enumerate}
\end{customlem}

Lemma \ref{lemC1} puts many conditions on $\sigma,\tau$ and $n$:
\begin{enumerate}
\item[(a)] Since $\sigma$ is compatible with itself, $\sigma\neq \tau$.
\item[(b)] Since the identity automorphism of $\cC_n$ is compatible with every autoequivalence of $\cC_n$, we have $\sigma, \tau\neq id$. (However, on the objects of $\cC_n$, the underlying permutations of $\sigma,\tau$ are allowed to be the identity.)
\item[(c)] For $n=1$, any two autoequivalences of $\cC_1$ are compatible. So, we must have $n\ge2$.
\end{enumerate}

\subsection{Classification Theorem}

Using Lemma \ref{lemC1}, the main theorem can be phrased as follows. We use the term ``continuous add-triangulation'' to refer to a continuous triangulation of the topological additive category generated by an equivalence covering $\widetilde \cM_\sigma$ of $\cM_0$ (Def.  \ref{def: continuous triangulation}, \ref{def: add-triangulation}).

\begin{customthm}{D}[Theorem \ref{thm: D in intro}]\label{thmD1}
Continuous add-triangulations, $\widetilde \cM_n(\sigma,\tau,\varphi)$, of $n$-fold equivalence coverings of $\cM_0$ are classified by triples $(\sigma,\tau,\varphi)$ where:
\begin{enumerate}
\item $\sigma$ is a $K$-linear automorphism of $\cC_n$.
\item $\tau$ is a $K$-linear autoequivalence of $\cC_n$ which commutes with $\sigma$ but which is anti-compatible with $\sigma$.
\item $\varphi:\sigma\to \tau$ is a skew-continuous natural isomorphism.
\end{enumerate}
Furthermore, continuous (strong) isomorphisms of triangulated categories $\widetilde \cM_n(\sigma,\tau,\varphi)\cong\widetilde \cM_n(\sigma',\tau',\varphi')$ are given by $\rho\in\Aut(\cC_n)$ so that $\sigma'\rho=\rho\sigma$, $\tau'\rho=\rho\tau$. In particular, $\widetilde \cM_n(\sigma,\tau,\varphi)\cong\widetilde \cM_n(\sigma,\tau,\varphi')$ for any $\varphi,\varphi'$.
\end{customthm}

To prove this theorem we show that the set of all distinguished triangles is determined by a single continuous family of distinguished triangles which we call the ``universal virtual triangle'' (Definition \ref{def: universal virtual triangle}). This is determined by $\sigma,\tau,\varphi$ and we show that such a triple determines a ``continuous Frobenius category'' ${add\,\overline \cM_\sigma(K[[t]])}$ whose underlying subcategory of indecomposable objects is an $n$-fold covering of the ``closed Moebius category'' $\overline \cM$ and whose stable category is $add\,\widetilde\cM_n(\sigma,\tau,\varphi)$. (Theorem \ref{thm: add M(s,t,phi) is continuously triangulated}.)

One consequence of the classification theorem, together with Lemma \ref{lemC1} and the fact that anti-compatibility is a symmetric relation (Corollary \ref{cor: anti-symmetry of s-t pairing}), is the following duality.

\begin{customcor}{E1}[Corollary \ref{duality of equivalence coverings}]\label{cor: E1}
If $\widetilde \cM_n(\sigma,\tau,\varphi)$ is a continuous add-triangulation of an $n$-fold equivalence covering of $\cM_0$ then so is $\widetilde \cM_n(\tau,\sigma,\varphi^{-1})$ (provided $\tau$ is an isomorphism).
\end{customcor}

We call $\widetilde \cM_n(\tau,\sigma,\varphi^{-1})$ the \emph{dual} of $\widetilde \cM_n(\sigma,\tau,\varphi)$. Finally, we classify the minimal examples given by $n=2$.

\begin{customcor}{E2}[Corollary \ref{cor: classification for n=2}]\label{cor: E2}
Up to isomorphism there are three add-triangulated 2-fold equivalence coverings of $\cM_0$: The continuous cluster category \cite{ccc} which is self-dual, Orlov's construction \cite{Orlov} (as a special case of a more general construction) and one new triangulated category which is the dual of Orlov's construction.
\end{customcor}

The last two sections deal with some of the technical details: the combinatorics of the category $\cC_n$ and construction of the continuous Frobenius categories needed to complete the proof of the classification Theorem \ref{thmD1}.

\subsection{Details of the category $\cC_n$}

In Section \ref{Sec 5: App} we discuss the combinatorics of the finite category $\cC_n$. This is the discrete category with object set $[n]=\{1,2,\cdots,n\}$ with $\cC_n(i,j)\cong K$ for all $i,j$, i.e., every morphism $i\to j$ is a scalar multiple of a basic morphism $x_{ji}$. The structures $\sigma,\tau,\varphi$ have easy descriptions in terms of these scalars. Section \ref{Sec 5: App} contains detailed proofs of needed technical results about these discrete structures. As an example, in Theorem \ref{thm: connected equivalence covers} we enumerate all connected add-triangulated equivalence coverings of $\cM_0$.

\subsection{Continuous Frobenius categories}

In the last section, Section \ref{ss: continuous Frobenius}, we show the existence of the objects $\widetilde\cM_n(\sigma,\tau,\varphi)$ of the classification Theorem \ref{thmD1} by constructing a continuous Frobenius category $add\,\overline \cM_\sigma(K[[t]])$.

Recall from Happel \cite{hap} that the stable category of a Frobenius category is triangulated. Following \cite{cfc}, we show that the stable category of the continuous Frobenius category $add\,\overline \cM_\sigma(K[[t]])$ is continuously triangulated in several ways determined by $\tau$ and $\varphi:\sigma\to \tau$.

Topologically, the subcategory $\overline \cM_\sigma(K[[t]])$ of indecomposable objects is a covering of the closed Moebius strip category. The stable category is the additive category of the corresponding open Moebius strip since the boundary consists of the projective-injective objects which become zero in the stable category.

%
%

\section{Topological $K$-categories and examples}\label{sec1}

In this section we construct the key examples of topological $K$-categories: the continuous path categories $\cP_J$ with object space $J\subseteq\RR$, the quotient category $\cR_c=\cP_\RR/\cI_c$, the circle category $\cS^1$ with object space $S^1=\RR/\pi\ZZ$ (the circle with diameter $1$) and the Moebius band category $\cM_0$. We begin with the basic definitions.

%
%

\subsection{Topological $K$-categories in general} 

\begin{defn}\label{def: topological categories and functors}
By a \emph{topological category} we mean a small category $\cR$ so that the set of objects $\cO b(\cR)$ and the set of morphisms $\cM or(\cR)$ are topological spaces and the four structure maps of the category are continuous maps:
\begin{enumerate}
\item source and target maps: $s,t:\cM or\to \cO b$
\item identity map: $id:\cO b\to \cM or$
\item composition: $c:\cM or\oplus \cM or\to\cM or$ where $\cM or\oplus \cM or$ is the subset of $\cM or^2$ on which composition is defined.
\end{enumerate}
Since $\cO b(\cR)$ is a retract of $\cM or(\cR)$, the topology on the object space is uniquely determined by the topology on the morphism space. A \emph{continuous functor} between topological categories is a functor $F:\cB\to\cC$ which is a continuous mapping on morphism spaces (and, consequently, also on object spaces). A {natural transformation} $\psi:F\to G$ is \emph{continuous} if the assignment of $\psi_X:FX\to GX$ to each $X\in \cB$ gives a continuous function $\psi:\cO b(\cB)\to\cM or(\cC)$ from the space of objects of $\cB$ to the space of morphisms of $\cC$.
\end{defn}

Recall that a \emph{$K$-category} is an additive category $\cA$ so that each hom set $\Hom(A,B)$ is a vector space over $K$ and composition is $K$-bilinear. By a \emph{topological $K$-category} we mean a $K$-category $\cA$ which is also a topological category so that 
\begin{enumerate}
\item The action of $K$ on hom sets gives a continuous mapping: $K\times \cM or(\cA)\to\cM or(\cA)$.
\item $\Hom(A,B)$ has the discrete topology for all $A,B$.
\end{enumerate}

\begin{eg}\label{eg: trivial example}
For any topological space $X$, let $KX$ be the topological $K$-category with object space $X$ and morphisms being scalar multiples of identity morphisms and zero morphisms, i.e., the morphism space is $K\times X\cup X^2$ where $(0,x)\in K\times X$ is identified with $(x,x)\in X^2$. Composition is given in the obvious way. Any continuous mapping $f:X\to Y$ induces a continuous $K$-linear functor $Kf:KX\to KY$. Also, for any topological $K$-category $\cA$ there is a continuous $K$-linear embedding $KX\to \cA$ where $X$ is the space of nonzero objects of $\cA$. We call $KX$ the \emph{trivial $K$-category} generated by $X$.
\end{eg}

\begin{defn}\label{def: continuous triangulation}
Let $\cR$ be a topological $K$-category. By a \emph{continuous triangulation} of $\cR$ we mean the structure of a triangulated category on $\cR$ so that:
\begin{enumerate}
\item The shift functor $TX=X[1]$ is a continuous linear functor $T:\cB\to\cB$.
\item The set of distinguished triangles $X\to Y\to Z\to TX$ forms a closed subspace ${\bf \Delta}\subset\cO b(\cR)^3\times \cM or(\cR)^3$.
\end{enumerate}
A \emph{continuous triangle functor} between continuously triangulated $K$-categories $(\cB,T,{\bf\Delta})$, $(\cB',T',{\bf\Delta}')$ is a pair $(F,\psi)$ where $F:\cB\to\cB'$ is a continuous linear functor and $\psi:FT\cong T'F$ is a continuous natural $K$-linear isomorphism so that for every distinguished triangle $(X,Y,Z,f,g,h)\in\bf\Delta$, $(FX,FY,FZ,Ff,Fg,\psi_X\circ Fh)\in \bf\Delta'$. We say $(F,\psi)$ is a (strong) \emph{isomorphism} of continuously triangulated categories if $F$ is an isomorphism and $FT=T'F$ (but $\psi:FT\cong T'F$ is not necessarily the identity).
\end{defn}

\begin{rem}\label{rem: weak triangle isomorphisms}
If we drop the assumption $FT=T'F$ we get a notion of isomorphism which is too coarse for our purposes. For example, all continuous triangulations of a covering category $\widetilde \cM_\sigma$ would be isomorphic,  the classification would be reduced to covering theory and we would lose any concept of what are the distinguished triangles in each case.
\end{rem}

\begin{eg}\label{rem: rescaling gives strong isomorphism}
Given a topological triangulation $(T,\Delta)$ of a topological $K$-category $\cR$ and any nonzero scalar $a\in K^\ast$, let $\Delta_a$ be the set of all sextuples $(X,Y,Z,f,g,h)$ so that $(X,Y,Z,f,ag,h)\in \Delta$. Then $(T,\Delta_a)$ gives another triangulation of $\cR$ which is strongly isomorphic to the first triangulation. The strong isomorphism is the identity functor $F=id_\cR$ together with $\psi_X=a\cdot id_{TX}:TX\to TX$. We say that the triangulation $(T,\Delta_a)$ is a \emph{rescaling} of the triangulation $(T,\Delta)$.
\end{eg}

In the construction of variations of the continuous cluster category we will construct only the full topological subcategory of indecomposable objects with $0$ attached by one-point compactification. We recall that the \emph{one-point compactification} $X_+$ of a locally compact Hausdorff space $X$ is given by adding a disjoint point $\ast$, defining the open neighborhoods of $\ast$ to be the complements in $X_+$ of compact subsets of $X$ and the other open subsets of $X_+$ are defined to be the open subsets of $X$. Recall that a continuous mapping $f:X\to Y$ is called \emph{proper} if the inverse image of every compact subset of $Y$ is a compact subset of $X$. Then the following basic fact follows directly from the definitions.

\begin{lem}\label{lem: proper maps induce maps on one point compactifications}
Any proper mapping $f:X\to Y$ between locally compact spaces $X,Y$ induces a continuous mapping on one-point compactifications $f_+:X_+\to Y_+$.\qed
\end{lem}

\begin{defn}\label{def: one-point compactification of locally compact category}
Let $\cB$ be a topological $K$-category with a locally compact space of objects and no zero object. Then the \emph{one-point compactification category} $\cB_+$ is defined to be the topological $K$-category whose object space is the one-point compactification of the object space of $\cB$ with the additional point being 0 and with morphism space
\[
	\cM or(\cB_+)=\cM or(\cB)\cup 0\times \cO b(\cB)\cup \cO b(\cB)\times 0
\] 
with topology given as follows. The subspace $\cM or(\cB)$ is an open subspace of $\cM or(\cB_+)$ with the same topology as before. A basic open neighborhood of any zero morphism $0:x_0\to y_0$ where either $x_0$ or $y_0$ is zero is defined to be $\cB_+(U,V):=\{f:x\to y\,|\, x\in U,y\in V\}$ where $U,V$ are open neighborhoods of $x_0,y_0$ respectively in $\cO b(\cB_+)=\cO b(\cB)_+$. The algebraic structure of $\cB_+$ is the obvious one given by the fact that composition of any morphism with a zero morphism is zero.
\end{defn}

A continuous $K$-linear functor $F:\cB\to\cC$ between locally compact $K$-categories is \emph{proper} if the induced map on object spaces is proper. Since the topology on $\cM or(\cB_+)$ is given in terms of the topology on the object space, we have the following extension of the basic Lemma \ref{lem: proper maps induce maps on one point compactifications}.

\begin{prop}\label{prop: proper functors induce continuous functors on one-pt compactification categories}
A continuous proper $K$-linear functor $F:\cB\to\cC$ between locally compact topological $K$-categories induces a unique continuous $K$-linear functor on one-point compactification categories: $F_+:\cB_+\to \cC_+$.\qed
\end{prop}

In all examples, the space of indecomposable objects will be locally compact. The entire category can then be canonically constructed using the \emph{James construction} given as follows.

\begin{defn}\label{def: James construction}
Let $\cB$ be a topological $K$-category with no zero object and with locally compact space of objects. Let $\cB_+$ be the one-point compactification category. Then we define the \emph{topological additive category generated by $\cB$} to be the category $add^{top}\cB$ (denoted $add^{sa}\cB$ in \cite{ccc}) with object space given by the James construction with $0$ as base point \cite[p.224]{Ha}, i.e., it is a quotient of the union of all products 
\[
	\cO b(add^{top}\cB)=\coprod \cO b(\cB_+)^n/\sim
\]
with the quotient topology where the objects are ordered formal direct sums $X_1\oplus\cdots\oplus X_n$ of nonzero objects of $\cB_+$ with summands deleted when they converse to $0$. Morphisms spaces are given by products of morphism spaces: $\Hom(\oplus X_i,\oplus Y_j)=\prod_{ij}\Hom(X_i,Y_j)$, again with the quotient topology:
\[
	\cM or(add^{top}\cB)=\coprod \cM or(\cB_+)^{nm}/\sim
\]
This is a strictly monoidal category which is not strictly symmetric since $f\oplus g\neq g\oplus f$ in general. See \cite{ccc} for more details. 
\end{defn}


\subsubsection{Quotient and orbit categories}

Recall that an \emph{ideal} in a topological $K$-category $\cB$ is a subset $\cI$ of the morphism space of $\cB$ so that $\cI(x,y):=\cI\cap \cB(x,y)$ is a vector subspace of $\cB(x,y)$ for every $x,y\in\cB$ and so that, $f\circ g\in \cI$ if either $f$ or $g$ is in $\cI$. The \emph{quotient category} $\cB/\cI$ is the topological $K$-category with the same object space as $\cB$ but with hom-spaces $(\cB/\cI)(x,y)=\cB(x,y)/\cI(x,y)$ so that the entire morphism space is given the quotient topology with respect to the surjective map $\Mor(\cB)\to \Mor(\cB/\cI)$. Any set of morphisms $X$ generates an ideal $\cI_X$, namely the intersection of all ideals containing that set. Then $\cI_X(x,y)$ is the vector space spanned by all morphisms $x\to y$ factoring through some element in $X$. The ideal generated by a set of objects is defined to be the ideal generated by the identity morphisms of those objects.

Any continuous linear functor $F:\cB\to \cB'$ which is zero on every morphism in $\cI$ induces a unique continuous linear functor $\cB/\cI\to \cB'$. Also, the \emph{kernel} of $F$, the set of all morphisms in $\cB$ which go to zero in $\cB'$, is always an ideal.

\smallskip

Recall that the action of a discrete group $G$ on a space $X$ is called \emph{properly discontinuous} if every $x_0\in X$ has an open neighborhood $U$ so that $gU\cap hU=\emptyset$ when $g\neq h$ in $G$. In that case $X$ will be a covering space of the orbit space $X/G$.

Suppose that $\cB$ is a topological $K$-category and $F$ is a continuous linear automorphism of $\cB$ which acts {property discontinuously} on object and morphism spaces of $\cB$, i.e., the action of the group $G$ of automorphisms of $\cB$ generated by $F$ is properly discontinuous. The \emph{orbit category} $\cB/F$ is the topological $K$-category with object and morphism spaces given by the orbits of the action of this group $G$. Note that any automorphism or endomorphism of $\cB$ which commutes with $F$ induces an automorphism/endomorphism of $\cB/F$.

\subsection{Outline of construction of circle and Moebius band categories}

First, we construct the ``continuous path category'' $\cP_\RR$ (Def. \ref{def: path category}). This has object space $\RR$ and morphisms $\cP_\RR(x,y)=K$ if $x\le y$ and $\cP_\RR(x,y)=0$ otherwise. The generator of $\cP_\RR(x,y)$ (corresponding to $1\in K$) is called the ``{basic morphism}'' (Def. \ref{def: path category}). Composition of basic morphisms is defined to be a basic morphism. For any $J\subseteq \RR$, $\cP_J$ is the full subcategory with object set $J$. Nonzero morphisms $f:x\to y$ in the path category have length $\ell(f)=y-x\ge0$. For any $c>0$, the ``{truncated path category}'' (Def. \ref{def: truncated path category}) is given by $\cR_c=\cP_\RR/\cI_c$ where $\cI_c$ is the ideal of all morphisms of length $\ge c$ (and all zero morphisms). Figure \ref{fig1} illustrates the morphism set of this category.

We take $c=\pi$ and define the ``{circle category}'' (Def. \ref{def: S1 category}) to be $\cS^1=\cR_\pi/G_\pi=\cR_\pi/G_{-\pi}$ where $G_t$, $t\in\RR$, is the continuous family of continuous automorphisms of $\cR_\pi$ given by sending $x$ to $x+t$ and basic morphisms to basic morphisms.

Let $\cP_\RR^2=\cP_\RR\otimes\cP_\RR$. This is the $K$-category with object space $\RR^2$ and morphisms $\cP_\RR^2(x,y)=K$ if $x_1\le y_1$ and $x_2\le y_2$ and $\cP_\RR^2(x,y)=0$ otherwise. Any composition of basic morphisms is a basic morphism. For any $c>0$ this category has an ideal $\cJ_c$ consisting of all morphisms which factor through an object $z=(z_1,z_2)$ with $|z_2-z_1|\ge c$. Let $\cD_c$ be the full subcategory of nonzero objects in the quotient category $\cP_\RR^2/\cJ_c$. Then $\cD_c$ is a locally compact $K$-category with object space equal to the set of all $(x,y)\in\RR^2$ with $|x-y|<c$.

The topological $K$-category $\cD_c$ has a continuous family of automorphisms $G^2_t$ given on objects by $G^2_t(x)=x+(t,t)$ and taking basic morphisms to basic morphisms. Another continuous automorphism of $\cD_c$ is $S(x_1,x_2)=(x_2,x_1)$. The ``{Moebius band category}'' (Def. \ref{def: Moebius strip category M0}) is $\cM_0=\cD_\pi/SG^2_{\pi}=\cD_\pi/SG^2_{-\pi}$. The object space of $\cM_0$ is the set of all $(x,y)\in \RR^2$ so that $|x-y|<\pi$ modulo the equivalence relation $(x,y)\sim (y+\pi,x+\pi)$. This is homeomorphic to the space of all unordered pairs of distinct points on the circle $\widetilde S^1=\RR/2\pi\ZZ$ where $(x,y)\in\cM_0$ corresponds to the pair $\{x,y+\pi\}\subset \widetilde S^1$. These two elements of $\widetilde S^1$ will be called the ``{ends}'' of the object $(x,y)\in \cM_0$ (Def. \ref{def: E(a,b) and ends}).

\subsection{Continuous path categories}

We construct two topological categories with object space $\RR$.

\begin{defn}\label{def: path category}
Let $\cP_\RR=\cP_\RR(K)$ be the topological $K$-category with object space $\RR$ and morphism space
\[
	\Mor(\cP_\RR):=K^\ast\times \{(x,y)\in \RR^2\,|\,x\le y\}\coprod 0\times \RR^2
\]
where $K^\ast=K\backslash\{0\}$ with the discrete topology. Composition of morphisms is defined by
\[
	(b,y,z)\circ(a,x,y)=(ab,x,z).
\]
For any subset $J\subseteq\RR$, $\cP_J$ is the full subcategory of $\cP_\RR$ with objects space $J$. We refer to each $\cP_J$ as a \emph{continuous path category}. The morphism $(1,x,y)$ is called a \emph{basic morphism}.
\end{defn}

For any $c>0$, let $\cI_c$ be the ideal in $\cP_\RR$ generated by the morphisms $(1,x,x+c)$ for all $x\in\RR$. Then $\cI_c$ consists of all morphism of length $\ge c$ and all zero morphisms where the \emph{length} of a morphism is defined to be $\ell(a,x,y)=y-x$.

\begin{defn}\label{def: truncated path category}
The \emph{truncated path category} $\cR_c=\cR_c(K)$ is defined to be the quotient category $\cR_c=\cP_\RR/\cI_c$. This is the topological $K$-category with object space $\cO b(\cR_c)=\RR$ and morphisms 
\[
	\cR_c(x,y) =\begin{cases} K & \text{if } x\le y<x+c\\
   0 & \text{otherwise}
    \end{cases}
\]
The set of all morphism is:
\[
	\cM or(\cR_c)=0\times \RR^2\coprod K^\ast\times\{(x,y)\in\RR^2\,:\, y\in [x,x+c)\}
\]
This has the quotient topology $\cM or(\cR_c)=\cM or'(\cR_c)/\sim$ where
\[
	\cM or'(\cR_c)=0\times \RR^2\coprod K^\ast\times\{(x,y)\in\RR^2\,:\, y\in [x,x+c]\}
\]
under the identification $(a,x,x+c)\sim (0,x,x+c)\in K^\ast\times \RR^2$. 
\end{defn}

\begin{rem}\label{rem: quotient topology} Quotient topology means the morphism $f:x\to y$ given by a fixed nonzero scalar $a\in K^\ast$ will converge to 0 when $y\to x+c$. However, a sequence of morphisms $f_i:x\to y_i$ with distinct scalars $a_i$ will not converge to anything, even if $y_i\to x+c$. The reason is that an open neighborhood of $(0,x,x+c)$ is given by a union over all $a\in K$ of open sets $(a,U_\vare(x),U_\vare(x+c))$ where $U_\vare(z)=(z-\vare,z+\vare)$. By choosing $\vare=|y_i-x|/2$ we get an open neighborhood of $(0,x,x+c)$ which avoid all of the points $f_i=(a_i,x,y_i)$.
\end{rem}

\begin{figure}[h!]\label{fig:example}
	\begin{subfigure}[h]{0.4\textwidth}
\begin{center}
\begin{tikzpicture}
	\draw[fill,color=gray!20!white] (-1,-1)--(2,2)--(1,2)--(-2,-1)--(-1,-1);
	\draw[very thick,->] (-2,0)--(2,0);
	\draw[very thick,->] (0,-1)--(0,2);
	\draw[thick] (-1,-1)--(2,2);
	\draw[dashed,thick] (-2,-1)--(1,2) (-.5,1) node{$c$};
\end{tikzpicture}
\end{center}
	\caption{Morphism space of $\cR_c(K)$.}
	\label{fig1A}
	\end{subfigure}\hspace{0.04\textwidth}
	\begin{subfigure}[h]{0.4\textwidth}
\begin{center}
\begin{tikzpicture}
\begin{scope}[yshift=-1cm]
	\draw[very thick,->] (-2,0) -- (2,0);
	\draw[thick] (-1.5,1)--(1,1) (1,1) .. controls (1.3,1) and (1.3,.8)..(1.3,.5) (1.5,0) .. controls (1.3,0) and (1.3,.2)..(1.3,.5); 
	\draw[thick] (-1.5,.75)--(1,.75) (1,.75) .. controls (1.25,.75) and (1.25,.6)..(1.3,.37) (1.5,0) .. controls (1.3,0) and (1.3,.15)..(1.25,.37); 
	\draw[thick] (-1.5,.5)--(1,.5) (1,.5) .. controls (1.2,.5) and (1.2,.4)..(1.25,.25) (1.5,0) .. controls (1.3,0) and (1.3,.1)..(1.25,.25); 
	\draw[thick] (-1.5,.25)--(1,.25) (1,.25) .. controls (1.2,.25) and (1.2,.2)..(1.25,.125) (1.5,0) .. controls (1.3,0) and (1.3,.05)..(1.25,.125); 
	\draw (-1.5,-.2) node{$x$} (1.5,-.2) node{$x+c$};
	\draw(-1.9,.6) node{$K$};
	\end{scope}
\end{tikzpicture}
\end{center}
	\caption{$\cR_c(x,y)$ goes to 0 when $y\to x+c$.}
	\label{fig1B}\end{subfigure}
\caption{The space of nonzero morphisms of $\cR_c(K)$ is a covering space of the contractible submanifold of $\RR^2$ shaded in (A).}\label{fig1}
\end{figure}
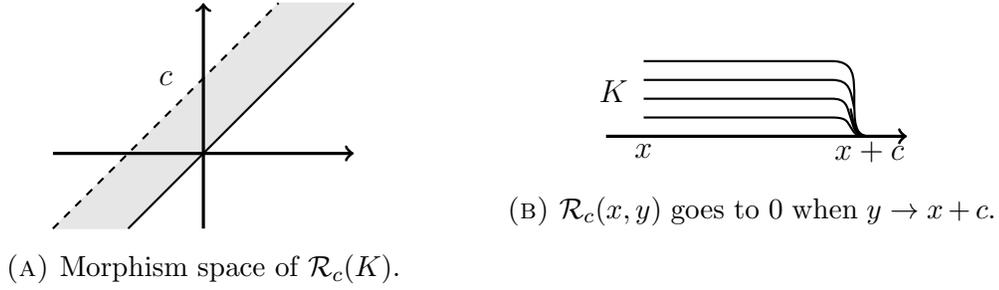

In \cite{ccc}, \cite{cfc}, the value $c=2\pi$ was taken. In this paper, we will take $c=\pi$ and define $S^1:=\RR/\pi\ZZ$, the ``\emph{circle with diameter 1}'' (with the standard circle denoted $\widetilde S^1=\RR/2\pi\ZZ$). The purpose of this will be apparent later.

\subsection{The circle category $\cS^1(K)$}

\begin{defn}\label{def: S1 category}
For any $t\in\RR$ let $G_t$ be the continuous linear automorphism of $\cR_\pi(K)$ given on objects by $G_t(x)=x+t$ and on morphisms by $G_t(a,x,y)=(a,x+t,y+t)$. (So, $G_0$ is the identity functor.) The \emph{circle category}  is defined to be the quotient of $\cR_\pi(K)$ modulo the continuous automorphism $G_\pi$. 
\[
	\cS^1=\cS^1(K):=\cR_\pi(K)/G_\pi
\]
This is a topological $K$-category with object space the circle $S^1=\RR/\pi\ZZ$. Elements of $S^1$ will be denoted $[x]=x+\pi\ZZ$.
\end{defn}

The fundamental domain of $G_\pi$ on objects is the closed interval $[0,\pi]$. (However, on morphisms, we should take $[0,2\pi]$.) By definition, $G_\pi$ sends the right point $\pi$ of the interval $[0,\pi]$ to its left endpoint 0. So, $G_\pi$ gives the ``holonomy'' of the covering map $p:\cR_\pi\to \cS^1$. 

\begin{prop}\label{prop: Rc is universal cover of S1}
The quotient functor $p:\cR_\pi(K)\to \cS^1(K)$ is a universal covering map on object spaces and morphism spaces and all four of these spaces are Hausdorff.
\end{prop}


\begin{proof}
The functor $G_\pi$ gives a properly discontinuous free $\ZZ$ action on object and morphism spaces of $\cR_\pi(K)$ both of which are contractible and locally simply connected.

The object spaces $\Ob(\cR_\pi)=\RR$ and $\Ob(\cS^1)=S^1$ are clearly Hausdorff as is $\Mor(\cP_\RR)$. So, it suffices to consider $\Mor(\cS^1)$. If $f\neq g$ have distinct source or target, they can be separated by disjoint open sets in $\Ob(\cS^1)^2$. So, assume $f\neq g\in\cS^1(X,Y)$ and $f\neq 0$. Then $\cS^1(X,Y)=K$ and $g=af$ for some $a\in K$. Then $f,g$ lift to distinct nonzero morphisms $\tilde f,a\tilde f\in \cP_\RR(x,y)$. Since $\tilde f$ does not lie in the closed set $\cI_\pi$, there is a small connected open neighborhood $U$ of $\tilde f$ disjoint from $\cI_\pi$. Then $U,aU$ will be disjoint open neighborhoods of $\tilde f,a\tilde f$ giving disjoint open neighborhoods of $f,g$.
\end{proof}

For any two objects $x,y\in \RR=\Ob(\cR_\pi)$ there is a unique integer $k$ so that $x\le (G_\pi)^k(y)=y+\pi k<x+\pi$ and $\cR_\pi(x,(G_\pi)^k(y))=K$ in that case and $\cR_\pi(x,(G_\pi)^j(y))=0$ for $j\neq k$. Thus, $\cS^1(X,Y)=K$ for all $X,Y\in \cS^1$. So, the morphism set is in bijection with $K\times S^1\times S^1$. However, this bijection is not a homeomorphism since $(a,x,y)$ converges to $(0,x,x+\pi)\sim (0,x,x)$ as $y\to x+\pi$ (Fig \ref{fig1}(B)). The space of nonzero morphisms of $\cS^1(K)$ is homeomorphic to $K^\ast\times S^1\times [0,\pi)$.

Define the \emph{support of an object $x$} in a topological $K$-category $\cB$ to be the set of all $y\in\cO b(\cB)$ so that $\cB(x,y)\neq0$ with the quotient topology with respect to the target map $t:\cM or(\cB)\to \cO b(\cB)$ restricted to the nonzero morphisms with source $x$. Then the support of $[x]\in \cS^1(K)$ is homeomorphic to $[x,x+\pi)$ which maps onto $S^1$ by a continuous bijection.

\subsection{Moebius strip category $\cM_0(K)$}

Let $\cP_\RR^2=\cP_\RR\otimes\cP_\RR$ be the topological $K$-category with object space $\RR^2$ and morphisms given by $\cP_\RR^2(x,y)=\cP_\RR(x_1,y_1)\otimes\cP_\RR(x_2,y_2)$. This is $K$ if $x_1\le y_1$, $x_2\le y_2$ and $0$ otherwise. We take the topology induced by the inclusion $\Mor(\cP_\RR^2)\subseteq K\times \RR^4$.

For any $c>0$, let $\cJ_c$ be the ideal in $\cP_\RR^2$ generated by all objects $x=(x_1,x_2)$ so that $|x_1-x_2|\ge c$. Let $\cD_c=\cD_c(K)$ be the full subcategory of nonzero objects in $\cP_\RR^2/\cJ_c$. This is the topological $K$-category with object space
\[
	\Ob(\cD_c)=\{x\in\RR^2\,|\, |x_2-x_1|<c\}.
\]
Note that, if $y_1\ge x_2+c$ then any morphism $(x_1,x_2)\to (y_1,y_2)$ will factor through $(x_2+c,x_2)\in \cJ_c$. Therefore, the space of nonzero morphisms of $\cD_c$ is given by
\[
	\Mor^\ast(\cD_c)=K^\ast\times \{(x,y)=((x_1,x_2),(y_1,y_2))\in\RR^2\times \RR^2\,|\, x_1\le y_1<x_2+c,x_2\le y_2<x_1+c\}.
\]
As in the case of $\cR_c(K)$, there is a zero morphism between any two objects and a morphism $(a,x,y):x\to y$ converges to zero when either $y_1\to x_2+c$ or $y_2\to x_1+c$.

\begin{prop}
$\cR_c$ is isomorphic as topological $K$-category to the full subcategory of $\cD_c$ of diagonal objects $x=(x_1,x_2)\in \RR^2$ with $x_1=x_2$.\qed
\end{prop}

As before, we take $c=\pi$. See Figure \ref{fig2}.

\begin{defn}\label{def: Moebius strip category M0}
To construct $\cM_0$ we need two automorphisms of $\cD_\pi$.
\begin{enumerate}
\item Let $S$ be the automorphism of $\cD_\pi$ given by \emph{switching the coordinates}: $S(x_1,x_2)=(x_2,x_1)$, $S(a,x,y)=(a,Sx,Sy)$.
\item For any $t\in\RR$ let $G^2_t=G_t\otimes G_t$ be the automorphism of $\cD_\pi$ given by $G^2_t(x)=x+(t,t)=(G_t(x_1),G_t(x_2))$, $G^2_t(a,x,y)=(a,G_t^2(x),G_t^2(y))$.
\end{enumerate}
These are continuous $K$-linear automorphisms of $\cD_\pi$ and $G^2_t,SG^2_t$ are fixed point free and properly discontinuous for $t\neq0$. The \emph{Moebius band (or Moebius strip) category} is defined to be the orbit category 
\[
	\cM_0:=\cD_\pi/SG^2_\pi.
\]
\end{defn}
We take as fundamental domain of the action of $SG^2_\pi$ the set of all $x\in \RR^2$ so that $0\le x_1+x_2\le 2\pi$ (and morphism from such $x$ to those $y$ where $0\le y_1+y_2\le 4\pi$). The point $SG^2_\pi(x)=(x_2+\pi,x_1+\pi)$ is identified with $x=(x_1,x_2)$.

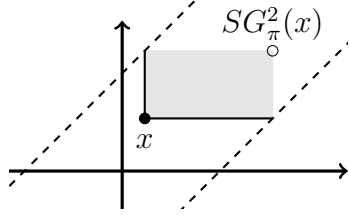
\begin{figure}[htbp]
\begin{center}
\begin{tikzpicture}
\clip (-1.5,-.5) rectangle (3,2.5);
\draw[fill,color=gray!20!white]  (2,.7) rectangle (.3,1.6);
	\draw[very thick,->] (-1.8,0)--(3,0);
	\draw[very thick,->] (0,-1)--(0,2);
	\draw[thick] (2,.7)--(.3,.7)--(.3,1.6);
	\draw[fill] (.3,.7) circle[radius=2pt];
	\draw[dashed,thick] (-1.8,-.5)--(1,2.3) (0.3,-1)--(3.6,2.3) (.3,.4) node{$x$};
	\draw (2,1.6) circle[radius=2pt] node[above]{$SG^2_\pi(x)$};
\end{tikzpicture}
\end{center}
	\caption{Support of $x$ in $\cD_\pi(K)$ is shaded.}
\label{fig2}
\end{figure}

\begin{defn}\label{def: E(a,b) and ends}
The objects of $\cM_0$ can be viewed as unordered pairs of distinct point on the circle $\widetilde S^1=\RR/2\pi\ZZ$ of radius 1. The pair $(a,b)\sim (b,a+2\pi)$ with $a<b<a+2\pi$ corresponds to $(a,b-\pi)\sim (b,a+\pi)$ which is in $\cM_0$ since $|a+\pi-b|<\pi$. We denote this object $E(a,b)$ and we call $a,b\in \widetilde S^1$ the \emph{ends} of $E(a,b)$.
\end{defn}

Analogous to Proposition \ref{prop: Rc is universal cover of S1} we have the following.

\begin{prop}\label{prop: Dc is universal cover of M0}
The quotient functor $p:\cD_\pi(K)\to \cM_0(K)$ is a universal covering map on object spaces and morphism spaces and all four spaces are Hausdorff.
\end{prop}

\subsection{Contractible morphisms}

The topological $K$-categories $\cS^1$ and $\cM_0$ share the important property that every nonzero morphism is a scalar multiple of a ``contractible morphism''.

\begin{defn}\label{def: contractible morphism}
A nonzero morphism $f:X\to Y$ in a topological $K$-category $\cC$ will be called \emph{contractible} if there is a continuous path $Z_t$, $t\in[0,1]$, in the object space of $\cC$ so that $Z_0=X$, $Z_1=Y$ and continuous families of morphisms $f_t:X\to Z_t$, $g_t:Z_t\to Y$ so that $f_0=id_X$, $g_1=id_Y$ and $g_t\circ f_t=f$ for all $t$. In particular, $f_1=f$.
\end{defn}

We observe that every faithful (no nonzero morphism goes to zero) continuous functor $\cC\to\cD$ sends contractible morphisms to contractible morphisms.

\begin{prop}\label{prop: morphisms in S,M are contractible}
Every nonzero morphism in $\cS^1$ and $\cM_0$ is a scalar multiple of a uniquely determined contractible morphism.
\end{prop}

\begin{proof}
Any nonzero morphism in $\cS^1$ or $\cM_0$ can be lifted up to a nonzero morphism $(a,x,y)$ in its universal covering $\cR_\pi$ or $\cD_\pi$ which is unique determined up to deck transformations. This is $a$ times the basic morphism $(1,x,y)$ which is contractible since
\[
	(1,x,y)=(1,z_t,y)(1,x,z_t)
\]
where $z_t=ty+(1-t)x$.
\end{proof}

\begin{defn}\label{def: support of a morphism}
The \emph{support} of any nonzero morphism $f:X\to Y$ in a topological $K$-category $\cC$ is defined to be the space of all $Z\in\Ob(\cC)$ so that $f$ factors through $Z$.
\end{defn}

\begin{prop}
The categories $\cS^1$ and $\cM_0$ have the property that every nonzero morphism has contractible support.
\end{prop}

\begin{proof}
This is immediate: Any morphism $f$ in $\cS^1$ lifts to a morphism $\tilde f:x\to y$ in $\RR$ where $x\le y<x+\pi$. The support of $f$ is the image in $\cS^1$ of the closed interval $[x,y]\subseteq\RR$. This is contractible since $|y-x|<\pi$. Similarly a nonzero morphism $f$ in $\cM_0$ lifts to a morphism $\tilde f:(x_1,x_2)\to (y_1,y_2)$ and the support of $f$ is the image of $[x_1,y_1]\times [x_2,y_2]$ which is a contractible subset of the support of $x=(x_1,x_2)$. (See Figure \ref{fig2}.)
\end{proof}


%
%

\setcounter{section}{1}
\section{Equivalence coverings}

We will define $n$-fold equivalence coverings of a topological $K$-category $\cB$ and show that they are classified by their holonomy which is a homomorphism from the fundamental group of the object space of $\cB$ to the group of automorphism of $\cC_n$. We give the definition of $\cC_n$ and derive the basic properties of $\Aut(\cC_n)$. We assume that all objects of $\cB$ are nonzero.

\subsection{Basic topology, projective bundles}

To do covering theory we assume that the object space of $\cB$ is \emph{locally simply connected} which means it has a basis for its topology of simply connected open sets. We recall that a \emph{basis} for a topology on $X$ is a collection of open subsets $U_\alpha$ called \emph{basic open sets} satisfying the following.
\begin{enumerate}
\item $X$ is the union of all $U_\alpha$.
\item The intersection of any two, say $U_\alpha\cap U_\beta$ is a union of other basic open sets $U_\gamma$. 
\item A subset $V$ of $X$ is open if and only if it is the union of basic open sets $U_\alpha$.
\end{enumerate}

We all know that $n$-fold covering spaces $p:\tilde X\to X$ of a connected and locally simply connected Hausdorff space $X$ are classified by homomorphisms
\[
	\sigma:\pi_1(X,x_0)\to S_n
\]
for any \(x_0\in X\) where $S_n$ is the symmetric group on $n$ ``letters'' which are taken to be the elements of $[n]=\{1,2,\cdots,n\}$. We call $\sigma$ the \emph{holonomy} of the covering. 

The categorical version of the $n$-point set $[n]$ is the finite category $\cC_n$ discussed below. In good cases, an ``equivalence covering'' of $\cB$ will be classified by a homomorphism
\[
	\sigma:\pi_1(\Ob(\cB),X_0)\to \Aut(\cC_n)
\]
where homomorphisms $\sigma,\sigma'$ give the same equivalence covering if and only if they differ by \emph{conjugation}, i.e., there exists an element $\tau\in\Aut(\cC_n)$ so that $\sigma'(x)=\tau \sigma(x)\tau^{-1}$ for all $x\in \pi_1(\Ob(\cB),X_0)$.

By Proposition \ref{prop: Aut Cn is projective} below, $\Aut(\cC_n)$ is a subgroup of $PGL_n(K)$, the quotient of $GL_n(K)$ modulo its center which is the group $K^\ast I_n$ of nonzero scalar multiples of the identity matrix $I_n$. (In particular, $\Aut(\cC_1)$ is trivial, being a subgroup of $PGL_1(K)=1$.) So, such a homomorphism $\sigma$ also classifies a projective bundle over $\Ob(\cB)$. We recall the definition.

\begin{defn}\label{def: vector bundle}
Given a topological space $X$ with basic open sets $\{U_\alpha\}$, we define a \emph{system of projective transition matrices} for $X$ to be a family of ``projective matrices'' $g_{\beta\alpha}\in PGL_n(K)$ for all pairs of basic open sets $U_\alpha\subset U_\beta$ having the property that
\begin{equation}\label{eq: compatibility of transition matrices}
	g_{\gamma\beta}g_{\beta\alpha}=g_{\gamma\alpha}
\end{equation}
whenever $U_\alpha\subset U_\beta\subset U_\gamma$. This implies, in particular, that $g_{\alpha\alpha}$ is the identity in $PGL_n(K)$. Two systems of projective matrices $\{g_{\beta\alpha}\}$ and $\{g_{\beta\alpha}'\}$ are \emph{equivalent} if there exist elements $h_\alpha\in PGL_n(K)$ so that $g_{\beta\alpha}'=h_\beta g_{\beta\alpha}h_\alpha^{-1}$ for all $U_\alpha\subset U_\beta$. If $h=h_\alpha$ is independent of $\alpha$ we say that $\{g_{\beta\alpha}'\}$ is \emph{conjugate} to $\{g_{\beta\alpha}\}$ by $h$.

If there exist liftings $\tilde g_{\beta\alpha}$ of $g_{\beta\alpha}$ to $GL_n(K)$ so that $\tilde g_{\gamma\beta}\tilde g_{\beta\alpha}=\tilde g_{\gamma\alpha}
$ whenever $U_\alpha\subset U_\beta\subset U_\gamma$, we say the system $\{g_{\beta\alpha}\}$ is \emph{linearizable}. A system of projective matrices $g_{\beta\alpha}\in PGL_n(K)$ for $X$ determines a (flat) \emph{projective bundle} over $X$. If the system is linearizable, this bundle will be the projectivization of a flat vector bundle $E$ with projection $p:E\to X$ given by the construction:
\[
	E=\coprod_{\alpha} U_\alpha\times K^n/\sim
\]
where we identify $(x,v)\in U_\alpha\times K^n$ with $(x,\tilde g_{\beta\alpha}(x)(v))\in U_\beta\times K^n$ whenever $x\in U_\alpha\subset U_\beta$ with projection $p:E\to X$ given by $p(x,v)=x$.
\end{defn}

An alternate description of such a vector bundle $p:E\to X$ is a family of vector spaces $E_x:=p^{-1}(x)$ parametrized by $x\in X$ together with isomorphisms $h_\alpha:E_x\cong K^n$ for all basic open neighborhoods $U_\alpha$ of $x$ so that
\[
	h_\beta=\tilde g_{\beta\alpha}h_\alpha:E_x\cong K^n
\]
whenever $x\in U_\alpha\subset U_\beta$. Then we give $E$ a topology so that $h_\alpha$ gives a homeomorphism $\pi^{-1}(U_\alpha)\cong U_\alpha\times K^n$.

\begin{eg}\label{eg: standard system}
Suppose that $X$ is a connected space with a basis of simply connected open sets $U_\alpha$. Let $\sigma$ be a homomorphism $\pi_1(X,x_0)\to PGL_n(K)$. Then a system of projective transition matrices for $X$ can be given as follows.
\begin{enumerate}
\item For each $\alpha$, choose a point $x_\alpha\in U_\alpha$ and a path $\lambda_\alpha$ from $x_0$ to $x_\alpha$.
\item For any $U_\alpha\subset U_\beta$, let $g_{\beta\alpha}=\sigma(\lambda_\beta \gamma \lambda_\alpha^{-1})$ where $\gamma$ is any path from $x_\beta$ to $x_\alpha$ in $U_\beta$.
\end{enumerate}
We call this a \emph{standard system} of (projective) transition matrices given by $\sigma$.
\end{eg}

\subsection{The category $\cC_n$} An $n$-fold equivalence covering category of a topological $K$-category $\cB$ will be a topological $K$-category $\widetilde\cB$ which is locally isomorphic to $\cB\otimes \cC_n$ (Def. \ref{defn: B otimes Cn}) where $\cC_n$ is a finite category with $n$ isomorphic objects.




\begin{defn}\label{def: Cn}
For every $n\ge1$ let $\cC_n$ be the $K$-category with object set $[n]:=\{1,2,\cdots,n\}$ and $\cC_n(j,i)=Kx_{ij}$. I.e., every hom set is one dimensional over $K$ and is generated by 
\[
x_{ij}:j\to i. 
\]
We define composition $K$-bilinearly by 
\begin{equation}\label{eq: basic morphisms}
 x_{ij}x_{jk}=x_{ik}\quad \forall i,j,k.
\end{equation}
In particular, $x_{ii}=id_i$ and $x_{ji}=x_{ij}^{-1}$. Any choice of basis elements, i.e., any system of nonzero morphisms $x_{ij}\in \cC_n(j,i)$ satisfying \eqref{eq: basic morphisms} will be called a \emph{multiplicative basis} for $\cC_n$.
\end{defn}

We need to find all multiplicative bases for $\cC_n$. This will allow us to determine all automorphisms of $\cC_n$. Note that elements of $\cC_n(j,i)=K x_{ij}$ are all scalar multiples of $x_{ij}$. So, any other generator of $\cC_n(j,i)$ has the form $x_{ij}'=a_{ij}x_{ij}$ where $a_{ij}\in K^\ast$.

\begin{prop}\label{prop: comparison of multiplicative bases}
The following conditions on scalars $a_{ij}\neq0$ are equivalent.
\begin{enumerate}
\item The elements $x_{ij}'=a_{ij}x_{ij}$ form a multiplicative basis, i.e., satisfy \eqref{eq: basic morphisms}.
\item $a_{ij}a_{jk}=a_{ik}$ for all $i,j,k$
\item There are $c_i\in K^\ast$ so that $a_{ij}=c_i/c_j$.
\end{enumerate}
Furthermore, the scalars $c_i$ in $(3)$ are uniquely determined up to multiplication by a nonzero scalar, i.e., the $n$-tuple $(c_1,\cdots,c_n)$ forms a well-defined element of projective space $KP^{n-1}$.
\end{prop}

\begin{proof}
It is clear that (1),(2) are equivalent and that (3) implies (2). To see that (2) implies (3), let $c_j=a_{j1}$.
\end{proof}

Any collection of scalars $a_{ij}\in K^\ast$ satisfying Condition (2) in Proposition \ref{prop: comparison of multiplicative bases} will be called a \emph{multiplicative system} (of scalars). Since an automorphism of $\cC_n$ permutes the objects and takes one multiplicative basis to another we get the following.

\begin{cor}\label{cor: formula for automorphisms of Cn}
A $K$-automorphism $\sigma$ of $\cC_n$ is given on objects and morphisms as follows.
\begin{enumerate}
\item On objects, $\sigma$ is a permutation of $n$: $\sigma \in S_n$.
\item On morphisms we have 
\[
	\sigma(x_{ij})=a_{ij}x_{\sigma(i)\sigma(j)}
\]
where $a_{ij}\in K^\ast$ is a multiplicative system of scalars.
\end{enumerate}
Conversely, any permutation of $n$ and any multiplicative system $(a_{ij})$ defines an automorphism of $\cC_n$.\qed
\end{cor}

We have the following straightforward generalization.

\begin{prop}\label{prop: transition coef for any tau}
A $K$-linear functor $\tau:\cC_n\to \cC_m$ is given on objects and morphisms as follows.
\begin{enumerate}
\item On objects $\tau$ is a mapping $[n]\to [m]$.
\item On morphisms $\tau$ is given by
\[
	\tau(x_{ij})=b_{ij}y_{\tau(i)\tau(j)}
\]
where $x_{ij}$, $y_{pq}$ are multiplicative bases for $\cC_n,\cC_m$ and $b_{ij}$ is a multiplicative system of scalars. In particular $\tau$ is a faithful functor.
\end{enumerate}
Conversely, any mapping $[n]\to[m]$ and any multiplicative system of scalars $b_{ij}$ determines a unique $K$-linear functor $\tau:\cC_n\to\cC_m$.\qed
\end{prop}

\begin{defn}\label{def: transition coefficients}
The scalars $(a_{ij})$ and $(b_{ij})$ which occur in Corollary \ref{cor: formula for automorphisms of Cn} and Proposition \ref{prop: transition coef for any tau} above will be called the \emph{transition coefficients} of $\sigma,\tau$ with respect to $(x_{ij})$ for $\sigma$ and with respect to $(x_{ij})$ and $ (y_{pq})$ for $\tau$. Recall that these are multiplicative systems of scalars, i.e., satisfy the equivalent conditions listed in Proposition \ref{prop: comparison of multiplicative bases}.
\end{defn}


Recall that the \emph{monomial matrix group} $M_n(K^\ast)$ is the subgroup of $GL_n(K)$ of all matrices which are products $PD$ where $P$ is a permutation matrix and $D$ is a diagonal matrix with diagonal entries in $K^\ast$.

\begin{prop}\label{prop: Aut Cn is projective}
There is an epimorphism ${M_n(K^\ast)\to \Aut(\cC_n)}$ with kernel the center of $M_n(K^\ast)$ which is ${K^\ast I_n}$. Thus:
\[
	\Aut(\cC_n)\cong M_n(K^\ast)/K^\ast I_n \subset PGL_n(K).
\]
\end{prop}

\begin{proof}
The homomorphism $M_n(K^\ast)\to \Aut(\cC_n)$ is given by sending $PD$ to the automorphism $\sigma$ whose underlying permutation is given by $P$ in the sense that $Pe_i=e_{\sigma(i)}$ and whose transition coefficients are $a_{ij}=d_i/d_j$ where $d_i$ are the entries of $D$. For $\sigma$ to be the identity, $P$ must be the identity permutation and $d_i=d_j$ for all $i,j$, i.e., $PD\in K^\ast I_n$. An easy calculation shows that $PD\mapsto \sigma$ is a homomorphism.
\end{proof}

\begin{rem}\label{rem: PD: id cong sigma}
The monomial matrix $PD$ gives a natural isomorphism $id\cong \sigma$ which, on object $i$, is $d_ix_{\sigma(i)i}:i\to \sigma(i)$. This is natural, i.e., the following diagram commutes, since $d_ja_{ij}=d_i$.
\[
\xymatrixrowsep{20pt}\xymatrixcolsep{30pt}
\xymatrix{
j\ar[d]_{x_{ij}}\ar[r]^(.4){d_j} &
	\sigma(j)\ar[d]^{a_{ij}x_{\sigma(i)\sigma(j)}}\\
i \ar[r]^(.4){d_i}& 
	\sigma(i)
	}
\] 
\end{rem}

\begin{defn}\label{defn: B otimes Cn}For a topological $K$-category $\cB$ without zero objects, we define $\cB\otimes_K\cC_n$ to be the topological $K$-category with object space $\Ob(\cB)\times [n]$, the disjoint union of $n$ copies of $\Ob(\cB)$, and morphisms
\[
	(\cB\otimes \cC_n)(X\otimes i,Y\otimes j)=\cB(X,Y)\otimes \cC_n(i,j)=\cB(X,Y)\times x_{ji}.
\]
Since $\cC_n(i,j)=Kx_{ji}$, morphism $X\otimes i\to Y\otimes j$ can be written uniquely as $f\otimes x_{ji}$ where $f\in\cB(X,Y)$ and we give the morphism space of $\cB\otimes \cC_n$ the topology of $\Mor(\cB)\times\{x_{ij}\}$ where $\{x_{ij}\}$ denotes $\{x_{ij}\,|\, i,j\in \Ob(\cC)\}$. This is the disjoint union of $n^2$ copies of $\Mor(\cB)$. 
\end{defn}

For each $i\in[n]$, $\cB\otimes i\subseteq \cB\otimes \cC_n$ is a full subcategory isomorphic to $\cB$ and $id_X\otimes x_{ji}$ gives a continuous natural isomorphism $X\otimes i\cong X\otimes j$ for any $i,j\in[n]$. Given any $K$-linear functor $\tau:\cC_n\to \cC_m$ and any topological $K$-category $\cB$ we get an induced continuous faithful linear functor 
\[
id\otimes \tau:\cB\otimes \cC_n\to\cB\otimes\cC_m.
\]


\subsection{Equivalence coverings}

An equivalence covering of $\cB$ will be a category $\widetilde\cB$ which looks locally like $\cB\otimes \cC_n$ with transition maps of the form $id\otimes \sigma$ for $\sigma\in \Aut(\cC_n)$.

In the sequel, for $U,V\in \Ob(\cB)$, we use the notation $\cB(U,V)$ for the space of all morphisms $x\to y$ in $\cB$ where $x\in U$ and $y\in V$. Given continuous mappings $h_1:U\to U', h_2:V\to V'$ where $U',V'\subset\Ob(\cD)$, a mapping of morphism sets $F:\cB(U,V)\to \cD(U',V')$ will be called \emph{fiberwise linear} over $h_1,h_2$ if $F$ sends $\cB(x,y)$ to $ \cD(h_1(x),h_2(y))$ by a $K$-linear map for all $(x,y)\in U\times V$. For example, the morphism map of any linear functor is fiberwise linear over its object map.

\begin{defn}\label{def: equivalence covering}
An \emph{$n$-fold equivalence covering} of $\cB$ is a topological $K$-category $\widetilde \cB$ with maps $p,F_{\beta\alpha},g_{\beta\alpha}$ satisfying the following for $\{U_\alpha\}$ a basis of open sets for $X=\Ob(\cB)$.
\begin{enumerate}
\item $p:\widetilde X=\Ob(\widetilde\cB)\to X=\Ob(\cB)$ is an $n$-fold covering map. For each $U_\alpha$, let $\widetilde U_\alpha=p^{-1}(U_\alpha)$ and assume there is $h_\alpha:\widetilde U_\alpha\to U_\alpha\times[n]$ a homeomorphism over $U_\alpha$, i.e., $h_\alpha$ commutes with projection to $U_\alpha$.
\item For each $U_\alpha\subset U_\beta$,
\[
	F_{\beta\alpha}:\widetilde\cB(\widetilde U_\alpha,\widetilde U_\beta)\to (\cB\otimes \cC_n)(U_\alpha\times[n],U_\beta\times [n])=\cB(U_\alpha,U_\beta)\times\{x_{ji}\}
\] is a continuous fiberwise isomorphism over $h_\alpha,h_\beta$, i.e., $F_{\beta\alpha}$ gives a $K$-linear isomorphism $F_{\beta\alpha}:\widetilde\cB(\widetilde X,\widetilde Y)\cong \cB(X,Y)\times x_{ji}$ if $h_\alpha(\widetilde X)=X\otimes i$ and $h_\beta(\widetilde Y)=Y\otimes j$, which also satisfies the following.
	\begin{enumerate}
	\item $F_{\alpha\alpha}(id_x)=id_y\times x_{ii}$ if $h_\alpha(x)=(y,i)$. 
	\item $F_{\gamma\beta}(g)F_{\beta\alpha}(f)=F_{\gamma\alpha}(gf)$ if $f\in \widetilde\cB(\widetilde U_\alpha,\widetilde U_\beta)$ and $g\in \widetilde\cB(\widetilde U_\beta,\widetilde U_\gamma)$ are composable.
	\end{enumerate}
\item There are automorphisms $g_{\beta\alpha}\in \Aut(\cC_n)$ for all $U_\alpha\subset U_\beta$ satisfying the compatibility condition \eqref{eq: compatibility of transition matrices} whose underlying permutation matrix $\pi_{\beta\alpha}$ gives 
\[
h_\beta h_\alpha^{-1}=(inc:U_\alpha\hookrightarrow U_\beta)\times \pi_{\beta\alpha}:U_\alpha\times[n]\to U_\beta\times[n]
\]
and so that the following diagram commutes whenever $U_\alpha\subset U_{\alpha'}$ and $U_\beta\subset U_{\beta'}$.
\[
\xymatrix{
\widetilde\cB(\widetilde U_\alpha,\widetilde U_\beta) \ar[d]_\subset\ar[r]^(.4){F_{\beta\alpha}} &
	\cB( U_\alpha, U_\beta)\times \{x_{ji}\} \ar[d]^{inc\otimes g_{\beta'\beta}g_{\alpha'\alpha}^{-1}} \\
\widetilde\cB(\widetilde U_{\alpha'},\widetilde U_{\beta'}) \ar[r]^(.4){F_{\beta'\alpha'}} &
	\cB( U_{\alpha'}, U_{\beta'})\times \{x_{ji}\} 
	}
\]
where, if the automorphism ${g_{\beta'\beta}g_{\alpha'\alpha}^{-1}}$ of $\cC_n$ sends $x_{ji}$ to the scalar multiple $cx_{\pi(j)\pi'(i)}$ of $x_{\pi(j)\pi'(i)}$, then $inc\otimes {g_{\beta'\beta}g_{\alpha'\alpha}^{-1}}$ sends $(f,x_{ji})$ to $(cf,x_{\pi(j)\pi'(i)})$.
\end{enumerate}
\end{defn}

\begin{thm}\label{thm: classification of gab}
An $n$-fold equivalence covering of $\cB$ is completely determined up to continuous isomorphism by the family of compatible automorphisms $(g_{\beta\alpha})$ in $\Aut(\cC_n)$. When $\Ob(\cB)$ is connected and locally 1-connected, this system gives a homomorphism
\[
	\sigma:\pi_1(\Ob(\cB),X_0)\to \Aut(\cC_n)
\]
well-defined up to conjugation for any object $X_0$. Conversely, any such homomorphism gives an $n$-fold covering $\widetilde \cB_\sigma$ of $\cB$.
\end{thm}

\begin{proof}
It is standard covering theory that, for any discrete group $G$ and system of elements $g_{\beta\alpha}\in G$ satisfying the compatibility condition $g_{\gamma\beta}g_{\beta\alpha}=g_{\gamma\alpha}$ gives a covering space of $X$ and a homomorphism $\pi_1(X,x_0)\to G$ for any $x_0\in X$. For $X$ connected and locally 1-connected any such homomorphism $\sigma$ is realized by taking $\{g_{\beta\alpha}\}$ to be the standard system given by $\sigma$ as in Example \ref{eg: standard system}. Thus the only thing to show is that the other structure elements of an $n$-fold equivalence cover of $\cB$ are determined up to isomorphism by $\{g_{\beta\alpha}\}$.


The permutation part of $g_{\beta\alpha}$, denoted $\pi_{\beta\alpha}\in S_n$, determines the covering space $\widetilde X$ and covering map $p:\widetilde X\to X=\Ob(\cB)$ up to isomorphism. We can also choose the local trivializations $h_\alpha:\widetilde U_\alpha\cong U_\alpha\times[n]$ arbitrarily. Then (1) is satisfied.

The morphism set $\Mor(\widetilde\cB)$ is given by pasting together the sets $\cB(U_\alpha,U_\beta)\times\{x_{ji}\}$:
\[
	\Mor(\widetilde\cB)=\coprod \cB(U_\alpha,U_\beta)\times\{x_{ji}\}/\sim
\]
and giving this the quotient topology where the isomorphisms $id\otimes {g_{\beta'\beta}g_{\alpha'\alpha}^{-1}}$ are used to paste together the different copies of each morphism set. The compatibility condition \eqref{eq: compatibility of transition matrices} insures that these identifications are consistent. This gives (3).

It remains to prove (2)(b) which implies (2)(a). But, composable morphisms $f\in \widetilde\cB(U_\alpha,U_\beta), g\in \widetilde\cB(U_\beta,U_\gamma)$ are given by $f=(f_0\otimes x_{ji})$, $g=(g_0\otimes x_{kj})$ with composition $gf=(g_0f_0\otimes x_{ki})$.
\end{proof}

\begin{eg}\label{eg: equivalence coverings of KX}
Let $X$ be a connected and locally simply connected topological space and let $KX$ be the trivial $K$-category of $X$ (Example \ref{eg: trivial example}). For any homomorphism $\sigma:\pi_1(X,x_0)\to \Aut(\cC_n)$ consider the corresponding equivalence covering $\widetilde{KX}_\sigma$ of $KX$. 

Let $\cK_\sigma X$ be the topological subcategory of $\widetilde{KX}_\sigma$ containing all of the objects but only those morphisms lying over the same point in $X$. The fiber $\cK_\sigma X_{x_0}$ of $\cK_\sigma X$ over $x_0\in X$ is isomorphic to $\cC_n$. For any loop $\gamma$ in $X$ starting and ending at $x_0$, unique lifting of paths in $X$ to the covering spaces of objects and morphisms of $\cK_\sigma X$ gives an automorphism of $\cK_\sigma X_{x_0}$. This gives an automorphism $\tau[\gamma]$ of $\cK_\sigma X_{x_0}$ which induces a well-defined homomorphism 
\[
	\tau:\pi_1(X,x_0)\to \Aut(\cK_\sigma X_{x_0}).
\]
Conjugating with an isomorphism $\cK_\sigma X_{x_0}\cong\cC_n$ gives a homomorphism $\pi_1(X,x_0)\to \Aut(\cC_n)$ well-defined up to conjugation. By construction, this must be conjugate to $\sigma$.
\end{eg}

\subsection{Classification of equivalence coverings}

We will show that, under certain conditions, equivalence coverings of a topological $K$-category $\cB$ are classified by their holonomy up to conjugation. The statement is that $\widetilde \cB_\sigma$ and $\widetilde \cB_{\sigma'}$ are continuously isomorphic ``over $\cB$'' if and only if $\sigma,\sigma'$ are conjugate. We need the following definition to overcome the problem that there is, in general, no continuous functor $\widetilde \cB_\sigma\to\cB$.

\begin{defn}\label{def: equivalence of equivalence coverings}
Let $\widetilde \cB_\sigma$, $\widetilde \cB_{\sigma'}$ be two equivalence coverings of the same topological $K$-category $\cB$. A continuous linear functor $F:\widetilde \cB_\sigma\to \widetilde \cB_{\sigma'}$ will be called a \emph{continuous equivalence over $\cB$} if the object map of $F$ is a map of covering spaces over $\cO b\cB$, i.e., the following diagram commutes.
\[
\xymatrixrowsep{15pt}\xymatrixcolsep{10pt}
\xymatrix{
\cO b\widetilde\cB_\sigma\ar[rr]^{F}\ar[dr]&  &
	\cO b\widetilde\cB_{\sigma'}\ar[dl]\\
& \cO b\cB
	}
\]
If the functor $F$ is also an isomorphism, it will be called a \emph{continuous isomorphism over $\cB$}.
\end{defn}

\begin{lem}\label{lem: morphism between equiv covers}
Suppose $\cO b\cB$ is connected and locally simply connected and let $\widetilde\cB_\sigma$ and $\widetilde\cB_{\sigma'}$ be $n$ and $m$-fold equivalence coverings with holonomies $\sigma$ and $\sigma'$ respectively. Let $\tau:\cC_n\to \cC_m$ be a linear functor with the property that $\sigma'(\gamma)\circ\tau=\tau\circ\sigma(\gamma):\cC_n\to \cC_m$ for any $\gamma\in\pi_1(\cO b\cB,X_0)$. Then there is a continuous equivalence
\[
	F_\tau:\widetilde \cB_\sigma\to \widetilde\cB_{\sigma'}
\]over $\cB$ given on $\cB(U_\alpha,U_\beta)\otimes \cC_n$, using standard coordinates, by $id\otimes \tau$.  Thus, $F_\tau F_{\tau'}=F_{\tau\tau'}$. In particular, when $\tau$ is an isomorphism, $F_\tau$ is an isomorphism with inverse $F_{\tau^{-1}}$.
\end{lem}

\begin{proof}
It suffices to show that the mappings $id\otimes\tau$ are compatible with the standard identifications defining $\widetilde\cB_\sigma$ and $\widetilde\cB_{\sigma'}$, i.e., that the following diagram commutes.
\[
\xymatrix{
\cB(U_\alpha,U_\beta) \otimes \cC_n \ar[d]_{inc\,\otimes\, g_{\beta'\beta}g_{\alpha'\alpha}^{-1}}\ar[r]^{id\otimes\tau} &
	\cB(U_\alpha,U_\beta) \otimes \cC_m\ar[d]^{inc\,\otimes\, g_{\beta'\beta}'g_{\alpha'\alpha}'^{-1}} \\
\cB(U_{\alpha'},U_{\beta'}) \otimes \cC_n \ar[r]^{id\otimes\tau} &
	\cB(U_{\alpha'},U_{\beta'}) \otimes \cC_m
	}
\]
Assuming $\{g_{\beta\alpha}\}, \{g_{\beta\alpha}'\}$ are  standard (Example \ref{eg: standard system}), this is the following equation:
\[
	\tau \sigma(\lambda_{\beta'}\gamma \lambda_\beta^{-1}) \sigma(\lambda_{\alpha}\gamma^{-1} \lambda_{\alpha'}^{-1}) =\sigma'(\lambda_{\beta'}\gamma \lambda_\beta^{-1}) \sigma'(\lambda_{\alpha}\gamma^{-1} \lambda_{\alpha'}^{-1}) \tau
\]
which follows from the assumption that $\sigma'(\gamma)=\tau\sigma(\gamma)\tau^{-1}$ for any loop $\gamma$ in $\cO b\cB$.

$F_\tau F_{\tau'}=F_{\tau\tau'}$ since, in standard coordinates, it is given by $(id\otimes \tau)(id\otimes\tau')=id\otimes \tau\tau'$.
\end{proof}

\begin{thm}[Thm A]\label{thmA}
Suppose that $\cB$ is a topological $K$-category so that every object is nonzero and every automorphism of every object is a scalar multiple of the identity. Suppose also that $\cO b\cB$ is connected and locally simply connected. Then, for any homomorphism
\[
	\sigma:\pi_1(\Ob(\cB),X_0)\to \Aut(\cC_n)
\]
there exists an equivalence covering $\widetilde\cB_\sigma$ of $\cB$ with holonomy $\sigma$. Furthermore, any two equivalence coverings of $\cB$ are continuously isomorphic over $\cB$ if and only if their holonomies are conjugate.
\end{thm}

\begin{proof} The existence of an equivalence covering $\widetilde\cB_\sigma$ of $\cB$ with any holonomy $\sigma$ is a special case of Theorem \ref{thm: classification of gab}. By Lemma \ref{lem: morphism between equiv covers}, two equivalence coverings of $\cB$ with conjugate holonomies are continuously isomorphic over $\cB$. 

Conversely, suppose $\widetilde\cB_\sigma$ and $\widetilde\cB_\tau$ are continuously isomorphic over $\cB$. Let $\widetilde\cB_\sigma$ be the $n$-fold equivalence covering of $\cB$ with holonomy $\sigma$. The conditions on $\cB$ imply that $\widetilde \cB_\sigma$ contains a topological $K$-subcategory $\widetilde\cB_\sigma^\ast$ isomorphic to the equivalence covering $\widetilde{KX}_\sigma$ from Example \ref{eg: equivalence coverings of KX} with $X=\cO b\cB$. This subcategory is characterized by the property that any nonzero morphism $x\to y$ in $\widetilde\cB_\sigma$ lies in $\widetilde\cB_\sigma^\ast$ if and only if $x,y\in \cO b\widetilde \cB_\sigma$ lie over the same point in $X=\cO b\cB$. Therefore, the isomorphism $\widetilde\cB_\sigma\cong \widetilde\cB_\tau$ restricts to an isomorphism $\widetilde\cB_\sigma^\ast\cong \widetilde\cB_\tau^\ast$ giving an isomorphism $\widetilde{KX}_\sigma\cong \widetilde{KX}_\tau$. The formula for the holonomy of $\widetilde{KX}_\sigma$ given in Example \ref{eg: equivalence coverings of KX} then implies that $\sigma$ and $\tau$ are conjugate as claimed.
\end{proof}

\begin{thm}\label{thm: all continuous F are F-tau}
Let $\cB$ be as in Theorem \ref{thmA} above and suppose that every nonzero morphism in $\cB$ is a scalar multiple of a contractible morphism. Then any continuous equivalence $F:\widetilde\cB_\sigma\to \widetilde\cB_{\sigma'}$ is equal to $F_\tau$ for some discrete equivalence $\tau:\cC_n\to \cC_m$ satisfying $\sigma'(\gamma)\circ\tau=\tau\circ\sigma(\gamma)$ for all loops $\gamma$ in $\Ob\cB$.
\end{thm}

\begin{proof} As in the proof of Theorem \ref{thmA}, let $\widetilde\cB_\sigma^\ast$ and $\widetilde\cB_{\sigma'}^\ast$ denote the subcategories of $\widetilde\cB_\sigma$ and $\widetilde\cB_{\sigma'}$ isomorphic to the equivalence coverings $\widetilde{KX}_\sigma$ and $\widetilde{KX}_{\sigma'}$, respectively, where $X=\Ob \cB$. 

The equivalence $F$ induces an equivalence $F':\widetilde{KX}_\sigma\to\widetilde{KX}_{\sigma'}$ which, over the object $X_0\in\Ob\cB$, restricts to an equivalence $\tau:\cC_n\to \cC_m$. By the unique path lifting property for covering space, the morphism $F'$ must be equal to $F_\tau'$, the equivalence induced by $\tau$. The assumption that every nonzero morphism in $\cB$ is a scalar multiple of a contractible morphism implies that every nonzero morphism in every equivalence covering of $\cB$ is homotopic through nonzero morphisms to an isomorphism. This implies that $F$ is uniquely determined by its restriction to $\widetilde\cB_\sigma^\ast$. Therefore, $F=F_\tau$ as claimed.
\end{proof}

In the special cases $\cB=\cS^1$ and $\cM_0$, the fundamental group is $\ZZ$. So, the holonomy morphism $\pi_1\Ob\cB=\ZZ\to \Aut(C_n)$ is determined by the image of the generator which is a single element $\sigma\in \Aut(\cC_n)$. We denote the corresponding equivalence coverings of $\cS^1,\cM_0$ by $\widetilde \cS^1_\sigma,\widetilde \cM_\sigma$. By Proposition \ref{prop: morphisms in S,M are contractible} both categories satisfy the conditions of Theorem \ref{thm: all continuous F are F-tau}. Therefore, we have the following.

\begin{cor}[Thm B]\label{thmB}
Continuous equivalences $F:\widetilde\cS^1_\sigma\to \widetilde\cS^1_{\sigma'}$ and $\widetilde\cM_\sigma\to \widetilde\cM_{\sigma'}$ are equal to $F_\tau$ where $\tau$ is a linear equivalence $\cC_n\to \cC_m$ so that $\tau\circ \sigma=\sigma'\circ \tau$.\qed
\end{cor}

\begin{rem}\label{rem: tau is not unique}
Similar to the observation in Example \ref{eg: equivalence coverings of KX}, we note that $\tau$ is not uniquely determined in Theorem \ref{thm: all continuous F are F-tau} and Corollary \ref{thmB}. For example, the functor $F$ determines a linear equivalence from a subcategory of $\widetilde\cS^1_\sigma$ isomorphic to $\cC_n$ to a subcategory of $\widetilde\cS^1_{\sigma'}$ isomorphic to $\cC_m$. The choices of these isomorphisms determine $\tau$.
\end{rem}


\section{Skew-continuous natural transformations}

For any two autoequivalences $\sigma,\tau$ of $\cC_n$ there is a natural isomorphism $\varphi:\sigma\to \tau$. This natural isomorphism will be unique up to one scalar which we call the ``rescaling factor'' and will be ``continuous'' if its ``$\sigma$-continuity factor'' is $1$. The natural isomorphism $\varphi$ is ``skew-continuous'' if its $\sigma$-continuity factor is $-1$. 

We give the definitions and a brief explanation why the axioms of a triangulated category require $\varphi:\sigma\to\tau$ to be skew-continuous.

\subsection{Definition of skew-continuity}\label{sec: def of skew-continuity}

We will show that a natural isomorphism $\varphi:\sigma\to\tau$ always exists and is uniquely determined up to a \emph{rescaling factor} $r\in K^\ast$. For example, the only natural isomorphisms $\sigma\to \sigma$ are scalar multiples of the identity.

\begin{prop}\label{prop: natural isomorphisms are unique up to rescaling}
Given autoequivalences $\sigma,\tau$ of $\cC_n$ with transition coefficients $a_{ji},b_{ji}$, there is a unique $\overline c=[c_1,\cdots,c_n]\in KP^{n-1}$, with $c_i\in K^\ast$, so that 
\[
	c_ja_{ji}=b_{ji}c_i
\]
for all $i,j\in \cC_n$. Each lifting $(c_i)\in (K^\ast)^n$ of $\overline c$ gives a natural isomorphism $\varphi: \sigma\to \tau$ by 
\[
	\varphi_i=c_ix_{\tau(i)\sigma(i)}:\sigma(i)\to \tau(i).
\]
Conversely, all natural isomorphisms $\varphi:\sigma\to\tau$ are given in this way.
\end{prop}

\begin{proof}
Existence of $c_i$ so that $c_ja_{ji}=b_{ji}c_i$ is clear: Take $c_i=a_{1i}b_{i1}r$ for any $r\in K^\ast$. Then \[
b_{ji}c_i=a_{1i}b_{ji}b_{i1}r=a_{1j}a_{ji}b_{j1}r=c_ja_{ji}.
\]
Conversely, the equation $c_ja_{ji}=b_{ji}c_i$ implies, for $j=1$, that $c_i=c_1a_{1i}b_{1i}^{-1}=a_{1i}b_{i1}c_1$. So, these are all the choices. So, $\overline c\in KP^{n-1}$ is uniquely determined.

Any natural isomorphism $\varphi: \sigma\to \tau$ is, by definition, given by $\varphi_i=c_ix_{\tau(i)\sigma(i)}$ for $c_i\in K^\ast$ making the following diagram commute for all $i,j\in\cC_n$:
\[
\xymatrixrowsep{20pt}\xymatrixcolsep{40pt}
\xymatrix{
\sigma(i)\ar[d]_{a_{ji}x_{\sigma(j)\sigma(i)}}\ar[r]^{c_ix_{\tau(i)\sigma(i)}} &
	\tau(i)\ar[d]^{b_{ji}x_{\tau(j)\tau(i)}}\\
\sigma(j) \ar[r]^{c_jx_{\tau(j)\sigma(j)}}& 
	\tau(j)
	}
\]
In other words, $b_{ji}c_i=c_ja_{ji}$, i.e., $(c_i)$ is a representative of $\overline c$.
\end{proof}

If $\sigma,\tau$ are autoequivalences of $\cC_n$ which commute and $\varphi:\sigma\to \tau$ is a natural isomorphism then $\sigma(\varphi)$ and $\varphi(\sigma)$ are two natural isomorphisms
\[
	\sigma(\sigma)\to \sigma(\tau)=\tau(\sigma).
\]
By Proposition \ref{prop: natural isomorphisms are unique up to rescaling}, there is a unique scalar $s\in K^\ast$ so that 
\[
\sigma(\varphi)=s\varphi(\sigma).
\]
We call $s$ the \emph{$\sigma$-continuity factor} of $\varphi$. Note that any scalar multiple of $\varphi$ has the same $\sigma$-continuity factor. So, $s$ depends only on $\sigma$ and $\tau$ and not on $\varphi$. We use the notation:
\[
	s=\{\sigma,\tau\}
\]
We show in Corollary \ref{cor: anti-symmetry of s-t pairing} that $\{\tau,\sigma\}=\{\sigma,\tau\}^{-1}$.

\begin{defn}\label{def: anti-compatible}
Commuting autoequivalences $\sigma,\tau$ of $\cC_n$ will be called \emph{compatible}, resp. \emph{anti-compatible}, if $\{\sigma,\tau\}=1$, resp., $\{\sigma,\tau\}=-1$.
\end{defn}

\begin{defn}\label{def: skew-continuous}
A natural transformation $\varphi:\sigma\to\tau$ between commuting autoequivalences of $\cC_n$ is \emph{continuous} or \emph{skew-continuous} (with respect to $\sigma$) if its $\sigma$-continuity factor is $1$ or $-1$, respectively. In other words, for any $i\in\cC_n$, the square in the following diagram commutes or anticommutes, respectively.
\[
\xymatrix{
\sigma(i)\ar[d]_{\varphi_i} &
	\sigma^2(i)\ar[d]_{\sigma(\varphi_i)}\ar[r]^= &
	\sigma^2(i)\ar[d]^{\varphi_{\sigma(i)}}\\
\tau(i) & 
	\sigma\tau(i)\ar[r]^= &
	\tau\sigma(i)
	}
\]
Equivalently, using the scalars $c_i$ so that $\varphi_i=c_ix_{\tau(i)\sigma(i)}$, $\varphi$ is continuous if 
\[
	c_{\sigma(i)}=c_i a_{\tau(i)\sigma(i)}
\]
and $\varphi$ is skew-continuous if
\begin{equation}\label{eq: skew-commutativity condition}
	c_{\sigma(i)}=-c_i a_{\tau(i)\sigma(i)}.
\end{equation}
\end{defn}

The following summary of the above definitions and observations implies Lemma \ref{lemC1} in the introduction.

\begin{prop}\label{prop: summary of skew-continuity}
Let $\sigma,\tau$ be commuting autoequivalences of $\cC_n$. Then any natural isomorphism $\varphi:\sigma\to \tau$ has $\sigma$-continuity factor equal to $\{\sigma,\tau\}$. In particular, $\varphi$ is continuous, resp. skew-continuous if and only if $\sigma,\tau$ are compatible, resp. anti-compatible.\qed
\end{prop}


\subsection{Interpretation of skew-continuity}

Given an automorphism $\sigma$ of $\cC_n$ and an autoequivalence $\tau$ on $\cC_n$ which commutes with $\sigma$ we have, by Corollary \ref{thmB}, continuous linear functors:
\[
	F_\tau:\widetilde\cS^1_\sigma\to \widetilde\cS^1_\sigma,
	\quad F_\tau:\widetilde\cM_\sigma\to \widetilde\cM_\sigma.
\]
Since $\sigma$ obviously commutes with $\sigma$ we also have continuous automorphisms $F_\sigma:\widetilde\cS^1_\sigma\to \widetilde\cS^1_\sigma$ and $F_\sigma:\widetilde\cM_\sigma\to \widetilde\cM_\sigma$.

A natural transformation $\varphi:\sigma\to\tau$ induces a continuous natural transformation $F_\varphi:F_\sigma\to F_\tau$ if and only if $\varphi$ is ``continuous'', i.e., if $\{\sigma,\tau\}=1$, or, equivalently, $\sigma(\varphi_i)=\varphi_{\sigma(i)}$. More generally:

\begin{thm}
Let $\sigma,\sigma'$ be automorphisms of $\cC_n,\cC_m$ respectively. Let $\tau_1,\tau_2$ be $K$-linear functors $\cC_n\to \cC_m$ so that $\sigma'\tau_j=\tau_j\sigma$ for $j=1,2$ giving continous equivalences $F_{\tau_1},F_{\tau_2}:\widetilde\cS^1_{\sigma}\to \widetilde\cS^1_{\sigma'}$ and $F_{\tau_1},F_{\tau_2}:\widetilde\cM^1_{\sigma}\to \widetilde\cM^1_{\sigma'}$. Then, a natural transformation $\varphi:\tau_1\to\tau_2$ induces continuous maps $F_\varphi:\Ob(\widetilde\cS^1_{\sigma})\to \Mor(\widetilde\cS^1_{\sigma'})$ and $F_\varphi:\Ob(\widetilde\cM_{\sigma})\to \Mor(\widetilde\cM_{\sigma'})$ giving continuous natural transformations $F_{\tau_1}\to F_{\tau_2}$, if and only if $\sigma'(\varphi_i)=\varphi_{\sigma(i)}$ for every $i\in\cC_n$.
\end{thm}

\begin{proof} (for $\widetilde\cS^1$)
For each $i\in\cC_n$, $j\in\{1,2\}$ and $t\in [0,\pi)$ we have $F_{\tau_j}( t\otimes i)=t\otimes \tau_j(i)$ and
\[
	F_\varphi( t\otimes i)= id_{[0,\pi)}\otimes\varphi_i :F_{\tau_1}(t\otimes i)=t\otimes \tau_1(i) \xrightarrow{c_i} t\otimes\tau_2(i) =F_{\tau_2}(t\otimes i).
\]
In order for $F_\varphi$ to be continuous, we must have
\[
	F_\varphi\lim_{t\to \pi}(t\otimes i)=\lim_{t\to \pi}F_\varphi(t\otimes i).
\]
But, the left hand side is equal to
\[
	F_\varphi( 0\otimes\sigma(i))= id_{[0,\pi)}\otimes \varphi_{\sigma(i)}: 0\otimes \tau_1(\sigma(i)) \xrightarrow{c_{\sigma(i)}} 0\otimes \tau_2(\sigma(i)) 
\]
and the right hand side is equal to
\[
	 id_{[0,\pi)}\otimes \sigma'(\varphi_i):0\otimes\sigma'(\tau_1(i)) \xrightarrow{\sigma'(c_i)} 0\otimes\sigma'(\tau_2(i)).
\]
These are equal if and only if $\sigma'\tau_j=\tau_j\sigma$ for $j=1,2$ and $\sigma'(\varphi_i)=\varphi_{\sigma(i)}$.

For $\widetilde\cM$, the argument is very similar, with $t\otimes i$ replaces with $(x,y)\otimes i$ where $0\le x+y<2\pi$, $|x-y|<\pi$. For $F_\varphi$ to be continuous we must have
\[
	F_\varphi\lim_{x+y\to 2\pi}((x,y)\otimes i)=\lim_{x+y\to2 \pi}F_\varphi((x,y)\otimes i)
\]
\[
	F_\varphi((y-\pi,x-\pi)\otimes \sigma(i))=id\otimes \varphi_{\sigma(i)}=id\otimes \sigma'(\varphi_i).
\]
As in the case of $\widetilde\cS$, this holds if and only if $\sigma'\tau_j=\tau_j\sigma$ for $j=1,2$ and $\sigma'(\varphi_i)=\varphi_{\sigma(i)}$.
\end{proof}


We explain why the rotation axiom for a triangulated category forces us to use a skew-continuous natural transformation $\varphi:\sigma\to\tau$ in the definition of distinguished triangles in $add\,\widetilde\cM_\sigma$. We take a ``basic positive triangle'' which is the simplest example of a distinguished triangle in $add\,\widetilde\cM_\sigma$ where all terms are indecomposable (and thus lie in $\widetilde\cM_\sigma$):
\begin{equation}\label{eq: positive triangle}
	X\otimes i\xrightarrow f Y\otimes i \xrightarrow g Z\otimes i\xrightarrow \psi T (X\otimes i)=X\otimes\tau(i)
\end{equation}
Here $X=(x,y)$, $Y=(x,z)$, $Z=(y+\pi,z)$, $f,g$ are the unique contractible morphisms (Def. \ref{def: contractible morphism}) with the given sources and targets and $\psi$ is the composition:
\[
	\psi:Z\otimes i\xrightarrow h SG^2_\pi X\otimes i=X\otimes \sigma(i)\xrightarrow{id_X\otimes \varphi_i}X\otimes\tau(i)
\]
of the contractible morphism $h:(y+\pi,z)\otimes i\to (y+\pi,x+\pi)\otimes i=(x,y)\otimes \sigma(i)=X\otimes \sigma(i)$ and a natural isomorphism $F_\varphi=id_X\otimes \varphi_i:F_\sigma(X\otimes i)=X\otimes \sigma(i)\cong F_\tau(X\otimes i)=X\otimes \tau(i)$. 


Now consider the following diagram.
\[
\xymatrix{
X\otimes\sigma(i)\ar[d]^{id_X\otimes\varphi_i}\ar[r]^{\sigma(f)} &
Y\otimes\sigma(i)\ar[d]^{id_Y\otimes\varphi_i}\ar[r]^{\sigma(g)} &
Z\otimes\sigma(i)\ar[d]^{id_Z\otimes\varphi_i}\ar[r]^(.43){\sigma(h)} &
SG^2_\pi X\otimes \sigma(i)\ar[r]^(.53)=\ar[d]^{id\otimes \varphi_i}& 
X\otimes\sigma^2(i)\ar[d]^{id\otimes \sigma(\varphi_{i})}\ar[r]^{id\otimes \varphi_{\sigma(i)}} & 
X\otimes\tau\sigma(i)\ar[d]^{id\otimes\tau(\varphi_i)}\\
X\otimes\tau(i) \ar[r]^{\tau(f)}& 
Y\otimes\tau(i) \ar[r]^{\tau(g)}&
Z\otimes\tau(i) \ar[r]^(.43){\tau(h)}& 
SG^2_\pi X\otimes \tau(i)\ar[r]^(.53)= & 
X\otimes \sigma\tau(i) \ar[r]^{id\otimes \tau(\varphi_i)}& 
X\otimes\tau^2(i)
	}
\]
The top row is the distinguished triangle \eqref{eq: positive triangle} at $\sigma(i)$ instead of $i$. The morphisms $\sigma(f),\sigma(g),\sigma(h)$ are the required contractible morphisms since $\sigma$, being continuous, takes contractible morphisms to contractible morphisms. The second row is $F_\tau$ applied to \eqref{eq: positive triangle}. 


Since $\varphi$ is a natural transformation, the first three squares in this diagram commute. The fourth square commutes since $SG^2_\pi\otimes id=id\otimes \sigma$ in $\widetilde\cM_\sigma$. By the rotation axiom for triangulated categories this implies that the last square must anti-commute, i.e., 
\[
\sigma(\varphi_i)=-\varphi_{\sigma(i)}.
\]
This is because the shift functor $T=F_\tau$ applied to the distinguished triangle \eqref{eq: positive triangle} gives a distinguished triangle only when the sign of one of the arrow is reversed.

Thus, the axioms of a triangulated category require $\varphi$ to be skew-continuous.


\section{Classification of add-triangulated equivalence coverings of $\cM_0$}\label{sec4: classification thm}

This section contains the main result of this paper: the classification of all continuously add-triangulated equivalence coverings of $\cM_0$ assuming that they exist. The construction of these categories and the proof that they are continuously triangulated is a generalization of the construction of the continuous Frobenius category given in \cite{cfc}. We leave the details of that construction to Section \ref{ss: continuous Frobenius}, the last section, so as not to interrupt the narrative.

\subsection{Classification}

Recall that an equivalence covering of a topological category $\cX$ with a unique zero object is a topological category $\widetilde \cX$ with unique zero object and a functor $\widetilde \cX\to \cX$ which is an equivalence of $K$-categories and a continuous finite covering map on nonzero object spaces. (Unfortunately, the functor is not continuous on morphisms.)

\begin{defn}\label{def: add-triangulation}
 When $\cX$ is a topological $K$-category, a \emph{continuous add-triangulation} of an equivalence covering $\widetilde \cX$ is defined to be a continuous triangulation of the additive category $add^{top}\widetilde\cX$.
\end{defn}

We recall from Example \ref{rem: rescaling gives strong isomorphism} that, for any nonzero scalar $a\in K^\ast$ and any triangulated $K$-category with $\Delta$, the set of distinguished triangles, $\Delta_a$ is defined to be the set of triangles $(X,Y,Z,f,g,h)$ so that $(X,Y,Z,f,ag,h)\in \Delta$. We say that $\Delta,\Delta'$ are \emph{sign equivalent} if $\Delta'=\Delta$ or $\Delta'=\Delta_{-1}$. We also use the notion of ``strong isomorphism'' from Definition \ref{def: continuous triangulation}.

\begin{thm}[Theorem \ref{thmD1}]\label{thm: D in intro}
Continuously add-triangulated $n$-fold equivalence coverings of $\cM_0$ are given, up to sign equivalence, by triples $(\sigma,\tau,\varphi)$ where
\begin{enumerate}
\item $\sigma$ is a linear automorphism of the $n$ point $K$-category $\cC_n$.
\item $\tau$ is a continuous autoequivalence of $\cC_n$, i.e., one which commutes with $\sigma$.
\item $\varphi:\sigma\to\tau$ is a skew-continuous natural isomorphism, i.e., so that
\[
	\varphi_{\sigma(i)}=-\sigma(\varphi_i)
\]
for all $i\in\cC_n$.
\end{enumerate}
Given such a triple $(\sigma,\tau,\varphi)$, the corresponding add-triangulated equivalence covering, which we call $\widetilde \cM_n(\sigma,\tau,\varphi)$, has underlying continuous $K$-category $\widetilde \cM_\sigma$ with holonomy functor $F_\sigma$, shift functor $F_\tau$ and universal virtual triangle (defined in section \ref{ss: universal virtual triangle} below) given by
\[
	X \xrightarrow{\tiny\mat{1\\1}} I_1X\oplus I_2 X\xrightarrow{[-1,1]} F_\sigma X\xrightarrow{\varphi_X} F_\tau X.
\]
Furthermore, $\widetilde \cM_n(\sigma,\tau,\varphi)$ is strongly isomorphic to $\widetilde \cM_n(\sigma',\tau',\varphi')$ if and only if there is an automorphism $\rho$ of $\cC_n$ so that $\rho\circ\sigma=\sigma'\circ\rho$ and $\rho\circ\tau=\tau'\circ \rho$.
\end{thm}

Before proving this result, we point out some immediate consequences.

\begin{cor}\label{cor: classification up to strong isomorphism}
Add-triangulated $n$-fold equivalence covering of $\cM_0$ are given up to strong isomorphism by conjugacy classes of pairs $(\sigma,\tau)$ where $\sigma,\tau$ are commuting, skew-compatible, self-equivalences of $\cC_n$ and $\sigma$ is an automorphism of $\cC_n$.
\end{cor}

\begin{cor}[Corollary \ref{cor: E1}]\label{duality of equivalence coverings} If $\widetilde\cM_n(\sigma,\tau,\varphi)$ is an add-triangulated equivalence covering of $\cM_0$ where $\tau$ is an automorphism of $\cC_n$, then $\widetilde\cM_n(\tau,\sigma,\varphi^{-1})$ is also an add-triangulated equivalence covering of $\cM_0$.
\end{cor}

We call $\widetilde\cM_n(\tau,\sigma,\varphi^{-1})$ the \emph{dual} of $\widetilde\cM_n(\sigma,\tau,\varphi)$.

The proof of Theorem \ref{thm: D in intro} relies on the following lemmas.

\begin{lem}[Theorem \ref{thmA}]
The continuous $n$-fold equivalence coverings of $\cM_0$ are given by $\widetilde \cM_\sigma$ where $\sigma$ is an automorphism of $\cC_n$.
\end{lem}

\begin{lem}[Corollary \ref{thmB}]
Continuous functors $\widetilde \cM_\sigma\to \widetilde \cM_{\sigma'}$ are given by $K$-linear functors $\tau:\cC_n\to\cC_m$ so that $\tau\circ\sigma=\sigma'\circ\tau$. In particular, continuous autoequivalences of $\widetilde \cM_\sigma$ are given by autoequivalences $\tau$ of $\cC_n$ which commute with $\sigma$.
\end{lem}

\begin{lem}\label{main lemma}
Continuous triangulations are induced from a universal virtual triangle.
\end{lem}

We show in section \ref{ss: continuous Frobenius} below that each universal virtual triangle comes from a continuous Frobenius category with continuous choice of universal exact sequences. 
Here we consider the converse: Given a continuous triangulation, show that it comes from a universal virtual triangle and that, consequently, it is given as the stable category of a continuous Frobenius category. The first statement is Lemma \ref{main lemma} which is more precisely stated as Theorem \ref{thm: virtual triangle give all triangles} below. The second statement is Theorem \ref{thm: add M(s,t,phi) is continuously triangulated} proved in the last section.

\subsection{Universal virtual triangle}\label{ss: universal virtual triangle}

Let $X$ be an indecomposable object of some $\widetilde \cM_\sigma$ with chosen shift functor $F_\tau$. Suppose $p(X)=E(x,y)\in\cM_0$ where $E(x,y)$ denotes the equivalence class of $(x+\pi,y)$ in $\cM_0$ (Def. \ref{def: E(a,b) and ends}). Thus $x<y<x+2\pi$ and $0\le x<y<2\pi$. For any $0<\varepsilon_1< x+2\pi-y$ and $0<\varepsilon_2< y-x$, let $I_1^{\varepsilon_1} X$, $I_2^{\varepsilon_2} X$ be the unique objects over $E(x,x+2\pi-\varepsilon_1)$ and $E(y-\varepsilon_2,y)$ respectively so that there are contractible morphisms $f_1:X\to I_1^{\varepsilon_1} X$, $f_2:X\to I_2^{\varepsilon_2} X$. We denote these morphisms by ``1'' or, more generally, we denote by a scalar $a$ a morphism which is $a$ times a contractible morphism.

Up to rescaling, we are given the add-triangulated structure on $\cM_0$. (The whole point is to find an equivalence covering $\widetilde\cM_\sigma$ with a continuous triangulation which is algebraically equivalent to the given algebraic triangulation of $add\,\cM_0$.) So, we know from \cite{ccc} how to complete the triangle starting with the contractible morphisms $X\to I_1^{\varepsilon_1} X\oplus I_2^{\varepsilon_2} X$ to get:
\[
	X\xrightarrow{\binom11} I_1^{\varepsilon_1} X\oplus I_2^{\varepsilon_2} X\xrightarrow{[g_1,g_2]} Y \xrightarrow h F_\tau X
\]
We can choose $Y$ to be the object in $\widetilde \cM_\sigma$ in the contractible support of $X$ which lies over $E(y-\vare_2,x+2\pi-\vare_1)$. The morphisms $g_i$ will be scalar multiples of the unique contractible morphisms $I_i^\varepsilon X\to Y$. Dividing $g_1$ and $g_2$ by the scalar for $g_2$ and multiplying $h$ by the same scalar we may assume that $g_2$ is contractible, i.e. its scalar is 1. Then $-g_1$ will also be contractible. The morphism $h:Y\to F_\tau X$ factors uniquely through the unique contractible morphism $Y\to F_\sigma X$ inducing a morphism
\[
	\varphi_X:F_\sigma X\to F_\tau X.
\]
\begin{lem} $\varphi_X$ is induced by a skew-continuous natural transformation $\sigma\to \tau$.
\end{lem}

\begin{proof}
We have discussed this point many times. The object space of the equivalence covering $\cM_\sigma$ is an $n$-fold covering of the Moebius strip. We number the sheets $1,\cdots,n$ (choosing a fundamental domain). As we go around the Moebius strip, we go from sheet $i$ to sheet $\sigma(i)$. So, when $X$ lies in sheet $i$,$F_\sigma X$ lies in sheet $\sigma(i)$ and $F_\tau X$ lies in sheet $\tau(i)$. Since we have a continuous triangulated structure, the morphism $F_\sigma X\to F_\tau X$ is given by $c_ix_{\tau(i)\sigma(i)}$ where $c_i$ is a fixed scalar for all $X$ in sheet $i$. When we move $X$ continuously to the next sheet, sheet $\sigma(i)$, $X$ moves to $Y=F_\sigma X$ by definition and, by continuity, $\varphi_X$ moves to the morphism 
\[
	c_i \sigma(x_{\tau(i)\sigma(i)})=c_i a_{\tau(i)\sigma(i)}x_{\sigma\tau(i)\sigma^2(i)}.
\]
However, we are over a Moebius strip. So, the positions of $I_1^{\varepsilon_1} X,I_2^{\varepsilon_2} X$ have switched. So, the sign of $\varphi_{F_\sigma X}$ is the negative of the continuation of $\varphi_X$. In other words,
\[
	c_i a_{\tau(i)\sigma(i)}x_{\sigma\tau(i)\sigma^2(i)}=-c_{\sigma(i)}x_{\sigma\tau(i)\sigma^2(i)}.
\]
Comparing this with \eqref{eq: skew-commutativity condition} we see that the scalars $c_i$ define a skew-continuous natural transformation $\varphi:\sigma\to \tau$ as claimed.
\end{proof}

By continuity, as $X$ varies throughout all of $\widetilde\cM_\sigma$ and $\vare_1,\vare_2$ vary through all admissible real numbers (i.e. values so that the following diagram makes sense), the scalars $c_i$ must remain fixed on sheet $i$. So, we have a continuous family of distinguished triangles
\begin{equation}\label{eq: universal virtual triangles}
		X\xrightarrow{\binom11} I_1^{\varepsilon_1} X\oplus I_2^{\varepsilon_2} X\xrightarrow{(-1,1)} Y \xrightarrow{\tilde \varphi} F_\tau X
\end{equation}
where $\tilde\varphi$ is the composition of the unique contractible morphism $Y\to F_\sigma X$ with $\varphi_X:F_\sigma X\to F_\tau X$.

\begin{defn}\label{def: universal virtual triangle}
We denote this family of distinguished triangles by its virtual limit (when $\vare_1,\vare_2$ are infinitesmal, making $Y=F_\sigma X$)
\[
			X\xrightarrow{\binom11} I_1^0 X\oplus I_2^0 X\xrightarrow{(-1,1)} F_\sigma X \xrightarrow{ \varphi} F_\tau X
\]
We call this the \emph{universal virtual triangle}. This is just a fancy way to denote the continuous family of distinguished triangles given by \eqref{eq: universal virtual triangles}.
\end{defn}

\subsection{Construction of all distinguished triangles}
So far we have shown that a continuous triangulation of $add\, \widetilde\cM_\sigma$ has parameters $\sigma,\tau$ and $\varphi:\sigma\to \tau$ and these give rise to a ``universal virtual triangle''. In this subsection we prove the following theorem.

\begin{thm}\label{thm: virtual triangle give all triangles}
The universal virtual triangle defined above determines all distinguished triangles. In particular, the continuous triangulation of $add\, \widetilde\cM_\sigma$ is completely determined by $\tau$ and $\varphi:\sigma\to \tau$.
\end{thm}

First, consider morphisms $f:X\to Y$ in $add\,\widetilde\cM_\sigma$ which are \emph{generic} in the sense that $X$, $Y$ do not share any ends. (Recall that the \emph{ends} of $E(x,y)$ are $x,y\in \widetilde S^1=\RR/2\pi\ZZ$.)

\begin{lem}\label{lem: Z has size of X+Y when f:X to Y is generic}
Let $X\to Y\to Z\to TX$ be a distinguished triangle. If $X\to Y$ is generic, then $Z$ has the same number of components as $X\oplus Y$. 
\end{lem}

\begin{proof}
We know that distinguished triangles are exact on ends. So, $Z$ must have the same number of ends as $X\oplus Y$. But the number of components is half the number of ends.
\end{proof}

\begin{lem}\label{lem: distinguished triangles are limits of generic ones}
Every distinguished triangle is isomorphic to a continuous limit of distinguished triangles in which the first morphism is generic.
\end{lem}

\begin{proof}
If $f:X\to Y$ is not generic then, by a small perturbation of the components of $Y$ (moving points over $E(x,y)$ to points over $E(x+\vare,y+\vare)$ but staying on the same sheet of the covering) and by composing $f:X\to Y$ with short contractible morphisms to the new $Y$, we obtain a new distinguished triangle with more components in $Z$ so that, as $Y$ moves back to its original position, some of these new components go to zero. The limit is a distinguished triangle since the space of distinguished triangles is assumed to be closed. The limit must be equivalent to the distinguished triangle we started with since it has the same first morphism $f:X\to Y$. So, every distinguished triangle is equivalent to one which is a limit of distinguished triangles with generic first morphism.
\end{proof}

\begin{proof}[Proof of Theorem \ref{thm: virtual triangle give all triangles}]
By Lemma \ref{lem: distinguished triangles are limits of generic ones}, we only need to prove the theorem in the case of when $f:X\to Y$ is generic. Given any generic morphism $f:X\to Y$, consider the distinguished triangles given by:
\begin{equation}\label{eq: mock pushout diagram}
	X\xrightarrow{(f,1,1)} Y\oplus I_1^{\vare_1}X \oplus I_2^{\vare_2}X\to Z\to TX.
\end{equation}
Here, the first morphism is a monomorphism on ends and, therefore, the last morphism $Z\to TX$ is generic. By Lemma \ref{lem: Z has size of X+Y when f:X to Y is generic}, the number of components of $Y\oplus I_1^{\vare_1}X \oplus I_2^{\vare_2}X$ is equal to the number of components of $X\oplus Z$. Since $I_1^{\vare_1}X$, $I_2^{\vare_2}X$ each has the same number of components as $X$ this implies that $Z$ has the same number of components as $X\oplus Y$.

As $\vare_1,\vare_2$ both go to zero, the distinguished triangle \eqref{eq: mock pushout diagram} converges to a distinguished triangle $X\to Y\to Z'\to TX$ where, by Lemma \ref{lem: Z has size of X+Y when f:X to Y is generic}, $Z'$ has the same number of components as $X\oplus Y$. Since $Z$ converges to $Z'$ and they have the same number of components we conclude that the components of $Z$ converge to the components of $Z'$. In particular, they do not vanish. So, we get all distinguished triangles of the form
\[
	X\xrightarrow fY\to Z\to TX
\]
with $f$ generic as a limit (or equivalent to a limit) of distinguished triangles of the form \eqref{eq: mock pushout diagram}. Thus, we are reduced to showing that all distinguished triangles of the form \eqref{eq: mock pushout diagram} are determined by the universal virtual triangle.

As the values of $\vare_1,\vare_2$ change, the distinguished triangle \eqref{eq: mock pushout diagram} varies continuously with $Z$ having a fixed number of components. Thus, $X,Y$ are fixed and the components of $Z$ move around continuously. By continuity, the set of triangles for each value of the parameters $\vare_1,\vare_2$ determines the set of triangles for all values of these parameters.

Now deform the parameters $\vare_1,\vare_2$ to an infinitesimal amount less than their maximum possible values. Then the morphism $f:X\to Y$ will factor through $I_1^{\vare_1}X \oplus I_2^{\vare_2}X$. This will cause the induced map $Y\to Z$ to be a split monomorphism in the category $add\,\widetilde \cM_\sigma$. So, the distinguished triangle \eqref{eq: mock pushout diagram} becomes a direct sum of distinguished triangles of the form \eqref{eq: universal virtual triangles} and trivial triangles $0\to Y\xrightarrow{=} Y\to 0$. These are given by the universal virtual triangle by definition concluding the proof of the theorem.
\end{proof}


\subsection{Consequences}

\begin{lem}
Any subset of $[n]$ invariant under both $\sigma$ and $\tau$ generates a $(\sigma,\tau)$-invariant full subcategories of $\cC_n$ and therefore of $\widetilde\cM_n(\sigma,\tau,\varphi)$.\qed
\end{lem}

\begin{cor}\label{cor: n is even}
Let $\widetilde \cM_n(\sigma,\tau,\varphi)$ be an add-triangulated equivalence covering of $\cM_0$ so that $\tau$ is an automorphism of $\cC_n$. Then $n$ is even.
\end{cor}

\begin{proof}
Suppose $n$ is odd. Then, as permutations of $n$, $\sigma$ must have an odd number of odd cycles and the action of $\tau$ on these odd cycles must have an odd cycle. This gives a full subcategory $\cC_{pq}$ of $\cC_n$, with $p,q$ odd, whose object set we identify with $[p]\times[q]$ on which $\sigma$ acts by cyclically permuting the first factor and $\tau$ acts by permuting both factors but acting cyclically the second: $\sigma(i,j)=(i+1,j)$, $\tau(i,j)=(\tau_1(i),j+1)$.

For each $(i,j)\in[p]\times[q]$ we have the natural transformation $\varphi_{ij}$ which is a nonzero morphism
\[
	\varphi_{ij}:(i+1,j)\to (\tau_1(i),j+1)
\]
Since $\varphi$ is skew commutative, $\sigma(\varphi_{ij})=-\varphi_{i+1,j}$. Doing this $p$ times we get $\sigma^p(\varphi_{ij})=-\varphi_{ij}$. This means that the $p$-th power of the transition factor of $\sigma$ from $A_j=[p]\times j$ to $A_{j+1}$ is $(a_{A_j,A_{j+1}})^p=-1$ for every $j$. Composing these, the $pq$-th power of this transition factor will also be $-1$. But this is impossible since this is the transition factor $a_{A_j,A_{j}}=1$.
\end{proof}

\begin{cor}\label{cor: minimal case is even}
If $\widetilde \cM_n(\sigma,\tau,\varphi)$ is a minimal add-triangulated equivalence covering of $\cM_0$ then $\tau$ is an automorphism of $\cC_n$ and $n$ is even.
\end{cor}

\begin{proof}
This follows from Corollary \ref{cor: n is even} since $\tau$ must be an isomorphism: otherwise, the image of $\tau^n$ is a smaller $\sigma,\tau$ invariant full subcategory of $\cC_n$ which generates a proper equivalence subcover of $\widetilde \cM_n(\sigma,\tau,\varphi)$ contradicting its minimality.
\end{proof}


In particular, $\widetilde \cM_n(\sigma,\tau,\varphi)$ is minimal when $n=2$ since any proper subcovering would have $n=1$. 

\subsection{2-fold equivalence coverings of $\cM_0$}

We will describe all $2$-fold equivalence coverings $\widetilde \cM_2(\sigma,\tau,\varphi)$ of $\cM_0$. There are four cases for the underlying permutations of $\sigma,\tau$, although Case (0) is not possible. See Figure \ref{Figure99}.
\begin{enumerate}
\item[(0)] $\sigma=\tau=id$ (as permutations)
\item $\sigma=id$, $\tau=(12)$
\item $\sigma=(12),\tau=id$
\item $\sigma=\tau=(12)$
\end{enumerate}
In each case, we use the notation $a_{ij}$, $b_{ij}$ for the transition coefficients of $\sigma$ and $\tau$ with respect to a fixed choice of multiplicative basis (Definition \ref{def: transition coefficients}). The parameters $a_{12},b_{12}$ determine the others. The natural transformation $\varphi$ is specified by parameters $c_i\in K^\ast$, $i=1,2$ where $\varphi_i=c_ix_{\tau(i)\sigma(i)}$. Recall that these parameters satisfy the following conditions.
\begin{enumerate}
\item[(a)] (naturality of $\varphi$, Prop. \ref{prop: natural isomorphisms are unique up to rescaling}) $b_{ji}c_i=c_ja_{ji}$ for all $i,j$. 
\item[(b)] (skew-commutativity of $\varphi$ \eqref{eq: skew-commutativity condition}) $a_{\tau(i)\sigma(i)}c_i=-c_{\sigma(i)}$. 
\end{enumerate}
Finally, recall that ``rescaling'' of the distinguished triangles means multiplying all $\varphi_i$ and thus all $c_i$ by the same nonzero scalar $r$. This preserves the $c$-homogeneous relations (a), (b). Up to rescaling, only the ratio $c_1/c_2$ is well-defined (when $n=2$).

Our goal is to show that, in Cases 1,2,3, an add-triangulated category $\widetilde \cM_n(\sigma,\tau,\varphi)$ exists and is unique up to continuous isomorphism and rescaling.

We need the following trivial observation.

\begin{lem}
If $\sigma$ is the identity permutation then its transition coefficients $a_{ij}$ are well-defined, i.e., independent of the choice of $x_{ij}$.
\end{lem}

\begin{proof}
$a_{ij}$ is the unique eigenvalue of the action of $\sigma$ on $\cC_n(j,i)\cong K$.
\end{proof}

\noindent{\bf Case 0}: $\sigma=\tau=id$. The minimal invariant subcategory of $\cC_2$ is $\cC_1$ which contradicts Corollary \ref{cor: minimal case is even}. So, this case is not possible.
\smallskip

\noindent{\bf Case 1}: $\sigma=id,\tau=(12)$. Since $\{1,2\}$ is a single $\tau$-orbit, there is a multiplicative basis $x_{ij}$ for $\cC_2$ so that $\tau(x_{12})=x_{21}$ and $\tau(x_{21})=x_{12}$. I.e., $b_{ij}=1$ for all $i,j$. 

Since $\sigma,\tau$ commute, we must have $a_{12}=a_{21}$ with product $a_{11}=1$. So, $a_{12}=\pm1$.

The skew-commutativity of $\varphi$ implies that $a_{21}c_1=-c_1$. So, $a_{12}=a_{21}=-1$.

Natruality of $\varphi$ implies $b_{21}c_1=c_2a_{21}$. So, $c_1/c_2=-a_{21}/b_{21}=-1$. 

To summarize: $(a_{12},b_{12},c_1/c_2)=(-1,1,-1)$ and $\widetilde\cM_n(\sigma,\tau,\varphi)$ is well-defined up to rescaling of triangles as defined in Example \ref{rem: rescaling gives strong isomorphism}. I.e., the triangulated categories in Case 1 are, up to continuous isomorphism, given by the single parameter $c_1\in K^\ast$.

\smallskip

\noindent{\bf Case 2}: $\sigma=(12),\tau=id$. As in Case 1, we have: $a_{ij}=1$ for all $i,j$ and $b_{12}=b_{21}=\pm1$.

By skew-commutativity of $\varphi$ we have $a_{12}c_1=-c_2$. So, $c_1/c_2=-1$.

By naturality of $\varphi$ we have $b_{21}c_1=c_2a_{21}$. So, $b_{21}=b_{12}=-1$.

To summarize: $(a_{12},b_{12},c_1/c_2)=(1,-1,-1)$ and $\widetilde\cM_n(\sigma,\tau,\varphi)$ is well-defined up to the rescaling factor $c_1$.

\smallskip

\noindent{\bf Case 3}: $\sigma=\tau=(12)$. In this case we have a choice of normalizing either $\sigma$ or $\tau$. To help decide, we first evaluate $c_1/c_2$.

By skew-commutativity of $\varphi$ we have $a_{22}c_1=-c_2$. So, $c_1/c_2=-1$.

By naturality of $\varphi$ we have $b_{21}c_1=c_2a_{21}$. So, $a_{21}=-b_{21}$.

We still have two choices. We choose to normalize $\sigma$ and take a multiplicative basis so that $a_{ij}=1$ for all $i,j$ and $b_{12}=b_{21}=-1$.

To summarize: $(a_{12},b_{12},c_1/c_2)=(1,-1,-1)$, as in Case 2 and $\widetilde\cM_n(\sigma,\tau,\varphi)$ is well-defined up to the rescaling factor $c_1$. However, with a different choice of $x_{ij}$, we would have another description of the same category: $(a_{12},b_{12},c_1/c_2)=(-1,1,-1)$ as in Case 1.

\begin{rem}
Case 3, with parameters $c_1=1,c_2=-1$, is the continuous cluster category constructed in \cite{ccc},\cite{cfc}. The objects of the 2-fold equivalence cover of $\cM_0$ are ordered pairs of distinct points on the circle $\RR/2\pi\ZZ$ with holonomy given by rotation by $\pi$ (making $a_{12}=1$) and with the distinguished triangles rescaled by $-1$ in the holonomy. This gives $b_{12}=-1$. The details given in the present paper show why this is in fact continuous.

Case 1 is a variation of a special case of a construction Orlov \cite{Orlov}. In this case, Orlov actually constructs a 4-fold equivalence covering of $\cM_0$ but we interpret it as a recipe to construct a 2-fold covering. The indecomposable objects of Orlov's category are matrix factorizations on an ordered pair of points in $\RR/2\pi\ZZ$: 
\[
\xymatrix{
P_1 \ar@/_1pc/[r]^{d_1} & 
	P_0 \ar@/_1pc/[l]^{d_0}
	}
\]
However, we modify this by identifying $(P_1,P_0,d_\ast)$ with $(P_0,P_1,d_\ast)$, i.e., by taking unordered pairs of points. In our setting we take $d_0,d_1$ to be the unique contractible morphisms between the two points in $\cR_{2\pi}/G_{2\pi}$. This gives one object for every unordered pair of points on a circle of radius 1, i.e., for every object in $\cM_0$.

Orlov takes the shift functor to be given by changing the sign of both $d_0$ and $d_1$. This gives another isomorphic copy of each object. So, there are two copies of each object. Changing the sign of $d_i$ gives a recognizably distinct component in the space of objects. So, $\sigma$ is the identity permutation and $\tau=(12)$.

Case 2 is a new construction. The underlying topological category $\widetilde\cM_\sigma$ is the same as in Case 3, but the continuous triangulated structure is different. Note that $b_{12}=-1$ implies that, although the shift functor $\tau$ is the identity on objects, it is not the identity functor.
\end{rem}

\begin{cor}[Corollary \ref{cor: E2}]\label{cor: classification for n=2}
Up to rescaling (Example \ref{rem: rescaling gives strong isomorphism}), there are exactly three two-fold equivalence coverings of $\cM_0$ as shown in Figure \ref{Figure99}.
\end{cor}

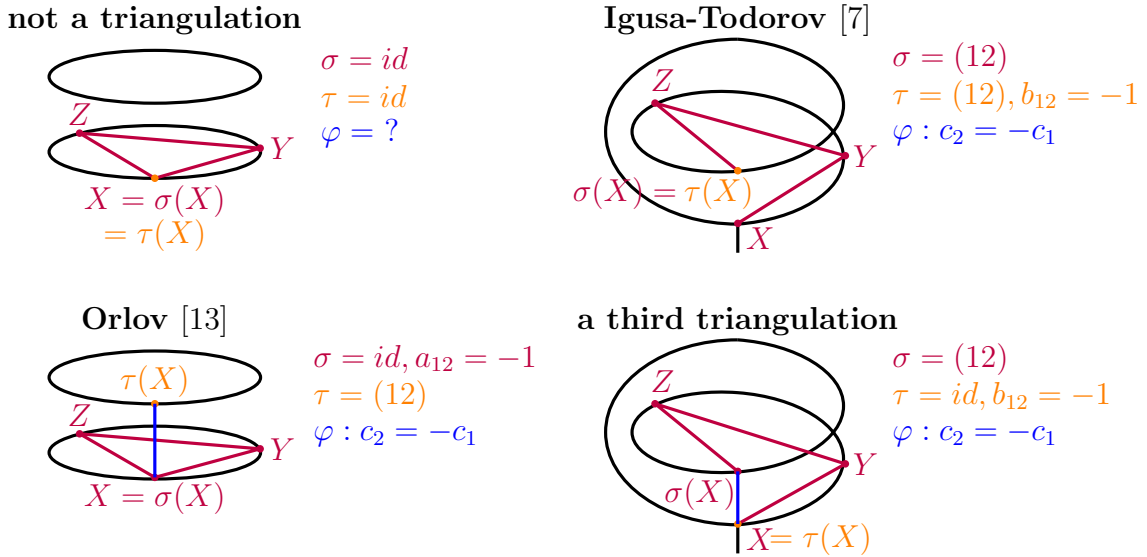
\begin{figure}[htbp]
\begin{center}
\begin{tikzpicture}[decoration={segment length=20pt}, very thick]
\draw (.25,4.25) ellipse (40pt and 10pt)
(.25,3.25) ellipse (40pt and 10pt)
(.25, 5) node {\textbf{not a triangulation}};

\draw [purple] (.25,2.6) node {$X = \sigma(X)$ }
(.25,2.9) -- (1.65, 3.3) 
(1.95, 3.3) node {$Y$}
(1.65, 3.3) -- (-.75,3.5)
(-.75, 3.75) node {$Z$}
(-.75,3.5) -- (.25,2.9)

(2.3,4.5) node[right] {$\sigma = id$};

\fill[purple] 
(1.65, 3.3) circle (.05cm)
(-.75, 3.5) circle (.05cm)
;

\fill[orange]
(.25, 2.9) circle (.05cm);

\draw [orange] (.25,2.15) node {$=\tau(X)$}
(2.3,4) node[right] {$\tau = id$}
;

\draw[blue] 
(2.3,3.5) node[right]  {$\varphi = $ ?};


\begin{scope}[yshift=-.5cm]
\node(a) {} node (b) at (1,1.25){} node (c) at (1,-1.75){};
\draw (.25,.75) ellipse (40pt and 10pt)
(.25,-.25) ellipse (40pt and 10pt)
(.25, 1.5) node {\textbf{Orlov \cite{Orlov}}};

\draw [purple] (.25,-.85) node {$X = \sigma(X)$}
(.25,-.58) -- (1.65, -.2)
(1.95, -.2) node {$Y$}
(1.65, -.2) -- (-.75, 0)
(-.75, .25) node {$Z$}
(-.75,0) -- (.25, -.58)
(2.2, 1) node[right] {$\sigma = id, a_{12}=-1$}
;

\fill[purple] 
(.25, -.58) circle (.05cm)
(1.65, -.2) circle (.05cm)
(-.75, 0) circle (.05cm)
;

\draw [orange] (.25,.7) node {$\tau(X)$}
(2.2,.5) node[right] {$\tau = (12)$}
;

\fill[orange] (.25,.4) circle (.05cm);

\draw[blue] (.25,-.58) -- (.25,.4)
(2.2,0) node[right] {$\varphi: c_2=-c_1$};
\end{scope}


\begin{scope}[xshift=.5cm]
\node(a) {} node (b) at (7.5,4.75){} node (c) at (7.5,1.75){};
\draw[decorate,decoration={coil, amplitude=40, aspect=.5}] (7.5,4.75) -- (c);
\draw (7.5,4.75) arc (90: 270 : 50pt and 35pt)
(7.5, 5) node {\textbf{Igusa-Todorov \cite{cfc}}}
;

\draw [purple] (7.8,2.1) node {$X $}
(7.5,2.3) -- (8.93, 3.2)
(9.2, 3.2) node {$Y$}
(8.93, 3.2) -- (6.4, 3.9)
(6.5, 4.2) node {$Z$}
(6.4,3.9) -- (7.5, 3)
(6, 2.7) node {$\sigma(X) =$}
(9.4, 4.5) node[right] {$\sigma = (12)$}
;

\fill[purple] (7.5,2.3) circle (.05cm)
(8.93, 3.2) circle (.05cm)
(6.4, 3.9) circle (.05cm)
;

\draw [orange] (7.25,2.7) node {$\tau(X)$}
(9.4,4) node[right] {$\tau = (12), b_{12}=-1$}
;

\fill[orange] (7.5,3) circle (.05cm);

\draw[blue] 
(9.4,3.5) node[right] {$\varphi: c_2 = -c_1$}
;
\end{scope}


\begin{scope}[yshift=-.5cm, xshift=.5cm]
\node(a) {} node (b) at (7.5,1.25){} node (c) at (7.5,-1.75){};
\draw[decorate,decoration={coil, amplitude=40, aspect=.5}] (7.5,1.25) -- (c);
\draw (7.5,1.25) arc (90: 270 : 50pt and 35pt)
(7.5, 1.5) node {\textbf{a third triangulation}};

\draw [purple] (7.8,-1.4) node {$X$}
(7.5,-1.2) -- (8.93, -.4)
(9.2, -.4) node {$Y$}
(8.93, -.4) -- (6.4, .4)
(6.5, .7) node {$Z$}
(6.4,.4) -- (7.5, -.5)
(7.0, -.8) node {$\sigma(X)$}
(9.4, 1) node[right] {$\sigma = (12)$}
;

\fill[purple] 
(8.93, -.4) circle (.05cm)
(6.4, .4) circle (.05cm)
(7.5, -.5) circle (.05cm)
;

\draw [orange] (8.6,-1.4) node {$= \tau(X)$}
(9.4,.5) node[right] {$\tau = id, b_{12}=-1$}
;

\fill[orange] (7.5,-1.2) circle (.05cm);

\draw[blue] (7.5,-1.2) -- (7.5,-.5)
(9.4,0) node[right] {$\varphi: c_2 = -c_1 $}

;
\end{scope}

\end{tikzpicture}
\caption{Schematic drawings of the three possible add-triangulated 2-fold equivalence coverings of $\cM_0$.}
\label{Figure99}
\end{center}
\end{figure}
%


%
%

\section{Automorphisms of ${\cC}_n$}\label{lin}\label{Sec 5: App}

We now discuss technical details about the discrete structures $\cC_n,\sigma,\tau,\varphi$ that we used to classify the continuous triangulations of equivalence coverings of the Moebius category $\cM_0$. This will allow us to construct continuous Frobenius categories in the last section thereby showing that all continuously triangulated categories listed in the classification theorem exist. In Theorem \ref{thm: connected equivalence covers} we use this discussion to list all connected add-triangulated coverings of $\cM_0$. More generally we show, in Corollary \ref{cor: only finitely many indecomposable n-fold coverings} that there are, up to isomorphism, only finitely many ``indecomposable'' add-triangulated equivalence coverings of $\cM_0$. We assume that $K$ is algebraically closed.

\subsection{Fattened Categories and the category $\cC_n$} By ``fattening'' a category we mean taking several isomorphic copies of each object. This is the discrete version of an equivalence covering. We will concentrate on fattening the category $\cC_1$ having one object with endomorphism ring $K$ and show that such a category is necessarily isomorphic to $\cC_n$ from Definition \ref{def: Cn}. All categories will be $K$-categories and all functors will be $K$-linear.

\begin{thm}\label{thm: fat K is Cn}
Let $\cC$ be any category with $n$ objects, all of which are isomorphic with endomorphism ring $K$. Then $\cC$ is isomorphic to $\cC_n$.
\end{thm}

\begin{proof}
Label the objects of $\cC$ as $v_1,\cdots,v_n$. For each $i<n$ choose an isomorphism $f_i:v_i\cong v_{i+1}$ with inverse $g_i$. For any $i,j\in [n]$ let $y_{ji}:v_i\to v_j$ be given by the unique reduced composition of the morphism $f_k,g_k$ (reduced means $f_k\circ g_k$, $g_k\circ f_k$ do not occur in the composition). Then $i\mapsto v_i$ and $x_{ji}\mapsto y_{ji}$ gives an isomorphism $\cC_n\cong \cC$.
\end{proof}


\subsection{Automorphisms of $\cC_n$} Recall from Corollary \ref{cor: formula for automorphisms of Cn} that a $K$-linear automorphism $\sigma$ of $\cC_n$ is given by a permutation of $n$ and multiplicative system of scalars $a_{ij}$ called the transition coefficients of $\sigma$. The formula is
\[
	\sigma(x_{ij})=a_{ij}x_{\sigma(i)\sigma(j)}
\]
where $\{x_{ij}\in \cC_n(j,i)\}$ is a multiplicative basis for $cC_n$.

We will show that, if the underlying permutation of $\sigma$ is an $n$ cycle, then we can arrange for the transition coefficients $a_{ij}$ to be equal to 1 for all $i,j$ by changing the basis for the hom sets. In particular, $\sigma^n$ is the identity functor on $\cC_n$. More generally, we can arrange for $a_{ij}$ to be equal to 1 when $i,j$ are in the same orbit of $\sigma$.

\begin{eg}
Suppose $\sigma=(123\cdots n)$ is a single $n$-cycle on objects. Given $\sigma(x_{ij})=a_{ij}x_{i+1,j+1}$, by Proposition \ref{prop: comparison of multiplicative bases} we can choose $c_i\in K^\ast$ so that $a_{ij}=c_i/c_j$. Since $K$ is algebraically closed, there exists a $d\in K^\ast$ so that
\[
	d^n=c_1c_2\cdots c_n
\]
Thus $d$ is the \emph{geometric mean} of the $c_i$ and $d$ is unique up to multiplication by an $n$-th root of unity. Then we can choose another multiplicative basis $x_{ij}'=b_{ij}x_{ij}$ depending on the transition coefficients of $\sigma$ with respect to $x_{ij}$ by:
\begin{enumerate}
\item $b_{ii}=1$
\item $b_{ij}=d^{j-i}(c_ic_{i+1}\cdots c_{j-1})^{-1}$ if $i<j$. In particular $b_{i,i+1}=d/c_i$ if $i\le n-1$. 
\item $b_{ij}=b_{ji}^{-1}$ if $j<i$. In particular $b_{n1}=b_{1n}^{-1}=c_1\cdots c_{n-1}d^{1-n}=d/c_n$. 
\end{enumerate}
Therefore, $b_{i,\sigma(i)}=d/c_i$ for all $i$. It is easy to see that $(b_{ij})$ is multiplicative. And:
\[
	\sigma(x_{i,i+1}')=\sigma(b_{i,i+1}x_{i,i+1})=\frac d{c_i}\sigma(x_{i,i+1})
	=\frac d{c_i}\frac{c_i}{c_{i+1}}x_{i+1,i+2}=\frac d{c_{i+1}}x_{i+1,i+2}=x_{i+1,i+2}'
\]
This implies that
\[
	\sigma(x_{ij}')=\sigma(x_{i,i+1}'\cdots x_{j-1,j}')=x_{i+1,i+2}'\cdots,x_{j,j+1}'=x_{i+1,j+1}'
\]
for all $i,j$. In fraction form, $b_{ij}=s_i/s_j$ where $s_i=c_1\cdots c_{i-1}d^{-i}$.
\end{eg}

This proves the following.

\begin{lem}
If $\sigma\in\Aut(\cC_n)$ is an $n$ cycle then $\cC_n$ has a multiplicative basis $(x_{ij})$ so that $\sigma(x_{ij})=x_{\sigma(i)\sigma(j)}$ for all $i,j\in\cC_n$. In particular, $\sigma^n$ is the identity automorphism of $\cC_n$.\qed
\end{lem}

\begin{prop}
Given any automorphism $\sigma$ of $\cC_n$, there exists a multiplicative basis $(x_{ij})$ so that the transition coefficients $(a_{ij})$ of $\sigma$ with respect to $(x_{ij})$ satisfy the following.
\begin{enumerate}
\item $a_{ij}=1$ if $i,j$ are in the same orbit of $\sigma$
\item $a_{ij}=a_{k\ell}$ if $i,k$ are in the same cycle of $\sigma$ and $j,\ell$ are in the same cycle of $\sigma$.
\end{enumerate}
Furthermore, $(1)$ implies $(2)$.
\end{prop}

\begin{proof} 
Choose any multiplicative basis $(y_{ij})$, for which $\sigma$ has {transition} coefficients $b_{ij} = s_i/s_j$. Since any $i$ lies in some cycle of $\sigma$, it suffices to choose a cycle, say $A$, and repeat the same argument for all cycles of $\sigma$. Let $d_A\in K^\ast$ be the geometric mean of the $s_i$ for $i\in A$. If $i_0$ is the smallest element of $A$ then $i=\sigma^m(i_0)$ for some $m\ge0$. Let $c_{i_0}=1$ and let
\[
	c_i=s_{i_0}s_{\sigma(i_0)}\cdots s_{\sigma^{m-1}(i_0)}/d_A^m
\]
if $m>0$. Let $a_{ij}=c_i/c_j$. Then $a_{i,\sigma(i)}=d_A/s_i$ for all $i$. Also, $(a_{ij})$ is a multiplicative system and $x_{ij}=a_{ij}y_{ij}$ is a new multiplicative basis. We have:
\[
	\sigma(x_{i,\sigma(i)})=\frac {d_A}{s_i}\sigma(y_{i,\sigma(i)})=\frac {d_A}{s_i}\frac{s_i}{s_{\sigma(i)}}y_{\sigma(i),\sigma^2(i)}=\frac {d_A}{s_{\sigma(i)}}y_{\sigma(i),\sigma^2(i)}=x_{\sigma(i),\sigma^2(i)}
\]
As in the example, this implies that $\sigma(x_{ij})=x_{\sigma(i),\sigma(j)}$ when $i,j$ are in the same $\sigma$ orbit.

To see that (1) implies (2), let $i,k \in A$ and $j,l \in B$ for two disjoint cycles $A, B$ of $\sigma$. Since $(a_{ij})$ is a multiplicative system, we have $a_{ij} = a_{ik}a_{kl}a_{lj}$. By (1), $a_{ik} = 1 = a_{lj}$, and so $a_{ij} = a_{kl}$.
\end{proof}

\begin{defn}
Given $\cC_n$ with an automorphism $\sigma$, a multiplicative basis $(x_{ij})$ for $\cC_n$ will be called a \emph{good multiplicative basis} (with respect to $\sigma$) if it satisfies the proposition above.

If $A,B,C,\cdots$ denote the orbits of the action of $\sigma$ on $n$, then let $a_{AB}$ be the element of $K^\ast$ so that $\sigma(x_{ij})=a_{AB}x_{\sigma(i)\sigma(j)}$ when $i\in A$ and $j\in B$. We call $a_{AB}$ the \emph{transition factors} of $\sigma$. Let $|A|$ denote the number of elements in the orbit $A$ of $\sigma$.
\end{defn}

We need to know to what extent the good basis and transition factors are uniquely determined by $(\cC_n,\sigma)$.

\begin{prop}\label{prop: change of good basis formula}
Let $(x_{ij})$ be a good basis for $(\cC_n,\sigma)$ with transition factors $a_{AB}$ and suppose that $(x_{ij}')$ is another good basis giving new transition factors $b_{AB}$ for $\sigma$. Then there are unique roots of unity $\delta_A$ for each orbit $A$ of $\sigma$ so that
\begin{enumerate}
\item $\delta_A^{|A|}=1$ for all $A$,
\item $b_{AB}=\delta_A a_{AB}\delta_B^{-1}$ for all orbits $A,B$ of $\sigma$ and
\item $x_{i\sigma(i)}'=\delta_A x_{i\sigma(i)}$ whenever $i\in A$.
\end{enumerate}
Furthermore, any collection of roots of unity $\delta_A$ satisfying (1) will occur for some $(x_{ij}')$.
\end{prop}

\begin{proof}
By Proposition \ref{prop: comparison of multiplicative bases}, $x_{ij}'=\frac{c_i}{c_j}x_{ij}$. When $i,j$ are in the same $\sigma$ orbit this gives
\[
	x_{\sigma(i)\sigma(j)}'=\sigma(x_{ij}')=\frac{c_i}{c_j}\sigma(x_{ij})=\frac{c_i}{c_j}x_{\sigma(i)\sigma(j)}=\frac{c_{\sigma(i)}}{c_{\sigma(j)}}x_{\sigma(i)\sigma(j)}
\]
Therefore, $\frac{c_i}{c_{\sigma(i)}}=\frac{c_j}{c_{\sigma(j)}}$ when $i,j$ are in the same orbit $A$ of $\sigma$. Denote this fraction by $\delta_A$. Then (3) will be satisfied. When $i,j$ are in different orbits, say $i\in A,j\in B$ we get:
\[
	b_{AB}x_{\sigma(i)\sigma(j)}'=\sigma(x_{ij}')=\frac{c_i}{c_j}\sigma(x_{ij})=\frac{c_i}{c_j}a_{AB}x_{\sigma(i)\sigma(j)}=b_{AB}\frac{c_{\sigma(i)}}{c_{\sigma(j)}}x_{\sigma(i)\sigma(j)}
\]
Therefore,
\[
	b_{AB}=\frac{c_ic_{\sigma(j)}}{c_jc_{\sigma(i)}}a_{AB}=\delta_A a_{AB}\delta_B^{-1}
\]
To see that condition (1) is satisfied suppose that $A = (12\cdots m)$ is one of the cycles of the permutation $\sigma$. Then
\[
	1=\frac{c_1}{c_1}=\frac{c_1}{c_2}\frac{c_2}{c_3}\cdots \frac{c_m}{c_1}=\delta_A^m
\]
Conversely, suppose $(\delta_A)$ is a collection of scalars satisfying (1). Then, for each orbit $A$ of $\sigma$, choose an element $i_0\in A$. Let $c_{i_0}=1$ and let $c_{\sigma^k(i_0)}=\delta_A^{-k}$. Then $x_{ij}'=\frac{c_i}{c_j}x_{ij}$ will satisfy (2) and (3).
\end{proof}

\begin{cor}
If orbits $A,B$ of $\sigma$ have the same size, say $m$, then the $m$-th power of $a_{AB}$, the $AB$ transition factor of $\sigma$, is well-defined, i.e., independent of the choice of $(x_{ij})$.
\end{cor}

\begin{proof}
By Proposition \ref{prop: change of good basis formula}, any other choice of transition factors $b_{AB}$ will be related to $a_{AB}$ by $b_{AB}=\delta_Aa_{AB}\delta_B^{-1}$ where $\delta_A,\delta_B$ are $m$-th roots of unity. Then $b_{AB}^m=\delta_A^ma_{AB}^m\delta_B^{-m}=a_{AB}^m$.
\end{proof}

The properties we have developed in this section will be revisited in later sections. 

\subsection{Other quivers with multiplicity}

Now we consider the case of two vertices. A $K$-representation of the quiver:
\[
\xymatrix{
1\ar[r]^\tau&2
}
\]
consists of two vectors spaces $V_1,V_2$ and a linear map $\tau:V_1\to V_2$. We fatten this quiver, replacing the vertices with triples $(\cC_n,\sigma_1,(x_{ij}))$ and $(\cC_m,\sigma_2,(y_{kl}))$ where $(x_{ij}),(y_{kl})$ are good multiplicative bases for $\sigma_1,\sigma_2$ so that the transition coefficients for $\sigma_1$, $a_{ij}=1$ when $i,j$ are in the same $\sigma_1$ orbit and similarly $a_{kl}'=1$ if $k,l$ are in the same $\sigma_2$ orbit. Our goal is to describe how $\tau$ behaves after this fattening.

A \emph{morphism}
\[
	\tau:(\cC_n,\sigma_1,(x_{ij}))\to(\cC_m,\sigma_2,(y_{k\ell}))
\]
is defined to be a linear functor $\tau:\cC_n\to\cC_m$ so that $\tau\sigma_1=\sigma_2\tau$. In other words:
\begin{enumerate}
\item On objects, $\tau$ is a set map $[n]\to[m]$ where $[n]=\{1,2,\cdots,n\}$.
\item On morphisms, $\tau$ is given by transition coefficients $b_{ij}$, i.e., $\tau(x_{ij})=b_{ij}y_{\tau(i),\tau(j)}$.
\item The equation $\tau\sigma_1=\sigma_2\tau$ is equivalent to:
\[
	a_{ij}b_{\sigma_1(i),\sigma_1(j)}=b_{ij}a_{\tau(i),\tau(j)}'
\]
for all $i,j\in[n]$ and $\tau\sigma_1(i)=\sigma_2\tau(i)$ for all $i\in [n]$.
\end{enumerate}
The proof of (3) is straightforward. Just expand both sides of $\tau\sigma_1(x_{ij})=\sigma_2\tau(x_{ij})$:
\[
	\tau\sigma_1(x_{ij})=a_{ij}\tau(x_{\sigma_1(i)\sigma_1(j)})=a_{ij}b_{\sigma_1(i),\sigma_1(j)}y_{\tau\sigma_1(i)\tau\sigma_1(j)}
\]
\[
	=\sigma_2\tau(x_{ij})=b_{ij}\sigma_2(y_{\tau(i)\tau(j)})=b_{ij}a_{\tau(i)\tau(j)}'y_{\sigma_2\tau(i)\sigma_2\tau(j)}
\]
Then compare the coefficients of $y_{\tau\sigma_1(i)\tau\sigma_1(j)}=y_{\sigma_2\tau(i)\sigma_2\tau(j)}$.

\begin{lem}\label{lem: good comparison basis}
Given any good basis $(y_{kl})$ for $(\cC_m,\sigma_2)$ there is a unique good basis $(x_{ij})$ for $(\cC_n,\sigma_1)$ so that all transition coefficients of $\tau:(\cC_n,\sigma_1)\to (\cC_m,\sigma_2)$ are $b_{ij}=1$ or, equivalently, $\tau(x_{ij})=y_{\tau(i)\tau(j)}$ for all $i,j\in\cC_n$.
\end{lem}

\begin{proof}
For every $i,j\in \cC_n$, the functor $\tau$ gives an isomorphism $\cC_n(j,i)\cong \cC_m(\tau(j),\tau(i))$. Let $x_{ij}:j\to i$ be the morphism which maps to $y_{\tau(i)\tau(j)}:\tau(j)\to \tau(i)$. For each $k\in\cC_n$, $x_{ij}x_{jk}=x_{ik}$ since both morphisms map to the same morphism $y_{\tau(i)\tau(j)}y_{\tau(j)\tau(k)}=y_{\tau(i)\tau(k)}$ in $\cC_m$. So, $(x_{ij})$ is a multiplicative basis for $\cC_n$. To see that it is a good basis, suppose that $i,j$ lie in the same orbit of $\sigma_1$. Then $\tau(i),\tau(j)$ lie in the same orbit of $\sigma_2$. So, $\sigma_2(y_{\tau(i)\tau(j)})=y_{\sigma_2\tau(i)\sigma_2\tau(j)}$. This implies that $\sigma_1(x_{ij})=x_{\sigma_1(i)\sigma_1(j)}$ since both map to the same basic morphism $\sigma_2\tau(j)\to \sigma_2\tau(i)$.
\end{proof}

Suppose that $(x_{ij})$ and $(y_{kl})$ are as given in Lemma \ref{lem: good comparison basis} above. Suppose $i,j\in\cC_n$ map to $\tau(i),\tau(j)$ in the same orbit of $\sigma_2$. Then, when $\tau$ is applied to the equation $\sigma_1(x_{ij})=a_{ij}x_{\sigma_1(i)\sigma_1(j)}$, we get:
\[
	\tau\sigma_1(x_{ij})=\tau(a_{ij}x_{\sigma_1(i)\sigma_1(j)})=a_{ij}y_{\tau\sigma_1(i)\tau\sigma_1(j)}=\sigma_2\tau(x_{ij})=y_{\sigma_2\tau(i)\sigma_2\tau(j)}=\sigma_2(y_{\tau(i)\tau(j)})
\]
which implies that $a_{ij}=1$. This proves the second part of the following lemma.

\begin{lem}\label{second lemma about s,t}
If $i,j$ are in different orbits, say $A,B$, of $\sigma_1$ and $\tau(i),\tau(j)$ are in the same orbit of $\sigma_2$ then $a_{ij}^k=1$ for any $k$ which is a common multiple of both $|A|$ and $|B|$. Furthermore, there is a good basis for $\cC_n$ so that $a_{ij}=1$.
\end{lem}

\begin{proof}
The good basis for $\cC_n$ given in the previous lemma has the property that $a_{ij}=1$ whenever $\tau(i),\tau(j)$ lie in the same $\sigma_2$ orbit in $\cC_m$. But the transition factors $a_{AB}$ are well defined up to roots of unity. So, $a_{ij}$ will be a product of an $|A|$-th root of unity and a $|B|$-th root of unity when $\tau(i),\tau(j)$ lie in the same $\sigma_2$ orbit. Then $a_{ij}^k=1$ as claimed.
\end{proof}

\subsection{Continuous automorphisms of $(\cC_n,\sigma)$}

We describe the group of all automorphisms $\tau$ of $(\cC_n,\sigma)$. First notice that in the discussion above, we now set $n=m$. Hence, $\tau$ acts on objects as a permutation of $n$ vertices; $\tau \in S_n$. The condition $\sigma_2\tau = \tau\sigma_1$ becomes $\sigma\tau = \tau\sigma$. So given $\sigma \in S_n$,  the \emph{centralizer} of $\sigma$ gives us the group of set maps that may underlie an automorphism of $(\cC_n, \sigma)$: 
\[
	C(\sigma) = \{\tau\in S_n | \tau\sigma = \sigma\tau\}.
\]
This group is know as the group of \emph{generalized permutation matrices} or \emph{monomial matrices}. Supposing $\sigma$ has cycle decomposition $\sigma = A_{1,1}\cdots A_{1, e_1}\cdots A_{r,1} \cdots A_{r, e_r}$, where $A_{i,1}, \ldots, A_{i,e_i}$ are cycles of length $\lambda_i$ for $1 \leq i \leq r$, $C(\sigma)$ has the structure of a product of semidirect products:
\[
	C(\sigma) = \prod\limits_{i=1}^r (C_{\lambda_i} \rtimes S_{e_i}),
\]
where $C_n$ is the cyclic group on $n$ elements. Then $\tau$ sends each orbit of $\sigma$ to another orbit of $\sigma$ with the same size. So, $\tau$ induces a permutation of the set of orbits of $\sigma$. Consider one cycle of this action. We define a \emph{$\sigma\tau$-orbit} to be the set of objects of $\cC_n$ which lie in such a cycle. In other words, a $\sigma\tau$-orbit is a collection of $\sigma$ orbits of the same size which are cyclically permuted under the action of $\tau$. This is a minimal subset of the set of objects of $\cC_n$ which is closed under $\sigma$ and $\tau$.

\begin{lem}\label{lem: bij are A roots of 1}
Let $\sigma$ be an automorphism of $\cC_n$ and let $\tau$ be an autoequivalence of $\cC_n$ which commutes with $\sigma$. Choose a good basis for $\sigma$ and let $A\subset[n]$ be one orbit of $\sigma$. Then the $\tau$-transition coefficient $b_{\sigma(i)i}$ is independent of $i\in A$ and $b_{\sigma(i)i}^{|A|}=1$.
\end{lem}

\begin{proof}Let $i,j$ be in orbits $A,B$ of $\sigma$. Then the equation $\sigma\tau(x_ji) = \tau\sigma(x_{ji})$ becomes:
\begin{equation}\label{tau}
	a_{\tau(B)\tau(A)}b_{ji} = b_{\sigma(j)\sigma(i)}a_{BA}
\end{equation} 
where $a_{BA}$ are transition factors of $\sigma$. When $j=\sigma(i)$, $A=B$, so $a_{BA}=a_{\tau(B)\tau(A)}=1$ and
\[
	b_{\sigma(i)i}=b_{\sigma^2(i)\sigma(i)}.
\]
Thus, $\lambda=b_{\sigma(i)i}$ is independent of $i\in A$. Let $m=|A|$. Then the product of all $b_{\sigma(i)i}$ is
\[
	\lambda^m= b_{i\sigma^{m-1}(i)}\cdots b_{\sigma^2(i)\sigma(i)}b_{\sigma(i)i}=b_{ii}=1
\]
proving the lemma.
\end{proof}

\begin{lem}\label{lem: transitions coef aij are n-th roots of one}
Suppose $\sigma,\tau$ are commuting automorphisms of $\cC_n$ and $\cC_n$ is a single $\sigma\tau$-orbit. Then, for any good multiplicative basis $(x_{ij})$ for $\sigma$, the transition coefficients $a_{ij}$ of $\sigma$ will all be $n$-th roots of unity.
\end{lem}

\begin{proof}
Let $A_1,\cdots,A_s$ be the orbits of $\sigma$. They have the same size, say $|A_i|=m$ with $n=ms$, and $\tau$ cyclically permutes these orbits: $\tau(A_i)=A_{i+1}$ where the indices are taken modulo $s$. Let $i\in A_p$, $j\in A_{p+1}$. Then the equation $\sigma^m\tau(x_{ji})=\tau\sigma^m(x_{ji})$ gives:
\[
	a_{A_{p+2}A_{p+1}}^m = a_{A_{p+1}A_p}^m
\]
where indices of $A$ are taken modulo $s$. However,
\[
	a_{A_1A_s}\cdots a_{A_3A_{2}} a_{A_2A_1}=a_{A_{1}A_{1}}=1.
\]
So, $a_{A_{p+1}A_p}^m$ is an $s$-th root of unity making $a_{A_{p+1}A_p}$ an $n$-th root of unity. So, all transition coefficients of $\sigma$ are $n$-th roots of unity.
\end{proof}

\begin{prop}\label{prop: transition coef aij, bij are n-th roots of 1}
Suppose $\sigma,\tau$ are commuting automorphisms of $\cC_n$ and $\cC_n$ is a single $\sigma\tau$-orbit. Then there exists a good multiplicative basis $(x_{ij})$ for $\sigma$ so that the transition coefficients $a_{ij},b_{ij}$ of $\sigma,\tau$ are all $n$-th roots of unity.
\end{prop}

\begin{proof} As in the proof of Lemma \ref{lem: transitions coef aij are n-th roots of one}, let $A_1,\cdots,A_s$ be the orbits of $\sigma$ with $\tau(A_p)=A_{p+1}$ and $|A_p|=m$. Renumber the objects so that $A_1=[m]$ and $\sigma(i)=i+1$ module $m$ on $A_1$. Choose a good multiplicative basis $(x_{ij})$ for $\sigma$ on each $A_k$. By Lemma \ref{lem: bij are A roots of 1}, $\lambda_k=b_{\sigma(i)i}$ for $i\in A_k$ is an $m$-th root of unity which depends only on $k$ (and on the multiplicative bases for $A_k$ and $A_{k+1}$ so that $\tau(x_{\sigma(i)i})=\lambda_k x_{\tau\sigma(i)\tau(i)}$ for all $i\in A_k$).

Since $\tau^s(A_1)=A_1$, $\tau^s(1)=k+1\in[m]$ for some $k<m$. Then $b_{k+1,1}=b_{k+1,k}\cdots b_{32}b_{21}=\lambda_1^k$ which is also an $m$-th root of unity. Let $\mu$ be any $s$-th root of $\lambda_1^k$. So, $\mu^n=\lambda_1^{km}=1$.

\underline{Claim}: There a good basis for $\sigma$ so that $b_{\tau(1)1}=b_{\tau^2(1)\tau(1)}=\cdots= b_{\tau^{s}(1)\tau^{s-1}(1)}=\mu$.

Proof: Let $x_{\tau(1)1}$ be any generator of $\cC_n(1,\tau(1))$. For $1\le p\le s-2$ let $x_{\tau^{p+1}(1)\tau^p(1)}=\mu^{-p}\tau^p(x_{\tau(1)1})$. This will satisfy the claim $b_{\tau^{p+1}(1)\tau^p(1)}=\mu$ for $0\le i\le s-2$. The last transition coefficient in the claim is then determined:
\[
	b_{\tau^s(1)\tau^{s-1}(1)}=b_{k+1,\tau^{s-1}(1)}=b_{k+1,1}b_{1\tau^{s-1}(1)}=\lambda^k_1\mu^{1-s}=\mu.
\] 
This proves the Claim.


The basis elements $x_{ij}$ for $i,j$ in the same $\sigma$ orbits and $x_{\tau^{i+1}(1)\tau^i(1)}$ given by the Claim extend uniquely to a good basis for all of $\cC_n$. For example, $x_{qp}=x_{q\tau^j(1)}x_{\tau^j(1)\tau^i(1)}x_{\tau^i(1)p}$ if $p\in A_{i+1}$ and $q\in A_{j+1}$. Since $\tau^k(1)\in A_{k+1}$ we get,
\[
	b_{qp}=\lambda_{j+1}^t\mu^{j-i}\lambda_{i+1}^r
\]
for some $t,r$. This is a product of $n$-th roots of unity and thus an $n$-th root of unity. The transition coefficients $a_{ij}$ are $n$-th roots of unity by Lemma \ref{lem: transitions coef aij are n-th roots of one}. 
\end{proof}

\begin{thm}\label{thm: all aij,bij are n! roots of 1}
Let $\sigma$ be an automorphism of $\cC_n$ and $\tau$ an autoequivalence of $\cC_n$ which commutes with $\sigma$ so that $\cC_n$ is not the disjoint union of two $\sigma\tau$ invariant subcategories. Then there exists a good multiplicative basis $(x_{ij})$ for $\sigma$ so that the transitions coefficients of both $\sigma$ and $\tau$ are $n!$-th roots of unity.
\end{thm}

We call $\cC_n$ \emph{indecomposable} if it is not the disjoint union of two $\sigma\tau$ invariant subcategories.

\begin{proof}
As in Fitting's Lemma, $\tau$ is an automorphism on the image of $\tau^k$ for sufficiently large $k$. Let $\cC_m$ be this image. By Proposition \ref{prop: transition coef aij, bij are n-th roots of 1} there is a good multiplicative basis for $\sigma$ on $\cC_m$ so that all $a_{ij},b_{ij}$ are $m$-th roots of unity when $i,j\in \cC_m$. 

For each orbit $A$ of $\sigma$ disjoint from $\cC_m$ there is an orbit $B$ of $\sigma$ in $\cC_m$ so that $\tau^k(A)=\tau^k(B)$. Then $|A|$ is a multiple of $|B|$. So, by Lemma \ref{second lemma about s,t}, for any good multiplicative basis for $\cC_n$ with respect to $\sigma$, we have $a_{AB}^{|A|}=1$. This implies that every transition coefficient $a_{ij}$ of $\sigma$ will be a product of three roots of unity of order $\le n$. In particular $a_{ij}^{n!}=1$.

Next we show that a good basis for $\sigma$ can be chosen so that $b_{ij}^{n!}=1$ for all $i,j$. Let $k\ge0$ be minimal so that $\tau^k(\cC_n)=\cC_m$. If $k=0$ then the good basis can be chosen so that $b_{ij}^n=1$ for all $i,j$. So, suppose $k\ge1$. Then, by induction on $k$, there is a good basis for $\tau(\cC_n)$ so that $b_{ij}^{n!}=1$ for all $i,j\in\tau(\cC_n)$. Let $A_1,\cdots,A_s$ be the orbits of $\sigma$ not in the image of $\tau$. By Lemma \ref{lem: bij are A roots of 1} we have $b_{ij}^{|A_p|}=1$ for any $i,j\in A_p$. In each $A_p$ choose one element $z_p$ and choose the basis element $x_{\tau(z_p)z_p}\in\cC_n(z_p,\tau(z_p))$ to be $\tau^{-1}(x_{\tau^2(z_p)\tau(z_p)})$, i.e., so that $b_{\tau(z_p)z_p}=1$. Then all other basis elements $x_{ij}$ will be determined and every $b_{ij}$ will be a product of transition coefficients already given. Since any product of $n!$-th roots of unity is an $n!$-th roots of unity, $b_{ij}^{n!}=1$ for all $i,j$.
\end{proof}

Since $\sigma,\tau$ are determined by their transition coefficients and underlying set maps, Theorem \ref{thm: all aij,bij are n! roots of 1} implies that, in the indecomposable case, there are only finitely many possibilities for $\sigma,\tau$ up to isomorphism. This extends to the case of our classification theorem.

\begin{cor}\label{cor: only finitely many indecomposable n-fold coverings}
For a fixed $n$, there are, up to strong isomorphism, only finitely many add-triangulated $n$-fold equivalence coverings of $\cM_0$ which are ``indecomposable'' in the sense that they are not disjoint unions of two add-triangulated subcoverings.
\end{cor}

\begin{proof} By Corollary \ref{cor: classification up to strong isomorphism}, equivalence coverings of $\cM_0$ are given by $\sigma,\tau$ which must be anti-compatible. For the equivalence covering to be indecomposable, $\sigma,\tau$ must be as in Theorem \ref{thm: all aij,bij are n! roots of 1}. So, there are only finitely many choices for $\sigma,\tau$. Anti-compatibility of $\sigma,\tau$ further reduces the set of possibilities. So, the list remains finite.
\end{proof}

\subsection{Skew-continuity of $\varphi:\sigma\to \tau$} \label{phi}

Recall that, if $\sigma,\tau$ are commuting autoequivalences of $\cC_n$, the $\sigma$-{continuity factor} $s=\{\sigma,\tau\}$ is defined to be the scalar $s\in K^\ast$ so that
\[
	\sigma(\varphi_i)=s\varphi_{\sigma(i)}
\]
for all $i\in\cC_n$. (See Section \ref{sec: def of skew-continuity}.) We give another characterization of this scalar.

\begin{prop}\label{prop: characterization of continuity factor}
The continuity factor $s=\{\sigma,\tau\}$ is the unique nonzero scalar so that, for every $i$ and nonzero $f:i\to \tau(i),g:i\to\sigma(i)$, the following diagram commutes.
\[
	\xymatrixcolsep{35pt}
\xymatrix{
i\ar[r]^(.45)f\ar[d]_g & \sigma(i)\ar[d]^{\sigma(g)}\\
\tau(i)\ar[r]^(.47){s\tau(f)} & \sigma\tau(i)
	}
\]
\end{prop}


\begin{proof}
This diagram is a combination of the following two commuting diagrams:
\[
	\xymatrixcolsep{35pt}
\xymatrix{
\sigma(i)\ar[r]^(.45){\sigma(f)}\ar[dd]_{\varphi_i}&\sigma^2(i)\ar[dd]^{\varphi_{\sigma(i)}}&& \sigma(i)\ar[r]^(.45){\sigma(f)}\ar@/_2pc/[dd]_{t\varphi_i} & \sigma^2(i)\ar@/^2pc/[dd]^{t\sigma(\varphi_i)}\\
 && &
i\ar[u]^f\ar[r]^(.45)f\ar[d]_g & \sigma(i)\ar[u]^{\sigma(f)}\ar[d]_{\sigma(g)}\\
\tau(i)\ar[r]^(.47){\tau(f)} & \sigma\tau(i) &&
\tau(i)\ar@{-->}[r] & \sigma\tau(i)
	}
\]
Since $\cC_n(\sigma(i),\tau(i))=K$, $gf^{-1}:\sigma(i)\to \tau(i)$ is a scalar multiple of $\varphi_i$. Thus $t\varphi_i\circ f=g$ for some $t\in K^\ast$. Since $\sigma$ is a functor, $t\sigma(\varphi_i)\circ\sigma(f)=\sigma(g)$. So, the right hand diagram commutes. The left hand diagram commutes since $\varphi$ is a natural transformation. By definition of $s=\{\sigma,\tau\}$ we get:
\[
	t\sigma(\varphi_i)\circ\sigma(f)=ts\varphi_{\sigma(i)}\circ\sigma(f)=ts\tau(f)\circ\varphi_i
\]
In other words, the dotted arrow in the right hand diagram can be filled in with the morphism $s\tau(f)$ as claimed.
\end{proof}

\begin{cor}\label{cor: anti-symmetry of s-t pairing}
Let $\sigma,\tau$ be commuting autoequivalences of $\cC_n$. Then
\[
	\{\tau,\sigma\}=\{\sigma,\tau\}^{-1}.
\]
\end{cor}

There is one additional piece of structure that is required to define a triangulation. Assume that we fix a system of good bases for $(\cC_n, \sigma)$ and a permutation on $n$ elements $\tau$ that commutes with $\sigma$. For any object $i$ of $\cC_n$, we need a natural $K$-linear isomorphism $\varphi_i:\sigma(i)\cong\tau(i)$. We write:
\[
	\varphi_i =  c_ix_{\tau(i)\sigma(i)}:\sigma(i)\to \tau(i)
\]
for each object $i$ in $\cC_n$, where $c_i \in K^*$ and $\cdot c_i$ indicates multiplication by $c_i$. 

We review the construction describing the roles of $\sigma,\tau,\varphi$.
\begin{enumerate}
\item $n$ is the number of sheets in the equivalence covering category $\cM_\sigma\to \cM_0$.
\item $\sigma$ is the holonomy functor which we use as the ``clutching functor'' to construct the $n$-fold equivalence covering category $\cM_\sigma$.
\item The shift functor is given by $\tau$. Since $T X\cong X$ in the underlying algebraic category, we must have $\sigma(i)\cong \tau(i)$ and this is used to define the last morphism in the distinguished triangle (from $(Z,i)$ to $(X,\tau(i))$:
\[
	(X,i)\to (Y,i) \to (Z,i)\to (X,\sigma(i)) \xrightarrow{\varphi_i} (X,\tau(i))=T (X,i)
\]
\item The naturality of the isomorphism $\varphi$ is required by the axiom of triangulated category which says that any morphism from $f:X\to Y$ to $f':X'\to Y'$ can be completed to a map of distinguished triangles. This will give three commuting squares, where the commutativity of the last square is the naturality of $\varphi$.
\item The relation $\sigma\tau=\tau\sigma$ is the requirement that the shift functor $T(X,i)=(X,\tau(i))$ should be continuous, i.e., it should commute with holonomy.
\end{enumerate}

These conditions can be rephrased in terms of the {transition} coefficients $a_{ij},b_{ij},c_i$ of $\sigma,\tau,\varphi$ as follows. First, naturality of $\varphi$ is:
\begin{equation}\label{eq: phi is natural}
	\varphi_i \circ \sigma(x_{ij}) = \tau(x_{ij})\circ \varphi_j\Rightarrow c_ia_{ij} = b_{ij}c_j
\end{equation}
for all $i,j$. In a triangulated category, the shift of a distinguished triangle is a distinguished triangle with the sign of the last arrow reversed. Since $\sigma(i)\cong \tau(i)$ we can replace $\tau$ with $\sigma$ to get the distinguished triangle:
\[
	(X,\sigma(i))\to (Y,\sigma(i)) \to (Z,\sigma(i))\to (X,\sigma^2(i)) \xrightarrow{-\varphi_{\sigma(i)}} T (X,\sigma(i))=(X,\tau\sigma(i))
\]
So, we get the condition (same as \eqref{eq: skew-commutativity condition})
\begin{equation}\label{eq: phi is skew commutative}
-\varphi_{\sigma(i)} = \varphi_i \circ \sigma(x_{\sigma^{-1}\tau(i)i})\Rightarrow c_{\sigma(i)} = -c_ia_{\tau(i)\sigma(i)}
\end{equation}
Finally, we have the commutativity of $\sigma$ and $\tau$ which translates to the condition:
\begin{equation}\label{eq: s,t commute}
	\sigma\tau(x_{ij})=\tau\sigma(x_{ij})\Rightarrow a_{\tau(i)\tau(j)}b_{ij}=a_{ij}b_{\sigma(i)\sigma(j)}
\end{equation}

As an example, we give a complete classification of the the connected add-triangulation equivalence coverings of $\cM_0$.

\begin{thm}\label{thm: connected equivalence covers}
Up to strong isomorphism, there are $n$ connected add-triangulated $n$-fold equivalence coverings $\widetilde\cM_n(\sigma,\tau,\varphi)$ of $\cM_0$ if $n$ is even and none if $n$ is odd. Furthermore:
\begin{enumerate}
\item $\sigma$ is an automorphism of $\cC_{n}$ whose underlying permutation is an $n$ cycle.
\item $\tau$ is an automorphism of $\cC_{n}$ whose underlying permutation is one of the $n$ powers of that of $\sigma$. And each of these occur (assuming $n$ is even).
\item $\varphi:\sigma\to \tau$ is any skew-continuous natural isomorphism.
\end{enumerate}
\end{thm}

\begin{proof} By the Classification Theorem \ref{thmD1}, add-triangulated equivalence coverings of $\cM_0$ are given by the triples $(\sigma,\tau,\varphi)$ discussed above. To be a connected covering of the Moebius band the holonomy must be cyclic. So, $\sigma$ must be an $n$-cycle. Take a good multiplicative basis. Then all $a_{ij}=1$. So, $\sigma$ is determined by its underlying permutation.

By definition, $\tau$ is an autoequivalence of $\cC_n$ which commutes with $\sigma$. Then, the underlying set map of $\tau$ must be a power of the underlying permutation of $\sigma$. We claim that the transition coefficients $b_{ij}$ of $\tau$ are uniquely determined. In fact,
\[
	b_{\sigma^k(i)i}=(-1)^k.
\]
To see this take $f=x_{\sigma(i)i}$ and $g=x_{\tau(i)i}$ in Proposition \ref{prop: characterization of continuity factor}. Then $\sigma(g)=x_{\sigma\tau(i)\sigma(i)}$ and $\tau(f)=b_{\sigma(i)i}x_{\tau\sigma(i)\tau(i)}$. So, $s\tau(f)\circ g=sb_{\sigma(i)i}x_{\tau\sigma(i)i}=\sigma(g)\circ f=x_{\tau\sigma(i)i}$ and
\[
	b_{\sigma(i)i}=s^{-1}=\{\sigma,\tau\}^{-1}=-1.
\]
By Lemma \ref{lem: bij are A roots of 1}, $b_{\sigma(i)i}^n=1$. So, $n$ must be even. By Corollary \ref{cor: classification up to strong isomorphism}, the strong isomorphism class of $\widetilde\cM_n(\sigma,\tau,\varphi)$ is independent of $\varphi$. So, the classification is complete.
\end{proof}




\section{Continuous Frobenius categories}\label{ss: continuous Frobenius}

In Section \ref{sec4: classification thm} we classified all possible continuously triangulated finite coverings of the Mobius strip category and showed that, if they exist, they are classified by the three parameters $\sigma,\tau,\varphi:\sigma\to \tau$. In this section we complete the classification by showing that each of these structures exists. The proof is by explicit construction of the corresponding continuous Frobenius categories. This is a generalization of the construction of the continuous Frobenius category from \cite{cfc}. 

\subsection{Circle categories} We construct covering categories $\widetilde\cP_\sigma(S^1)$ for the circle category with coefficients in $K[[u]]$. Consider the power series rings $K[[t]]\subset K[[u]]$, where $t=u^2$ with the discrete topology. Let $\cC_n(K[[u]])$ be the category with object set $[n]=\{1,2,\cdots,n\}$ and hom sets $\Hom(i,j)=K[[u]]$ for all $i,j$ with composition given by mulitplication. (Thus $\cC_n(K[[u]])=\cC_n\otimes_KK[[u]]$.) $K[[u]]$-linear automorphisms of $\cC_n(K[[u]])$ are given by pairs $(\sigma,(a_{ij}(u)))$ where $\sigma\in S_n$ is a permutation of $n$ and $a_{ji}(u)\in K[[u]]$ are power series transitions coefficients, just as in the case of $\Aut(\cC_n)$. Then $\Aut(\cC_n)$ is the subgroup of $\Aut(\cC_n(K[[u]]))$ of automorphism where the transition coefficients $a_{ji}(u)\in K$ are all constant. Given any $\sigma\in\Aut(\cC_n)$ we will choose a system of constants $c_i\in K^\ast$ so that $a_{ji}=c_jc_i^{-1}$. This is equivalent to choosing a lifting of $\sigma$ to an element $PD\in GL_n(K)\subset GL_n(K[[u]])$ where $P$ is a permutation matrix and $D$ is a diagonal matrix with diagonal entries $c_i$.


Consider the circle $S^1=\RR/\pi\ZZ$ from Section \ref{sec1}. This is the meridian circle of the Moebius band. The boundary of the Moebius band is $\widetilde S^1=\RR/2\pi\ZZ$, a two fold covering of $S^1$.

\begin{defn}\label{def: standard P(S1)}\cite{cfc} Let $\cP( S^1)$ denote the topological $K[[u]]$-category with objects $P_x$ for all $x\in S^1$ with the topology given by the circle: $\Ob(\cP( S^1))= S^1$. Every hom set $\Hom(P_x,P_y)$ will be a free $K[[u]]$-module generated by a \emph{basic morphism} $f_{yx}$ which we give the \emph{weight} $\alpha(x,y)\in [0,\pi)$ the forward distance from $x$ to $y$ with composition given by
\[
	f_{zy}\circ f_{yx}=\begin{cases} f_{zx} & \text{if } \alpha(x,y)+\alpha(y,z)<\pi\\
    uf_{zx}& \text{otherwise}
    \end{cases}
\]
In particular, $f_{xx}$ is the identity on $P_x$ and $f_{xy}f_{yx}=uf_{xx}$ if $x\neq y$. 
\end{defn}

As $y$ converges to $x$ from below, $f_{yx}$ converges to $uf_{xx}$. It will be convenient to use a different, continuous notation. For every pair of real number $a\le b$ with corresponding elements of $ S^1$ given by $[a]=a+\pi\ZZ$ and $[b]=b+\pi\ZZ$, let $\tilde f_{ba}:P_{[a]}\to P_{[b]}$ be the morphism $\tilde f_{ba}=u^nf_{ca}$ where $n\ge0$ is maximal so that $c=b-n\pi \ge a$. Then we have
\begin{equation}\label{eq: standard shift}
	\tilde f_{b+\pi,a+\pi}=\tilde f_{ba}
\end{equation}
\begin{equation}\label{eq: half-shift}
	\tilde f_{b+\pi,a}=u \tilde f_{ba}.
\end{equation}
The main point is that composition is given by the continuous formula:
\[
	(r\tilde f_{cb})(s\tilde f_{ba})=rs\tilde f_{ca}
\]
for all $a\le b\le c\in\RR$ and all $r,s\in R=K[[u]]$.

The topological space of basic morphisms is homeomorphic to $ S^1\times [0,\infty)$ where the homeomorphism sends $\tilde f_{ba}$ to $([a],b-a)$. The topological space of all morphisms is given as a quotient of $\RR\times [0,\infty)\times K[[u]]$ modulo the relations \eqref{eq: standard shift}, \eqref{eq: half-shift}.


\begin{defn}\label{def: n-fold covering of P(S1)}Given any automorphism $\sigma$ of $\cC_n$ and a choice of lifting to $GL_n(K)$ given by scalars $c_i$, we can construct an $n$-fold covering category $\widetilde\cP_\sigma(S^1)$ with holonomy $\sigma$ as follows. The object space of $\widetilde\cP_\sigma(S^1)$ is defined to be $\widetilde S^1_\sigma$, the $n$-fold covering space of $ S^1$ with holonomy $\sigma$:
\[
	\Ob(\widetilde\cP_\sigma(S^1))=\widetilde S^1_\sigma:=\RR\times [n]/\sim
\]
where $(x+\pi,i)\sim (x,\sigma(i))$. Equivalence classes will be denoted $[x,i]$. Morphisms are $K[[u]]$-linear combinations of basic morphisms:
\[
	\tilde f_{yx}\otimes x_{ji}:[x,i]\to [y,j]
\]
for all $x\le y\in\RR$ and $i,j\in[n]$. These are given to be continuously varying with respect to $x,y\in\RR$ and satisfy the following relations.
\begin{equation}\label{eq:manual shift}
	\tilde f_{y+\pi,x+\pi}\otimes x_{ji}=a_{ji}\tilde f_{yx}\otimes x_{\sigma(j)\sigma(i)}: [x,\sigma(i)]\to [y,\sigma(j)]
\end{equation}
\begin{equation}\label{eq:manual half-shift}
	\tilde f_{y+\pi,x}\otimes x_{ji}=c_j u\tilde f_{yx}\otimes x_{\sigma(j)i}:[x,i]\to [y+\pi,j]=[y,\sigma(j)]
\end{equation}
\end{defn}

Here is an example to see that composition is well defined. ($x\le y\le z$ in this example.)
\[
	(\tilde f_{z+2\pi,y+\pi}\otimes x_{kj})(\tilde f_{y+\pi,x}\otimes x_{ji})=(c_{\sigma(k)}a_{kj}u\tilde f_{zy}\otimes x_{\sigma^2(k)\sigma(j)})(c_ju\tilde f_{yx}\otimes x_{\sigma(j)i})
\]
\[
	\tilde f_{z+2\pi,x}\otimes x_{ki}= c_{\sigma(k)}u\tilde f_{z+\pi,x}\otimes x_{\sigma(k)i}=c_{\sigma(k)}c_k u^2\tilde f_{zx}\otimes x_{\sigma^2(k)i}.
\]
In the first line, we simplify each factor then compose. In the second line we compose first. The results are equal since $a_{kj}c_j=c_k$.


\begin{prop}\label{prop: P-sigma(S1)=P-tau(S1)}
$\widetilde \cP_\sigma(S^1)$ is $K[[u]]$-linearly equivalent to $\cP(S^1)$.
\end{prop}

In fact $\widetilde \cP_\sigma(S^1)$ is a $K[[u]]$-equivalence cover of $\cP(S^1)$ but we don't need to show this.

\begin{proof}
Let $F:\cP(S^1)\to \widetilde \cP_\sigma(S^1)$ and left inverse $G:\widetilde \cP_\sigma(S^1)\to \cP(S^1)$ be given as follows. Every object of $\cP(S^1)$ can be represented as $P_{[x]}$ where $0\le x<\pi$. Let $F(P_{[x]})=[x,k]\in \widetilde \cP_\sigma(S^1)$ where $k=\sigma(1)$. For every $[x],[y]\in S^1$, $\Hom(P_{[x]},P_{[y]})$ is the free $K[[u]]$-module generated by $\tilde f_{yx}$ if $0\le x\le y<\pi$ and by $\tilde f_{y+\pi,x}$ if $0\le y<x<\pi$. Let
\[
	F(\tilde f_{yx})=\tilde f_{yx}\otimes x_{kk}:[x,k]\to [y,k]
\]
\[
	F(\tilde f_{y+\pi,x})=c_1^{-1}\tilde f_{y+\pi,x}\otimes x_{1k}:[x,k]\to [y+\pi,1]=[y,k]
\]

The left inverse $G$ is given as follows. $G[x,i]=P_{[x]}$ and for any morphism of the form $u^n\tilde f_{y+\pi,x}\otimes x_{ji}:[x,i]\to [y+\pi,j]$ where $0\le x,y<\pi$, let $G(u^n\tilde f_{y+\pi,x}\otimes x_{ji})=c_ju^n\tilde f_{y+\pi,x}$. We let $n\ge-1$ when $y\ge x$. It is straightforward to show that $G\circ F$ is the identity functor on $\cP(S^1)$ and $F\circ G$ is equivalent to the identity functor by the isomorphism $\tilde f_{xx}\otimes x_{ki}:[x,i]\cong [x,k]$.
\end{proof}

Comparing definitions, we have the following easy theorem.

\begin{thm}
The quotient category of $\widetilde \cP_\sigma(S^1)$ modulo the ideal of all morphisms divisible by $u$ is continuously isomorphic to the equivalence covering $\widetilde\cS^1_\sigma$ of the circle category $\cS^1$ from Definition \ref{def: S1 category} and Theorem \ref{thmA}.\qed
\end{thm}


Now we need to consider a double covering of $\widetilde \cP_\sigma(S^1)$. For the circle $S^1$, the double cover is $\widetilde S^1=\RR/2\pi\ZZ$. For $\widetilde S^1_\sigma$, the double cover is:
\[
	{\widetilde{\widetilde {S_\sigma^1}}}:=\RR\times [n]\times \{+,-\}/\sim
\]
where $(x+\pi,i,\vare)\sim (x,\sigma(i),-\vare)$. Points in ${\widetilde{\widetilde {S_\sigma^1}}}$ are the equivalence classes $[x,i,\vare]$. We view ${\widetilde{\widetilde {S_\sigma^1}}}$ as a $2n$-fold covering of the small circle $S^1$. The sheets of the cover are labeled $(i,\vare)$. As the continuous parameter $x$ increases to $x+\pi$, the point $(x,i,+)$ moves to sheet $(\sigma(i),-)$. Then as $x$ keeps increasing to $x+2\pi$, the point comes back to the same point $[x]\in S^1$ on sheet $(\sigma^2(i),+)$. In particular, the path will only return to the same point after going around the circle an even number of times, moving a distance of $2\pi n$ where $\sigma^{2n}(i)=i$.

Note that every point in ${\widetilde{\widetilde {S_\sigma^1}}}$ is equivalent to a positive point since 
\[
[x,i,-]=[x-\pi,\sigma(i),+].\]
As a quotient of the space of positive points, we have
\[
	{\widetilde{\widetilde {S_\sigma^1}}}=\RR\times [n]/\sim
\] 
where $(x+2\pi,i)\sim (x,\sigma^2(i))$.


\begin{defn}
For any $\sigma\in \Aut(\cC_n)$ with lifting to $GL_n(K)$ given by scalars $c_i$, let ${\widetilde{\widetilde {\cP_\sigma}}}(S^1)$ be the topological $K[[t]]$ category with object space ${\widetilde{\widetilde {S_\sigma^1}}}$ and morphisms
\[
	r\tilde f_{yx}\otimes x_{ji}:[x,i,+]\to [y,j,+]
\]
for $x\le y\in \RR$ and $r\in K[[t]]$, considered as elements of 
\[
	\Mor\left({\widetilde{\widetilde {\cP_\sigma}}}(S^1)\right) :=K[[t]]\times \{(x,y)\in\RR^2\,|\, x\le y\}\times [n]^2/\sim
\]
(with the quotient topology), where the equivalence relation is given by
\[
	\tilde f_{y+2\pi,x+2\pi}\otimes x_{ji}=b_{ji}\tilde f_{yx}\otimes x_{\sigma^2(j),\sigma^2(i)}: [x,\sigma^2(i),+]\to [y,\sigma^2(j),+]
\]
\[
	\tilde f_{y+2\pi,x}\otimes x_{ji} = d_j t \tilde f_{yx}\otimes x_{\sigma^2(j)i}:[x,i,+]\to [y+2\pi, j,+]=[y,\sigma^2(j),+]
\]
where  $d_j=c_{\sigma(j)}c_j$ and $b_{ji}=a_{\sigma(j)\sigma(i)}a_{ji}=d_jd_i^{-1}$ are the scalars associated with $\sigma^2$ and its lifting $PDPD=P^2 (P^{-1}DP)D\in GL_n(K)$.
\end{defn}

When $n=1$ and $\sigma\in \Aut(\cC_1)$ is the identity element, the topological $K[[t]]$-category is continuously isomorphic  to the category $\cP(\widetilde S^1)$, where $\widetilde S^1=\RR/2\pi\ZZ$, which is used in \cite{cfc} to construct the continuous Frobenius category. Following the same proof as in Proposition \ref{prop: P-sigma(S1)=P-tau(S1)} we have the following.

\begin{prop}\label{prop: PP-sigma=PP-tau}
The $K[[t]]$-category ${\widetilde{\widetilde {\cP_\sigma}}}(S^1)$ is equivalent to the category $\cP(\widetilde S^1)$ from \cite{cfc}.
\end{prop}

\begin{cor}\label{cor: isomorphisms do not factor}
An isomorphism $X\to Y$ between indecomposable objects of ${\widetilde{\widetilde {\cP_\sigma}}}(S^1)$ cannot factor through an object $Z$ in $add\,{\widetilde{\widetilde {\cP_\sigma}}}(S^1)$ none of whose components is isomorphic to $X$.\qed
\end{cor}


\subsection{Matrix factorizations}

Following \cite{cfc}, we construct continuous Frobenius categories as categories of matrix factorizations for objects in $add\,{\widetilde{\widetilde {\cP_\sigma}}}(S^1)$.

By a \emph{matrix factorization of $t$} in an additive $K[[t]]$-category $\cP$, we mean a pair $(P,d)$ where $P\in \cP$ and $d:P\to P$ is an $K[[t]]$-linear endomorphism whose square is $t$ times the identity on $P$. A morphism of matrix factorizations $(P,d)\to (P',d)$ is a morphism $f:P\to P'$ in $\cP$ which commutes with $d$, i.e. $fd=df$.

\begin{defn}\label{def: M(x,y,i)}
For $\cP=add\,{\widetilde{\widetilde {\cP_\sigma}}}(S^1)$, the basic example of a matrix factorization is the object $M(x,y,i)=\left([x,i,-]\oplus [y,i,+],\mat{0 & d_+\\ d_- & 0}\right)$ for real numbers $x,y$ with $|y-x|\le \pi$ and $i\in[n]$ which we visualize as:
\[
\xymatrixrowsep{10pt}\xymatrixcolsep{10pt}
\xymatrix{
M(x,y,i): & [x,i,-]\ar@/_1.5pc/[rr]^{d_-}&&  [y,i,+]\ar@/_1.5pc/[ll]^{d_+}
	}
\]
where
\[
	d_-=c_i^{-1} f_{y,x-\pi}\otimes x_{i\sigma(i)}:[x,i,-]=[x-\pi,\sigma(i),+]\to [y,i,+]
\]
\[
	d_+=c_{\sigma^{-1}(i)i}^{-1}f_{x+\pi,y}\otimes x_{\sigma^{-1}(i)i}:[y,i,+]\to [x,i,-]=[x+\pi,\sigma^{-1}(i),+]
\]
Computation shows that $d_+d_-$ and $d_-d_+$ are both $t$ times identity morphisms of $[x,i,-]$ and $[y,i,+]$, respectively. Since we are using the James construction for $add\,{\widetilde{\widetilde {\cP_\sigma}}}(S^1)$, $\oplus$ is strictly commutative. So, 
\[
[x,i,-]\oplus [y,i,+]= [x-\pi,\sigma(i),+]\oplus [y-\pi,\sigma(i),-]=[y-\pi,\sigma(i),-]\oplus [x-\pi,\sigma(i),+]
\]
and we have:
\begin{equation}\label{eq: equality for M(x,y,i)}
	M(x,y,i)=M(y-\pi,x-\pi,\sigma(i)).
\end{equation}
\end{defn}

\begin{prop}\label{prop: Krull-Schmidt}
The category of finitely generated matrix factorizations of $t$ in $add\,{\widetilde{\widetilde {\cP_\sigma}}}(S^1)$ is Krull-Schmidt and each indecomposable object is isomorphic to some $M(x,y,i)$.
\end{prop}

\begin{proof}
This follows from Proposition \ref{prop: PP-sigma=PP-tau} since the analogous statement for $add\,\cP(\widetilde S^1)$ was shown in \cite{cfc}.
\end{proof}

\begin{rem}\label{rem: injective-projectives}
Note that, in the special case when $y=x-\pi$, $d_-:[x,i,-]=[y,\sigma(i),+]\to [y,i,+]$ is an isomorphism. When $y=x+\pi$, $d_+:[y,i,+]\to [x,i,-]=[y,\sigma^{-1}(i),+]$ is an isomorphism. Furthermore, these give the same object when $i$ is shifted:
\[
	M(x,x+\pi,i)=M(x,x-\pi,\sigma(i))=M(x+2\pi,x+\pi,\sigma^{-1}(i))
\]
We will denote this object $I_{\sigma(i)}(x-\pi)=I_{\sigma^{-1}(i)}(x+\pi)$. Thus $I_i(x)=M(x+\pi,x,i)$.
\end{rem}


\begin{thm}
Let $\overline \cM_\sigma(K[[t]])$ denote the full subcategory of the matrix factorization category of $t$ in $add\,{\widetilde{\widetilde {\cP_\sigma}}}(S^1)$ with objects $M(x,y,i)$ for all $x,y\in\RR$ with $|y-x|\le \pi$ and all $i\in[n]$ modulo the relation \eqref{eq: equality for M(x,y,i)}. Then the object space of $\overline \cM_\sigma(K[[t]])$ is an $n$-fold covering space of the closed Moebius strip.
\end{thm}

\begin{proof}
This follows from the description of the object space given in the statement. The space of all $M(x,y,i)$ for $|y-x|\le \pi$ with $i\in[n]$ is $n$ copies of the fundamental domain of the closed Moebius strip and the relation \eqref{eq: equality for M(x,y,i)} identifies the two ends in such a way that the holonomy is given by the underlying permutation of $\sigma$.
\end{proof}

Let $\overline \cM(K[[t]])$ denote $\overline\cM_1(K[[t]])$ where $\sigma=1$ is the identity element of $\Aut(\cC_1)$. We call $\overline\cM(K[[t]])$ the \emph{closed Moebius strip category}.


\begin{prop}
The topological full subcategory of $\overline \cM_\sigma(K[[t]])$ with objects $M(x,x,i)$ for all $(x,i)\in\RR\times[n]$ is continuously isomorphic to $\widetilde \cP_\sigma(S^1)$.
\end{prop}

\begin{proof}
There is a continuous embedding $J:\widetilde \cP_\sigma(S^1)\to \overline \cM_\sigma(K[[t]])$ given by sending $[x,i]$ to $M(x,x,i)$ and the morphism $\tilde f_{yx}\otimes x_{ji}:[x,i]\to [y,j]$ to the horizontal arrows:
\[
\xymatrixrowsep{25pt}\xymatrixcolsep{60pt}
\xymatrix{
[x-\pi,\sigma(i),+]\ar[d]_{c_i^{-1}\tilde f_{x,x-\pi}\otimes x_{i\sigma(i)}}\ar[rr]^{a_{ji}\tilde f_{y-\pi,x-\pi}\otimes x_{\sigma(j)\sigma(i)}} &&
	[y-\pi,\sigma(j),+]\ar[d]^{c_j^{-1}\tilde f_{y,y-\pi}\otimes x_{j\sigma(j)}}\\
[x,i,+] \ar[rr]^{\tilde f_{yx}\otimes x_{ji}}&& 
	[y,j,+]
	}
\]giving a morphism $M(x,x,i)\to M(y,y,j)$.
\end{proof}

Recall \cite{cfc} that an exact sequence of matrix factorizations for an additive $K[[t]]$-category $\cP$ is defined to be a short exact sequence $0\to(A,d)\to (B,d)\to (C,d)\to0$ so that the underlying sequence in $\cP$ is split exact.


\begin{thm}
$add\,\overline \cM_\sigma(K[[t]])$ is a Frobenius category with exact structure described above with indecomposable projective-injective objects given by $I_i(x)=M(x+\pi,x,i)$.
\end{thm}

\begin{proof}
In \cite{cfc} it is shown that the category of matrix factorizations of $t$ in $\cP(\widetilde S^1)$ is a Frobenius category equivalent to the subcategory $add\,\overline\cM(K[[t]])$. By Proposition \ref{prop: PP-sigma=PP-tau} the same is true for $add\,\overline \cM_\sigma(K[[t]])$.
\end{proof}

\begin{prop}\label{prop: isomorphisms in add M do not factor}
An isomorphism between indecomposable objects in ${add\,\overline \cM_\sigma(K[[t]])}$ cannot be written as a sum of compositions of morphisms none of which are isomorphisms.
\end{prop}

\begin{proof} By \cite{cfc} this holds in the category of matrix factorizations of $t$ in $\cP(\widetilde S^1)$ which is a Frobenius category equivalent to the subcategory $add\,\overline\cM(K[[t]])$.
\end{proof}


\subsection{Continuous triangulation of the stable category}

By \cite{hap}, the stable category of the Frobenius category $add\,\overline \cM_\sigma(K[[t]])$ is a triangulated category. The distinguished triangles are given by pushing out ``universal exact sequences'' of the following form for each objects $X$:
\begin{equation}\label{eq: general universal exact sequence}
	X\to I(X)\to T(X)
\end{equation}
where $I(X)$ is a projective-injective object and $X\to I(X)$ is an admissible monomorphism, i.e. a split monomorphism on the underlying objects in $add\,{\widetilde{\widetilde {\cP_\sigma}}}(S^1)$. We would like this exact sequence, objects and morphisms, to be continous in $X$ so that the resulting triangulated category is continuously triangulated. 

Since our task in this section is the construction of the continuously triangulated categories $\widetilde \cM_n(\sigma,\tau,\varphi)$ of the Classification Theorem \ref{thm: D in intro}, we take as given the structures $\sigma,\tau,\varphi$. 

For any $K$-linear autoequivalence $\tau$ of $\cC_n$ which commutes with $\sigma$, let $F_\tau$ be the continuous autoequivalence of ${\widetilde{\widetilde {\cP_\sigma}}}(S^1)$ given on objects by $F_\tau[x,i,+]=[x,\tau(i),+]$ and on morphisms by $F_\tau=id\otimes \tau$, i.e. $F_\tau$ sends $f_{yx}\otimes x_{ji}$ to $f_{yx}\otimes b_{ji}x_{\tau(j)\tau(i)}$ where $b_{ji}$ are the transition coefficients of $\tau$ given by $\tau(x_{ji})=b_{ji}x_{\tau(j)\tau(i)}$. So far, this requires only that $\tau$ commutes with $\sigma^2$. The fact that $\tau$ commutes with $\sigma$ implies that $F_\tau$ has the same formula on the negative points: 
\[
	F_\tau[x,i,-]=F_\tau[x-\pi,\sigma(i),+]=[x-\pi,\tau\sigma(i),+]=[x,\tau(i),-]
\]
and, similarly for morphisms by the equation $a_{\tau(i)\tau(j)}b_{ij}=a_{ij}b_{\sigma(i)\sigma(j)}$ from \eqref{eq: s,t commute}:
\[
	F_\tau(f_{yx}\otimes x_{ji}:[x,i,-]\to [y,j,-])=F_\tau(a_{ji}f_{y-\pi,x-\pi}\otimes x_{\sigma(j)\sigma(i)})
\]
\[
	=a_{ji}b_{\sigma(j)\sigma(i)}f_{y-\pi,x-\pi}\otimes x_{\tau\sigma(j)\tau\sigma(i)}=b_{ji}f_{yx}\otimes x_{\tau(j)\tau(i)}:[x,\tau(i),-]\to [y,\tau(j),-].
\]
Thus, $F_\tau$ induces an autoequivalence of $\overline\cM_\sigma(K[[t]])$ given on objects by $F_\tau(M(x,y,i))=M(x,y,\tau(i))$ and on morphisms by the equation above. Since $\sigma$ commutes with itself, we also have the automorphism $F_\sigma$ of $\overline\cM_\sigma(K[[t]])$.


\begin{defn} Let $\sigma,\tau,\varphi$ be given as in Theorem \ref{thm: D in intro}. For every object $M(x,y,i)$ in $\overline\cM_\sigma(K[[t]])$ we define the \emph{universal exact sequence} of $M(x,y,i)$ to be the following
\begin{equation}\label{eq: universal exact sequence of M(x,y,i)}
	M(x,y,i)\xrightarrow{j} M(x,x+\pi,i)\oplus M(y+\pi,y,i)
	\xrightarrow{p} M(x,y,\tau(i))
\end{equation}
where $j=(j_1,j_2)$ with $j_1=id\oplus (f_{x+\pi,y}\otimes x_{ii}):[x,i,-]\oplus [y,i,+]\to [x,i,-]\oplus [x+\pi,i,+]$ and similarly, $j_2=(f_{y+\pi,x}\otimes x_{ii})\oplus id$. The morphism $p$ is the composition of two morphisms:
\[
	M(x,x+\pi,i)\oplus M(y+\pi,y,i)
	\xrightarrow{q} M(y+\pi,x+\pi,i)=M(x,y,\sigma(i))\xrightarrow{\varphi_i} M(x,y,\tau(i))
\]
where $q=(-q_1,q_2)$ with $q_1=id\oplus(f_{y+\pi,x}\otimes x_{ii})$ and $q_2=(f_{x+\pi,y}\otimes x_{ii})\oplus id$. Finally, $\varphi=\varphi_1\oplus\varphi_2$ is the isomorphism given on the summands of $M(x,y,\sigma(i))$ by $\varphi_1=f_{xx}\otimes c_ix_{\tau(i)\sigma(i)}$ and $\varphi_2=f_{yy}\otimes c_ix_{\tau(i)\sigma(i)}$ where $c_i$ are the structure constants of $\varphi$ defined by $\varphi_i=c_ix_{\tau(i)\sigma(i)}:\sigma(i)\to \tau(i)$ and satisfying:
\[
	c_ja_{ji}=b_{ji}c_i,\quad c_{\sigma(i)}=-c_ia_{\tau(i)\sigma(i)}
\] 
from \eqref{eq: phi is natural} and \eqref{eq: phi is skew commutative}=\eqref{eq: skew-commutativity condition}. 
\end{defn}


To see that \eqref{eq: universal exact sequence of M(x,y,i)} is exact, note that $j_1$ is the identity map on the first factor, $j_2$ is the identity map on the second factor and $p$ sends the remaining factors of the middle term isomorphically to the two summands of $M(x,y,\tau(i))$. Also, a key required property is that the middle terms in the sequence \eqref{eq: universal exact sequence of M(x,y,i)} is projective-injective. To emphasize this we rewrite the sequence as:
\[
	M(x,y,i)\to I_{\sigma(i)}(x-\pi)\oplus I_i(y)\to TM(x,y,i)
\]
where $TM(x,y,i)=M(x,y,\tau(i))$. We also use the notation $SM(x,y,i)=M(x,y,\sigma(i))$.

\begin{prop}\label{prop: exact sequence is well-defined}
The sequence \eqref{eq: universal exact sequence of M(x,y,i)} is well-defined, i.e. we obtain the {same sequence} if we use the notation $M(x,y,i)=M(y-\pi,x-\pi,\sigma(i)$:
\begin{equation}\label{eq: same universal sequence}
	M(y-\pi,x-\pi,\sigma(i))\xrightarrow{j'} M(y-\pi,x,\sigma(i))\oplus M(y,x-\pi,\sigma(i))\xrightarrow{p'} M(y-\pi,x-\pi,\tau\sigma(i)).
\end{equation}
In particular we have a well defined functor $I$ giving the middle term: $IM(x,y,i)=I_{\sigma(i)}(x-\pi)\oplus I_i(y)$, although the order of the components is not well-defined.
\end{prop}

\begin{proof}
To see that \eqref{eq: same universal sequence} is the same as \eqref{eq: universal exact sequence of M(x,y,i)}, we note first that the terms are the same except that the terms in the middle are switched. But that is OK since, by our definitions, direct sum is strictly commutative! Since the summands are switched, $q=(-q_1,q_2)$ changes to $q'=(-q_2,q_1)=-q$. But the sign of $\varphi$ also changes since $\varphi$ is skew-commutative. Therefore, $p'=p$. So, the sequences are equal and the universal sequence \eqref{eq: universal exact sequence of M(x,y,i)} is well-defined. 
\end{proof}

The key point is the continuity of the universal sequence.

\begin{lem}
The exact sequence \eqref{eq: universal exact sequence of M(x,y,i)} is a continuous function of $M(x,y,i)$ as it varies in the compact Hausdorff space $\Ob(\overline \cM_\sigma(K[[t]]))$ which is an $n$-fold covering of the closed Moebius strip.\qed
\end{lem}

Recall from \cite{hap} that the \emph{stable category} of a Frobenius category is defined to be the quotient category with the same objects but modding out morphisms which factor through projective-injective objects. As in \cite{ccc} we get the following.


\begin{thm}
The stable category of the Frobenius category $add\,\overline \cM_\sigma(K[[t]])$ is $add\,\widetilde\cM_\sigma$, the additive category of the equivalence covering $\widetilde \cM_\sigma$ of $\cM_0$.\qed
\end{thm}

We recall from \cite{hap} the construction of all distinguished triangles in the stable category.

\begin{defn}
Given any morphism $f:X\to Y$ in $add\,\overline \cM_\sigma(K[[t]])$, a distinguished triangle $X\xrightarrow {\overline f} Y\xrightarrow {\overline g} Z\xrightarrow {\overline h} TX$ is given as the pushout of the direct sum of universal exact sequences for the summands of $X$:
\begin{equation}\label{eq: pushout diagram giving triangle}
\xymatrix{
X\ar[d]_f\ar[r]^j &
	IX\ar[d]\ar[r]^p &
	TX\ar[d]^=\\
Y \ar[r]^g& 
	Z \ar[r]^h&
	TX
	}
\end{equation}
Here $X\to IX\to TX$ is a direct sum of the universal sequences \eqref{eq: universal exact sequence of M(x,y,i)} for each component of $X$. In the stable category, we delete all the projective-injective objects in $X,Y,Z$ and replace maps $f,g,h$ by the stable maps $\overline f,\overline g,\overline h$ which are $f,g,h$ modulo morphisms which factor through projective-injective objects.
\end{defn}

\begin{eg}\label{eg: positive triangle}
Suppose $0<x<y<z<\pi$
Take $X=M(x,y,i)$, $Y=M(x,z,i)$ and $f=id\oplus (f_{zx}\otimes x_{ii}):[x,i,-]\oplus [y,i,+]\to [x,i,-]\oplus [z,i,+]$. Then 
\[
IX= I_{\varphi(i)}(x-\pi)\oplus I_i(y)=[x,i,-]\oplus [x+\pi,i,+]\oplus [y+\pi,i,-]\oplus [y,i,+]
\]
So, the $I_{\varphi(i)}(x-\pi)$ term remains and
\[
	Z=I_{\varphi(i)}(x-\pi)\oplus M(y+\pi,z,i).
\]
Since the negative sign in the universal sequence \eqref{eq: universal exact sequence of M(x,y,i)} is on the first summand $I_{\varphi(i)}(x-\pi)$, when we go to the stable category, the scalars for the first two maps are $+1$:
\[
	M(x,y,i)\xrightarrow 1M(x,z,i)\xrightarrow1 M(y+\pi,z,i)\xrightarrow{c_i} M(x,y,\tau(i)).
\]
\end{eg}

\begin{rem}
We observe that, after applying the forgetful functor $add\,\overline \cM_\sigma(K[[t]])\to {add\,\widetilde{\widetilde {\cP_\sigma}}(S^1)}$, the diagram \eqref{eq: pushout diagram giving triangle} can be rewritten in the following form.
\[
\xymatrix{
X\ar[d]_f\ar[r]^(.4){\binom1r} &
	X\oplus TX\ar[d]^{f\oplus id}\ar[r]^(.6){(-r,1)} &
	TX\ar[d]^{id}\\
Y \ar[r]^(.4){\binom1s}& 
	Y\oplus TX \ar[r]^(.6){(-s,1)}&
	TX
	}
\]
where $r=s\circ f:X\to TX$. The operator $d$ which acts compatibly on $X,Y,TX$ gives an induced action of $d$ on $Z=Y\oplus TX$. Thus $(Z,d)$ is uniquely determined up to isomorphism by $f:X\to Y$. But we should keep in mind that ${add\,\overline \cM_\sigma(K[[t]])}$ contains $n$ isomorphic copies of each indecomposable object.
\end{rem}

As in the main theorem \ref{thm: D in intro}, the stable category of $add\,\overline \cM_\sigma(K[[t]])$ together with the distinguished triangles given above is denoted $add\,\widetilde\cM(\sigma,\tau,\varphi)$.


\begin{thm}\label{thm: add M(s,t,phi) is continuously triangulated}
$add\,\widetilde\cM(\sigma,\tau,\varphi)$ is a continuously triangulated category.
\end{thm}

\subsection{Proof of Theorem \ref{thm: add M(s,t,phi) is continuously triangulated}} Since $\tau$ commutes with $\sigma$, the shift functor $T=F_\tau$ is continuous. It remains to show that the set of distinguished triangles is a closed set. To do this we analyze the limiting behavior of morphisms and objects. There are two cases: when some objects or morphisms go to zero and when they don't. One thing is clear: in a limit the number of objects can only decrease and the number of nonzero morphisms can only decrease. We recall that our objects have fixed direct sum decompositions and thus morphisms are given by matrices whose entries are morphisms between the indecomposable objects $M(x,y,i)$.

The components $[x,i,-]=[x-\pi,\sigma(i),+]$ and $[y,i,+]$ of $M(x,y,i)$ will be called the \emph{ends} of $M(x,y,i)$. Sometime, we will call $[x,i,-]$ the \emph{negative end} and $[y,i,+]$ the \emph{positive end}. These terms are not well-defined. They depend on the notation used to describe the objects. (However, the unordered pair of ends is well-defined.)

Recall that the \emph{support} of any object $X$ is the set of all objects $Y$ for which there is a nonzero morphism $X\to Y$. We already know that the support of $M(x,y,i)$ is the set of all $M(x',y',j)$ where $x\le x'<y+\pi$ and $y\le y'<x+\pi$. We also recall that every nonzero morphism is a scalar multiple of a basic morphism. In particular, there is a well defined scalar $a\in K^\ast$ associated to every nonzero morphism and this scalar is constant on families of morphism (except when the morphism goes to zero).

\begin{defn}\label{def: 4 kinds of morphisms}
There are four kinds of morphisms in $\widetilde\cM(\sigma,\tau,\varphi)=\widetilde \cM_\sigma$:
\begin{enumerate}
\item a nonzero morphism $f:X\to Y$ is \emph{stably nonzero} if there exist open neighborhoods, $U,V$ of $X,Y$ so that $\Hom(A,B)\neq 0$ for all $A\in U,B\in V$. Equivalently, $X=M(x,y,i)$ and $Y=M(x',y',j)$ where $x<x'<y+\pi$ and $y<x'<x+\pi$. (See Figure \ref{fig2}.)  \item a nonzero morphism $f:X\to Y$ is \emph{marginal} if it is not stably nonzero. Equivalently, $X,Y$ share an end and $f$ is an isomorphism at that end. For example, any isomorphism $f:X\cong Y$ is marginal since $X,Y$ have the same two ends and $f$ is an isomorphism at each end.
\item a zero morphism $f:X\to Y$ is \emph{stably zero} if it is not a limit of nonzero morphisms.
\item a \emph{limiting null morphism} is a zero morphism which is a limit of nonzero morphisms.
\end{enumerate}
\end{defn}

\begin{prop}\label{prop 6.24} 
\begin{enumerate}
\item If $f:X\to Y$ is stably nonzero then there are contractible open neighborhoods $U,V$ of $X,Y$ so that, for all $A\in U$, $B\in V$ there is a unique $f':A\to B$ which is nonzero and homotopic to $f$ through homotopies $f_t:A_t\to B_t$ where $A_t\in U$, $B_t\in V$. Furthermore, the set of all such maps $f'$ forms an open neighborhood of $f$ in $\Mor(\widetilde\cM_\sigma)$.
\item $f:X\to Y$ is marginal if $X=M(x,y,i)$ and either $Y=M(x,y',j)$ for some $y\le y'<x+\pi$ or $Y=M(x',y,j)$ for some $x\le x'<y+\pi$ (or both).
\item Nonzero morphisms are either marginal or stably nonzero. 
\item Every nonzero morphism has a contractible neighborhood in $\Mor(\widetilde\cM_\sigma)$ consisting of nonzero maps with the same scalar.
\item Limiting null morphisms have the form $M(x,y,i)\to M(y+\pi,z,j)$ where $y<z\le x+\pi$ or $M(x,y,i)\to M(w,x+\pi,j)$ where $x<w\le y+\pi$. Any zero morphism with domain and range not of this form is stably zero.
\end{enumerate}
\end{prop}

\begin{proof}
All of this follows directly from a description of the support of any objects $M(x,y,i)$ and the fact that we are taking the discrete topology on $K$.
\end{proof}

To prove Theorem \ref{thm: add M(s,t,phi) is continuously triangulated}, we need to show that, when a family of distinguished triangles
\begin{equation}\label{eq: family of converging triangles}
	X_\alpha\xrightarrow{f_\alpha} Y_\alpha\xrightarrow{g_\alpha} Z_\alpha\xrightarrow{h_\alpha} TX_\alpha
\end{equation}
converges to a sequence
\begin{equation}\label{eq: limit triangles}
	X_\infty	\xrightarrow{f_\infty} Y_\infty\xrightarrow{g_\infty} Z_\infty\xrightarrow{h_\infty} TX_\infty
\end{equation}
the limit sequence is a distinguished triangle.

\underline{Case 1}: 
First consider the case when the terms $X_\infty,Y_\infty,Z_\infty$ in the limiting sequence \eqref{eq: limit triangles} has the same number of summands as (almost all of) the terms in \eqref{eq: family of converging triangles}. Suppose further that all morphisms in \eqref{eq: limit triangles} are stable: either nonzero or stably zero. In that case, we lift each distinguished triangle in \eqref{eq: family of converging triangles} to a push-out diagram \ref{eq: pushout diagram giving triangle}. By Proposition \ref{prop 6.24}, the scalars associated to nonzero morphisms become constant. By the following lemma, the limit is a push-out diagram of the universal exact sequence $X_\infty\to IX_\infty\to TX_\infty$ making \eqref{eq: limit triangles} into a distinguished triangle.


\begin{lem}\label{lem: limit of exact sequence is exact}
If a family of exact sequence in ${add\,\overline \cM_\sigma(K[[t]])}$, say $A_\alpha\xrightarrow{f_\alpha}B_\alpha\xrightarrow{g_\alpha}C_\alpha$, converges to a sequence $A_\infty\xrightarrow{f_\infty}B_\infty\xrightarrow{g_\infty}C_\infty$, the limiting sequence is also exact.
\end{lem}

\begin{proof}
In the topological $K[[t]]$-category ${add\,\overline \cM_\sigma(K[[t]])}$, morphisms cannot go to zero and objects cannot go to zero. Therefore, in any converging family of objects and morphisms, the number of objects becomes stable (constant) and the scalars associated to each morphism between indecomposables also becomes constant. So, zero morphisms cannot become nonzero and $g_\infty\circ f_\infty=0$.

Also, isomorphisms between indecomposables in ${add\,\overline \cM_\sigma(K[[t]])}$ do not factor through nonisomorphism by Proposition \ref{prop: isomorphisms in add M do not factor} and the same holds in ${\widetilde{\widetilde {\cP_\sigma}}(S^1)}$ by Corollary \ref{cor: isomorphisms do not factor}. Thus, the property of being exact, which is equivalent to certain component morphism in the category ${\widetilde{\widetilde {\cP_\sigma}}(S^1)}$ being isomorphisms, is preserved in the limit and the limit sequence is exact.
\end{proof}

This lemma implies that the family of pushout diagrams in ${add\,\overline \cM_\sigma(K[[t]])}$ associated to the family of distinguished triangles \eqref{eq: family of converging triangles} converges to a pushout diagram and the limiting sequence \ref{eq: limit triangles} is a distinguished triangle.

\underline{Case 2}:
Suppose that the objects $X_\alpha,Y_\alpha,Z_\alpha$ do not converge to zero objects but some of the morphisms between some components go to zero. In other words some of the morphism between components of $X_\infty,Y_\infty,Z_\infty$ are limiting null morphisms (Definition \ref{def: 4 kinds of morphisms}(4)). 

In this case, we apply the concept from Remark \ref{rem: quotient topology} which implies that the scalar associated to morphisms converging to zero can take only finitely many values. If we choose an ultrafilter for the set of parameters $\{\alpha\}$ we can assume that there is only one limiting scalar. This implies that, when we lift objects and morphisms in the stable category to ${add\,\overline \cM_\sigma(K[[t]])}$, the limit is well-defined. So, Lemma \ref{lem: limit of exact sequence is exact} applies and the push-out diagrams in ${add\,\overline \cM_\sigma(K[[t]])}$ converge to a pushout diagram and therefore the limiting sequence \eqref{eq: limit triangles} is a distinguished triangle.

\underline{Case 3}: Suppose that some components of $X_\alpha,Y_\alpha,Z_\alpha$ converge to zero so that the number of components in the limit sequence \eqref{eq: limit triangles} is strictly smaller than the number of components in the family of triangles \eqref{eq: family of converging triangles}. 

In this case, when the triangles \eqref{eq: family of converging triangles} are lifted to pushout diagrams, some of the components converge to projective-injective objects. As in Case 2, using Remark \ref{rem: quotient topology}, we may assume that the scalars associated to the morphisms in this diagram are constant and, therefore, the morphisms have well-defined limits. Again Lemma \ref{lem: limit of exact sequence is exact} applies to show that the limit diagram is a pushout diagram and the limiting sequence \eqref{eq: limit triangles} is a distinguished triangle.

This resolves all the cases and concludes the proof of Theorem \ref{thm: add M(s,t,phi) is continuously triangulated}

\subsection{Universal virtual triangles}

To conclude the proof of the main Theorem \ref{thm: D in intro}, we need to show that in the stable category $add\,\widetilde\cM(\sigma,\tau,\varphi)$ the universal virtual triangles given in \eqref{eq: universal virtual triangles} are distinguished triangles up to sign equivalence. This straightforward calculation will conclude this paper.

Let $X=M(x,y,i)$ where $|y-x|<\pi$. Let 
\[
f=(f_1,f_2):X\to I_1^\vare X\oplus I_2^\vare X=M(y+\pi-\vare,y,i)\oplus M(x,x+\pi-\vare,i)
\]
where $f_1=f_{y+\pi-\vare,x}\otimes x_{ii}\oplus id$ and $f_2=id\oplus f_{x+\pi-\vare,y}\otimes x_{ii}$. We denote these as $f=(1,1)$. Since the morphism $X\to I_{\sigma(i)}(x-\pi)\oplus I_i(y)$ factors through this map, the pushout is $M(y+\pi-\vare,x+\pi-\vare,i)\oplus I_{\sigma(i)}(x-\pi)\oplus I_i(y)$ and the pushout diagram is the following.
\[
\xymatrixrowsep{50pt}\xymatrixcolsep{60pt}
\xymatrix{
M(x,y,i)\ar[d]_{\binom11}\ar[r]^(.4){\binom11} &
	I_{\sigma(i)}(x-\pi)\oplus I_i(y)\ar[d]^{\mat{0 & 0\\ 1 & 0\\0 & 1}}\ar[r]^(.55){(-c_i,c_i)} &
	M(x,y,\tau(i))\ar[d]^=\\
I_1^\vare X\oplus I_2^\vare X \ar[r]_(.4){\mat{-1 & 1\\ 1 & 0\\0 & 1}}& 
	Y\oplus I_{\sigma(i)}(x-\pi)\oplus I_i(y) \ar[r]^(.6){(-c_i,-c_i,c_i)}&
	M(x,y,\tau(i))
	}
\]
Thus, we obtain the distinguished triangle
\[
	X\xrightarrow{\binom11} I_1^\vare X\oplus I_2^\vare X \xrightarrow{(-1,1)} Y \xrightarrow{-c_i} TX
\]
which is sign equivalent to the desired system of distinguished triangles.



\section*{Acknowledgements}

Research for this project was funded by the National Security Agency NSA Grant \#H98230-13-1-0247. The first author also acknowledges support from a GAANN grant (Graduate Assistance in Areas of National Need) during this research period. Both authors are very grateful for the support and encouragement of Gordana Todorov. These results were presented by the first author at the Auslander Conference at Woods Hole in May, 2015 which was supported by the National Science Foundation. This revised version was presented by the second author at a workshop at Tsinghua University in Beijing in July 2017. Both of these conferences were very fruitful and enjoyable events and the authors would like to thank the organizes of both events. Finally, the authors are grateful to the referee for many very helpful suggestions.


\begin{thebibliography}{1}

\bibitem{A} Claire Amiot, \emph{Cluster categories for algebras of global dimension 2 and quivers with potential}, Annales de l'institut Fourier. Vol. 59. No. 6. 2009.


\bibitem{BIRS} Aslak~Bakke Buan, Osamu Iyama, Idun Reiten, and Jeanne Scott, \emph{Cluster structures for 2-Calabi-Yau categories and unipotent groups}, Compositio Mathematica 145, no. 4 (2009): 1035--1079.


\bibitem{BMRRT} Aslak~Bakke Buan, Robert~J. Marsh, Markus Reineke, Idun Reiten, and Gordana Todorov, \emph{Tilting theory and cluster combinatorics}, Adv. Math. \textbf{204} (2006), no.~2, 572--618.


\bibitem{FZ}
Sergey Fomin, and Andrei Zelevinsky, \emph{Cluster algebras I: foundations}, Journal of the American Mathematical Society 15.2 (2002): 497--529.

\bibitem{hap} Dieter Happel, \emph{Triangulated Categories in the Representation Theory of Finite Dimensional Algebras}, London Mathematical Society Lecture Note Series 119, Cambridge University Press (1988).

\bibitem{Ha} Allen Hatcher, \emph{Algebraic Topology}, Cambridge Univ. Press, Cambridge, 2002.


\bibitem{cfc} Kiyoshi Igusa and Gordana Todorov, \emph{Continuous {F}robenius categories}, Proceedings of the Abel Symposium 2011 (2013), 115--143.

\bibitem{ccc} Kiyoshi Igusa and Gordana Todorov, \emph{Continuous cluster categories I}, Algebras and Representation Theory, vol 18, no 1 (2015), 65--101.

\bibitem{ccc2} Kiyoshi Igusa and Gordana Todorov, \emph{Continuous cluster categories II: Continuous cluster-tilted categories}, arXiv preprint arXiv:1909.05340 (2019).

\bibitem{cccD} Kiyoshi Igusa and Gordana Todorov, \emph{Continuous cluster categories of type D}, arXiv preprint arXiv:1309.7409 (2013).

\bibitem{posets}
Kiyoshi Igusa and Gordana Todorov, \emph{Cluster categories coming from cyclic posets}, Communications in Algebra, vol 43 (2015), 4367--4402. 

\bibitem{IRT} Kiyoshi Igusa, Job Rock, and Gordana Todorov, \emph{Continuous quivers of type A (I) The generalized barcode theorem}, arXiv preprint arXiv:1909.10499 (2019).

\bibitem{Orlov} Dmitri Orlov, \emph{Triangulated categories of singularities and D-branes in Landau-Ginzburg models.} (Russian) Tr. Mat. Inst. Steklova 246 (2004), 3, 240--262.; transl. Proc. Steklov Inst. Math. 246 (2004), 3, 227--248.

\bibitem{P} Pierre-Guy Plamondon, \emph{Cluster algebras via cluster categories with infinite-dimensional morphism spaces}, Compositio Mathematica 147.6 (2011): 1921--1954.


\bibitem{R}Idun Reiten, \emph{Cluster categories}, Proceedings of the International Congress of Mathematicians 2010 (ICM 2010) (In 4 Volumes) Vol. I: Plenary Lectures and Ceremonies Vols. II--IV: Invited Lectures. 2010.

\end{thebibliography}
\end{document}